\documentclass{amsart}
\usepackage{amssymb, mathrsfs}
\usepackage[bookmarksnumbered, colorlinks, linkcolor=blue, anchorcolor=blue, citecolor=blue, plainpages]{hyperref}

\newtheorem{theorem}{Theorem}[section]
\newtheorem{corollary}[theorem]{Corollary}
\newtheorem{lemma}[theorem]{Lemma}
\newtheorem{proposition}[theorem]{Proposition}
\theoremstyle{definition}
\newtheorem{definition}[theorem]{Definition}
\theoremstyle{remark}
\newtheorem{remark}[theorem]{Remark}

\numberwithin{equation}{section}

\newcommand{\CN}{\mathbb{C}^n}
\newcommand{\BN}{\mathbb{B}_n}
\newcommand{\SN}{\mathbb{S}_n}
\newcommand{\g}{{\rm g}}

 \DeclareMathOperator*{\re}{Re}
\DeclareMathOperator*{\im}{Im}

\begin{document}
	
	\title[Weighted norm inequalities]
	{Weighted norm inequalities of various square functions and Volterra integral operators on the unit ball}

	\author[C. Pang]{Changbao Pang}
	\address{%
		School of Mathematics and Computer Science\\
		Shanxi Normal University\\
		Taiyuan, 030031\\
		China}
	\email{cbpangmath@sxnu.edu.cn}

	\author[M. Wang]{Maofa Wang}
	\address{%
		School of Mathematics and Statistics\\
		Wuhan University\\
		Wuhan, 430072\\
		China}
	\email{mfwang.math@whu.edu.cn}
	
	\author[B. Xu]{Bang Xu}
	\address{%
		Department of Mathematics\\
		University of Houston\\
		Houston, TX 77204-3008\\
		USA}
	\email{bangxu@whu.edu.cn}
	
	\author[H. Zhang]{Hao Zhang}
	\address{%
		Department of Mathematics\\
		University of Illinois Urbana-Champaign\\
		USA}
	\email{hzhang06@illinois.edu}

	\subjclass[2020]{32A35; 32A37; 32A50; 47B38.}
	
	\keywords{Hardy space; BMOA; Muckenhoupt weight; area integral; Volterra integral operator.}
	
	\date{December 2, 2024}

	
	\thanks{C. Pang was supported by National Natural Science Foundation of China (Grant No. 12301153), Fundamental Research Program of Shanxi Province (Grant No. 202303021222191), and Scientific and Technological Innovation Programs of Hight Education Institutions in Shanxi (Grant No. 2022L265). M. Wang was supported by National Natural Science Foundation of China (Grant No. 12171373). B. Xu was supported by National Natural Science Foundation of China (Grant Nos. 12071355, 12325105).}
	
	
	\begin{abstract}
		In this paper, we investigate various square functions on the complex unit ball. We prove the weighted inequalities of the Lusin area integral associated with  Poisson integral in terms of $A_p$ weights for all $1<p<\infty$; this gives an affirmative answer to an open question raised by Segovia and Wheeden. In addition, we get an equivalent characterization of weighted Hardy spaces by means of the Lusin area integral in the context of holomorphic functions. We also obtain the weighted inequalities for Volterra integral operators. 
	\end{abstract}

	\maketitle

	\section{Introduction}
	This paper is devoted to establishing the weighted norm inequalities of various square functions on the unit ball of several complex variables. Area integral, also called square function, has been an indispensable tool in complex analysis and harmonic analysis with deep relation in many important topics, such as singular integral theory, Volterra integral operators, Hardy spaces and functional analysis \cite{AB,Ca,CMS,X}. The study of area integral could date back to Littlewood and Paley \cite{LP1,LP2}. Later, Marcinkiewicz and Zygmund \cite{MZ} proved the boundedness of Lusin area integral in the unit disc.  Since then, area integral has attracted a lot of attention and is hugely studied in the literature. One of the most interesting aspects is the relationship between Hardy spaces and area integral. In \cite{FS}, Fefferman and Stein developed an equivalent characterization of Hardy spaces on Euclidean spaces in terms of area integral. Via the theory of distribution functions and area integral developed by Burkholder and Gundy \cite{BG}, Geller \cite{Ge} characterized the Hardy spaces in terms of area integral  on the Siegel upper half space of type II. The area integral characterization of Hardy spaces for holomorphic functions or M-harmonic functions in $\BN$ was studied in \cite{AB,ABC,BBG,C}; and we refer to \cite{KL,NSW,P,PW,S2,Zy} and references therein for more information on this subject matter.
	
	Another significant but challenging subject is to study the weighted norm inequalities for area integral. The weighted norm inequalities were first studied by Muckenhoupt \cite{MB}, where he established the weighted norm inequalities for Hardy-Littlewood maximal function. Concerning area integral, Gundy and Wheeden \cite{GW} developed weighted norm inequalities between the Lusin area integral and non-tangential maximal function on  Euclidean spaces.  Besides, Lerner \cite{Ler,Ler3} proved 
	$$\|Tf\|_{L^p_\omega}\leq C[\omega]_{A_p}^{\max\{1/2,1/(p-1)\}}\|f\|_{L^p_\omega}, \ \ \forall \ 1<p<\infty,$$
	where $T$ is one of the following area integrals: the intrinsic square function, the Lusin area integral and the Littlewood-Paley $\g$-function; moreover, the exponent in the $A_p$ weight norm is optimal. This already gives a complete picture of weighted inequalities of several area integrals on Euclidean spaces. With regard to spaces of homogeneous type, Bui and Duong \cite{BD} obtained sharp weighted inequalities of the area integral associated with a semigroup whose kernels satisfy the Gaussian upper bound.
	It is worthwhile to note that the peak of the weighted inequality theory is the complete resolution of the famous $A_2$ conjecture and matrix $A_2$ conjecture, respectively by Hyt\"{o}nen  \cite{HT} and Domelevo, Petermichl, Treil and Volberg \cite{DPTV}.  
	
	To the best of our knowledge, the first weighted inequality of the area integral in the complex setting is due to Segovia and Wheeden \cite{SW}, where they showed that the Lusin area integral on the unit circle is bounded on weighted $L^p$ spaces for any $1<p<\infty$, and they left the weighted norm inequalities of the area integral on the higher dimensional unit balls $\BN$ as an open question. In this paper, we are devoted to the study of the above open question. However, note that there are several different types of area integral on the higher dimensional unit balls. As a result, we will investigate the weighted norm inequalities of several types of area integrals on $\BN$ respectively. 
	
   The weighted inequalities of the area integral associated with the Bergman gradient were first established by Petermichl and Wick in \cite{PW}. This is the first attempt to handle weighted inequalities of area integral on the higher dimensional unit balls. To better illustrate Petermichl and Wick's remarkable results \cite{PW}, we first present notation on $A_p$ weights and weighted $L^p$ spaces. Denote by $\SN$ the unit sphere $\partial\BN$. Let $dv$ and $d\sigma$ be the normalized volume measure on $\BN$, and the normalized surface measure on $\SN$ respectively. Recall that a positive integrable function $\omega$ on $\SN$ is called an $ A_p$ weight with $1<p<\infty$ if and only if
	$$[\omega]_{A_p}:=\sup_{B}\left(\frac{1}{\sigma(B)}\int_B\omega d\sigma\right)
	\left(\frac{1}{\sigma(B)}\int_B\omega^{-1/(p-1)}d\sigma\right)^{p-1}<\infty,$$
	where the supremum is taken over all nonisotropic metric balls $B\subset\SN$ (see Subsection \ref{NMB} for the definition of nonisotropic metric balls). By a slight abuse of notation, let $A_p$ be the set of all $A_p$ weights. If $\omega$ is a positive measurable function on $\SN$, denote by $L^p_\omega$ the space of all complex-valued measurable functions on $\SN$ that are $p$-th power integrable with respect to the measure $\omega d\sigma$. 
	
	 Then we need to introduce the area integral associated with the Bergman gradient. Given a positive integer $k$, the space $\mathcal{C}^k(\BN)$ is defined to be the set of all complex-valued functions on $\BN$ with continuous derivatives of order $\gamma$ for all multi-indices $\gamma\in \mathbb{N}^n$ with $|\gamma|\leq k$. Let $\widetilde{\bigtriangledown}$ be the Bergman gradient (see the definition in Section \ref{bergra}). The Kor\'anyi (admissible, non-tangential) approach region for $\zeta\in\SN$ and $\alpha>1/2$ is defined by
	$$D_\alpha(\zeta)=\{z\in\BN:|1-\langle z,\zeta\rangle|<\alpha(1-|z|^2)\}.$$
	The Lusin area integral associated with the Bergman gradient is defined as $$S_\alpha^{\widetilde{\bigtriangledown}}(u)(\zeta)=\left(\int_{D_\alpha(\zeta)}|\widetilde{\bigtriangledown} u(z)|^2\frac{dv(z)}{(1-|z|^2)^{n+1}}\right)^{1/2}, \quad \forall  \zeta\in\SN, $$
	where $u\in\mathcal{C}^1(\BN)$ and $\alpha>1/2$.
	
	In addition, we also need to present Poisson integral. For any $f\in L^1$, the Poisson integral, denoted by $P[f]$, is given by
	$$P[f](z)=\int_{\SN}P(z,\zeta)f(\zeta)d\sigma(\zeta), \ \ \forall  z\in\BN,$$
	where $P(z,\zeta)$ is the invariant Poisson-Szeg\"o kernel defined on $\BN\times\SN$:
	$$P(z,\zeta)=\frac{(1-|z|^2)^n}{|1-\langle z,\zeta\rangle|^{2n}}.$$
	Note that the Poisson integral uniquely solves the Dirichlet problem for the Laplace-Beltrami operator determined by the Bergman metric in $\BN$ \cite{Sto}.

	Petermichl and Wick in \cite{PW} proved the following theorem:
	
	\
	
\noindent \textbf{Theorem A.} Let $\alpha>1/2$ and $\omega\in A_2$. For any $f\in L_\omega^2$, 
	\begin{align}\label{PWres}
		\|S_\alpha^{\widetilde{\bigtriangledown}}(P[f])\|_{L^2_\omega}\leq C\widetilde{Q}_2(\omega)\|f\|_{L^2_\omega},
	\end{align}
	where $C>0$ depends only on $\alpha$ and $n$, and
	$$\widetilde{Q}_2(\omega)=\sup_{z\in\BN}P[\omega](z)P[\omega^{-1}](z)<\infty.$$

	In \cite{PW}, the authors pointed out that the linear dependence of $\widetilde{Q}_2(\omega)$ is optimal, and they also proved that there exists a constant $C>0$ such that 
	$$(1/C) [\omega]_{A_2}\leq \widetilde{Q}_2(\omega)\leq C[\omega]_{A_2}^2.$$
	From \eqref{PWres} and the celebrated Rubio extrapolation theorem, one easily deduces that if $1<p<\infty$, $\omega\in A_p$ and $\alpha>1/2$, then for any $f\in L_\omega^p$,  
	\begin{align}\label{bdS}
		\|S_{\alpha}^{\widetilde{\bigtriangledown}}(P[f])\|_{L^p_\omega}\leq C[\omega]_{A_p}^{2\max\{1,1/(p-1)\}}\|f\|_{L^p_\omega},
	\end{align}
	where the constant $C>0$ depends only on $\alpha$, $n$ and $p$. We would like to stress that Theorem A is the first result on weighted inequalities of area integral on the higher dimensional unit balls.
	
	Despite the area integral associated with the Bergman gradient, another important type of area integral, which is related to Hardy spaces, is defined in terms of radial differential operators. Now we introduce such type of area integral. The radial differential operators, denoted by $R$ and $\overline{R}$, are defined as follows:
	$$Ru(z)=\sum_{i=1}^{n}z_i\frac{\partial u}{\partial z_i}(z),  \quad 
	\overline{R}u(z)=\sum_{i=1}^{n}\overline{z}_i\frac{\partial u}{\partial \overline{z}_i}(z), \quad \forall  u\in\mathcal{C}^1(\BN),$$
	where $z=(z_1,\cdots,z_n)\in\BN$. Let $X\in\{R,\overline{R}\}$. For $u\in\mathcal{C}^1(\BN)$, the Lusin area integral in terms of radial differential operators is defined by
	$$ S_{\alpha}^X(u)(\zeta)=\left(\int_{D_\alpha(\zeta)}
	\left|Xu(z)\right|^2\frac{dv(z)}{(1-|z|^2)^{n-1}}\right)^{1/2}, \quad \forall  \zeta\in\SN. $$
    One has the following pointwise estimate: for any $f\in L^1$
    	\begin{align}\label{hzhang1}
     S_{\alpha}^X(P[f])\leq CS_{\alpha}^{\widetilde{\bigtriangledown}}(P[f]).
     \end{align}
    See Section \ref{bergra} for more details. 
	
	Motivated by Segovia and Wheeden's remarkable paper \cite{SW} and  also by Petermichl and Wick's remarkable paper \cite{PW}, we aim to study the weighted inequalities of the area integral $S_{\alpha}^X$. The following theorem is our first main result: 
	
	\begin{theorem}\label{Main1}
		Let $1<p<\infty$ and $\omega\in A_p$.
		\begin{description}
			\item[(i)] Let $\alpha>1/2$. For any $f\in L^p_\omega$,
			$$\|S_{\alpha}^X(P[f])\|_{L^p_\omega}\leq C[\omega]_{A_p}^{\max\{1/2,1/(p-1)\}}\|f\|_{L^p_\omega},$$
			where the constant $C>0$ depends only on $\alpha$, $n$ and $p$. Moreover, the above exponent $\max\{1/2,1/(p-1)\}$ is optimal.
			\item[(ii)] Let $\alpha>1$. For any $f\in L^p_\omega$ with $P[f](0)=0$,
			$$\|f\|_{L^p_\omega}\leq C[\omega]_{A_p}^{\frac{1}{p-1}\max\{1/2,p-1\}}\left(\sum_{X\in\{R,\overline{R}\}}\|S_\alpha^{X}(P[f])\|_{L^p_\omega}\right),$$
			where the constant $C>0$ depends only on $\alpha$, $n$ and $p$.
		\end{description}
	\end{theorem}
    
    Compared with Petermichl and Wick's remarkable result \cite{PW} (i.e. Theorem A), our Theorem \ref{Main1} (i) deals with the area integral associated with the radial differential operators, while they considered the Bergman gradient. Besides, we also use different definitions of weights. It seems that the weight $\widetilde{Q}_2(\omega)$ is not usually used, and maybe larger than ours. Indeed, at the time of this writing, we do not know whether the following inequality holds:
    $$ (1/C) [\omega]_{A_2}\leq \widetilde{Q}_2(\omega)\leq C[\omega]_{A_2}.$$
    Note that combine Theorem A, \eqref{bdS} and \eqref{hzhang1}, and we can easily establish some weighted inequalities for $S_{\alpha}^X$. However, we would like to emphasize that while the exponent in \eqref{bdS} is not optimal, our Theorem \ref{Main1} provides the optimal orders in terms of $[\omega]_{A_p}$. 
    
    Besides, Theorem \ref{Main1} can also be seen as the higher dimensional variant of weighted inequalities of area integral in \cite[Theorem 1]{SW}, which gives an affirmative answer to Segovia and Wheeden's question. As another important application of Theorem \ref{Main1}, we give an alternative proof of Theorem A. We would also like to remark that with the help of Theorem \ref{Main1}, we can show the optimal weighted inequalities for $S_{\alpha}^{\widetilde{\bigtriangledown}}$ in terms of $[\omega]_{A_p}$.
    
    \begin{corollary}\label{Mcor1}
    	Let $1<p<\infty$ and $\omega\in A_p$.
    	\begin{description}
    		\item[(i)] Let $\alpha>1/2$. For any $f\in L^p_\omega$,
    		$$\|S_{\alpha}^{\widetilde{\bigtriangledown}}(P[f])\|_{L^p_\omega}\leq C[\omega]_{A_p}^{\max\{1/2,1/(p-1)\}}\|f\|_{L^p_\omega},$$
    		where the constant $C>0$ depends only on $\alpha$, $n$ and $p$. Moreover, the above exponent $\max\{1/2,1/(p-1)\}$ is optimal.
    		\item[(ii)] Let $\alpha>1$. For any $f\in L^p_\omega$ with $P[f](0)=0$,
    		$$\|f\|_{L^p_\omega}\leq C[\omega]_{A_p}^{\frac{1}{p-1}\max\{1/2,p-1\}}\|S_\alpha^{\widetilde{\bigtriangledown}}(P[f])\|_{L^p_\omega},$$
    		where the constant $C>0$ depends only on $\alpha$, $n$ and $p$.
    	\end{description}
    \end{corollary}

    On the other hand, we also establish the converse weighted inequalities, i.e. Theorem \ref{Main1} (ii) and Corollary \ref{Mcor1} (ii), while \cite[Theorem 1]{SW} and \cite{PW} do not involve these converse ones. It seems that there are some essential difficulties to adapt the arguments in \cite{SW} to our setting, since their proof for the unit disc depends heavily on the analyticity of holomorphic functions. This prompted us to seek for some new methods. Our approach to Theorem \ref{Main1} is based on the local mean oscillation technique and sparse domination arising from harmonic analysis, which is quite different from that of Segovia and Wheeden \cite{SW} or the Bellman function method used in \cite{PW}. To that end, we consider the following Littlewood-Paley type square function
	$$\g^*_{X,\lambda}(u)(\zeta)=\left(\int_{\BN}|Xu(z)|^2\left(\frac{1-|z|^2}{|1-\langle z,\zeta\rangle|}\right)^{\lambda n}\frac{dv(z)}{(1-|z|^2)^{n-1}}\right)^{1/2},\quad \forall  \zeta\in\SN, $$
	where $u\in\mathcal{C}^1(\BN)$ and $\lambda\in\mathbb{N}$ with $\lambda\geq4$. By the definition of the Kor\'anyi approach region, we have the following pointwise estimate
	\begin{align}\label{Area1}
		S_{\alpha}^X(u)(\zeta)\leq \alpha^{\lambda n/2}\g^*_{X, \lambda}(u)(\zeta), \quad \forall  \zeta\in\SN.
	\end{align}
	Therefore together with (\ref{Area1}), Theorem \ref{Main1} (i) follows directly from the following lemma. In addition, Theorem \ref{Main1} (ii) follows from Corollary \ref{Mcor1} by duality. As for the optimality of the exponent, we follow a similar argument in \cite{Ler4} to given an example to show it.
	
	\begin{lemma}\label{Rd1}
		Let $1<p<\infty$ and $\omega\in A_p$. Then  for any $f\in L^p_\omega$,
		$$\|\g^*_{X,\lambda}(P[f])\|_{L^p_\omega}\leq C[\omega]_{A_p}^{\max\{1/2,1/(p-1)\}}\|f\|_{L^p_\omega},$$
		where the constant $C>0$ depends only on $\lambda$, $n$ and $p$.
	\end{lemma}
	
	Our second result presents an equivalent characterization of weighted Hardy spaces in terms of the Lusin area integral. We come to the introduction of weighted Hardy spaces. Let $H(\BN)$ be the space of all holomorphic functions on $\BN$. For $1<p<\infty$, the weighted Hardy space $H_\omega^p$ is the set consisting of  all $f\in H(\BN)$ such that $$\|f\|_{H_\omega^p}^p:=\sup_{0<r<1}\int_{\SN}|f(r\zeta)|^p\omega(\zeta)d\sigma(\zeta)<\infty.$$
	In particular, when $\omega\equiv1$, we write $H^p$ for simplicity as the usual Hardy space. The following well-known theorem \cite{AB,FS,P} characterizes Hardy spaces in terms of the Lusin area integral $S_\alpha^R$.
	
	\
	
	\noindent \textbf{Theorem B.} Let $0<p<\infty$ and $f\in H^p$. Let $\alpha>1$. Then $f\in H^p$ if and only if $S_{\alpha}^R(f)\in L^p$. Moreover, if $f(0)=0$, then
	$$(1/C)\|f\|_{H^p}\leq\|S_{\alpha}^R(f)\|_{L^p}\leq C\|f\|_{H^p},$$
	where the constant $C>0$ depends only on $\alpha$, $n$ and $p$. 
	
	\
	
	The following theorem is our second main result, which can be seen as an analogous variant of Theorem B in the weighted setting.
	
	\begin{theorem}\label{Main2}
		Let $1<p<\infty$ and $\omega\in A_p$.
		\begin{description}
			\item[(i)] Suppose $\alpha>1/2$. Then for any $f\in H_\omega^p$, 
			$$\|S_{\alpha}^R(f)\|_{L^p_\omega}\leq C[\omega]_{A_p}^{\max\{1/2,1/(p-1)\}}\|f\|_{H^p_\omega},$$
			where the constant $C>0$ depends only on $\alpha$, $n$ and $p$.
			\item[(ii)] For $f\in H(\BN)$, if $S_{\alpha}^R(f)\in L^p_\omega$ for some $\alpha>1$, then $f\in H_\omega^p$. In particular, if $f(0)=0$, then
			$$\|f\|_{H^p_\omega}\leq C[\omega]_{A_p}^{\frac{1}{(p-1)^2}\max\{1/2,p-1\}}\|S_{\alpha}^{R}(f)\|_{L^p_\omega},$$
			where the constant $C>0$ depends only on $\alpha$, $n$ and $p$.
		\end{description}
	\end{theorem}

	Theorem \ref{Main2} (ii) is different from Theorem \ref{Main1} (ii) as now we require $f\in H(\BN)$. The upper bound of the weighted area integral in Theorem \ref{Main2} just follows from Theorem \ref{Main1} by virtue of Theorem \ref{Hp}. Turning to the lower bound, we apply duality, and combine Corollary \ref{Mcor1}. Thanks to Theorem \ref{Main1} and Corollary \ref{Mcor1}, our approach to the lower bound avoid the usual standard discussions of distribution functions in the weighted case \cite{BG,GW}.
	
	
	
	\

	Another important application of Theorem \ref{Main2} is about the weighted inequalities of Volterra integral operators. Given $g\in H(\BN)$, the Volterra integral (extended Ces\`aro) operator $J_g$ for $f\in H(\BN)$ is given by
	$$J_g f(z)=\int_{0}^{1}f(tz)Rg(tz)\frac{dt}{t}, \ \ \forall  z\in\BN.$$
	The study of Volterra integral operators has become a central topic in complex analysis.  Recall that BMOA is the set consisting of all holomorphic functions $f\in H^2$ such that
	$$\|f\|_{\mathrm{BMO}}^{2}:=|f(0)|^{2}+\operatorname*{sup}{\frac{1}{\sigma(B)}}\int_{B}|f-f_B|^{2}d\sigma<\infty,$$
	where the supremum is taken over all nonisotropic metric ball $B$ and $f_B=\frac{1}{\sigma(B)}\int_Bf$ is the average of $f$ over $B$. Let $n=1$. The integral operator $J_g$ was first introduced by Pommerenke \cite{Po}, and he proved that $J_g$ is bounded on $H^2$ if and only if $g\in BMOA$. Later on, Aleman and Siskakis \cite{AS} extended Pommerenke's result and established the boundedness of $J_g$ on $H^p$. The boundedness of $J_g$ between Hardy spaces for all index choices was studied in \cite{AC}. Besides, the spectra of $J_g$ on $H^p_\omega$ was considered in \cite{APe}. In higher dimensions, the general variant of Volterra integral operators $J_g$ (as defined here) was introduced by Hu \cite{Hu}, where he also investigated the boundedness and compactness of $J_g$ on the mixed norm space $H_{p,q}(\varphi)$. When $n\geq1$, Pau \cite{P} completely established the boundedness of $J_g$ between Hardy spaces $H^p$ and $H^q$ for the full range $0<p,q<\infty$. In particular, when $0<p=q<\infty$, Pau proved that $J_g$ is bounded on $H^p$ if and only if $g\in BMOA$. The interested reader is referred to \cite{AC, AS, CPPR, MPPW, PPWG, PRMem} for more information.
	
	The following theorem concerns the weighted inequalities of $J_g$.
	
\begin{theorem}\label{Main3}
	Let $1<p<\infty$ and $\omega\in A_p$. Let $g\in BMOA$. Then for any $f\in H_\omega^p$,
	$$\|J_gf\|_{H_\omega^p}\leq C[\omega]_{A_p}^{\varphi(p)}\|f\|_{H_\omega^p},$$
	where
	$$\varphi(p)=\frac{1}{(p-1)^2}\max\{1/2,p-1\}+\max\{1/2,1/(p-1)\}$$
	and the constant $C>0$ depends only on $n$, $p$ and $g$.
\end{theorem}
	
	We would like to remark that this is the first time that the weighted inequalities for Volterra integral operators are established. However, we do not know whether the exponent $\varphi(p)$ is optimal, which is also an interesting question for further study.
	
	Let us now give the idea of the proof of Theorem \ref{Main3}.  For $f\in H^p_\omega$ and $g\in BMOA$, it has been shown in \cite{Hu} that
	$$R(J_gf)(z)=f(z)Rg(z), \quad \forall  z\in\BN.$$
	It then follows that for $\alpha>1/2$,
	$$S_\alpha^R(J_gf)(\zeta)=\left(\int_{D_\alpha(\zeta)}|f(z)|^2\frac{(1-|z|^2)(Rg(z))^2dv(z)}{(1-|z|^2)^n}\right)^{1/2},\quad \forall  \zeta\in\SN.$$
	On the other hand, Theorem \ref{Main2} yields
	$$\|J_gf\|_{H_\omega^p}\leq C[\omega]_{A_p}^{\frac{2}{(p-1)^2}\max\{1/2,p-1\}}\|S_\alpha^R(J_gf)\|_{L^p_\omega},$$
	where $\alpha>1$. Hence, to finish the proof of Theorem \ref{Main3}, it suffices to show
	$$\|S_\alpha^R(J_gf)\|_{L^p_\omega}\leq C[\omega]_{A_p}^{\max\{1/2,1/(p-1)\}}\|f\|_{H_\omega^p}.$$
	
	Now we give our effort to the study of $S_\alpha^R(J_gf)$. It is known from \cite[Section 5.4]{Z} that for $g\in H(\BN)$,  $g\in \text{BMOA}$  if and only if the measure $(1-|z|^2)|Rg(z)|^2dv(z)$ is a Carleson measure. Recall that a positive Borel measure $\mu$ on $\BN$ is called a Carleson measure if there exists a constant $C>0$ such that $$\mu(\{z\in\BN: |1-\langle z,\zeta\rangle|^{1/2}<r\})\leq Cr^{2n}$$ for all $\zeta\in\SN$ and $r>0$. In the following, we consider the following more general area integral
	$$S^{\mu}_{\alpha}(u)(\zeta)=\left(\int_{D_\alpha(\zeta)}|u(z)|^2\frac{d\mu(z)}{(1-|z|^2)^n}\right)^{1/2},\quad \forall  \zeta\in\SN,$$
	where $u$ is a measurable function on $\BN$, and $\mu$ is a Carleson measure on $\BN$. We would like to remark that the area integral $S_\alpha^\mu$ turns to be quite useful in the theory of Volterra integral operators, Hardy spaces, Bergman spaces, and tent spaces; see \cite{Co,LvP,PWZ,PRS,Wu} for more information. Furthermore, as before we have
	\begin{align}\label{Area2}
		S_{\alpha}^\mu(u)(\zeta)\leq \alpha^{\lambda n/2} \g^*_{\mu,\lambda}(u)(\zeta), \quad \forall  \zeta\in\SN,
	\end{align}
	where $\lambda\in\mathbb{N}$ with $\lambda\geq4$, and $\g^*_{\mu,\lambda}$ is given by
	$$\g^*_{\mu,\lambda}(u)(\zeta)=\left(\int_{\BN}|u(z)|^2\left(\frac{1-|z|^2}{|1-\langle z,\zeta\rangle|}\right)^{\lambda n}\frac{d\mu(z)}{(1-|z|^2)^n}\right)^{1/2}, \quad \forall  \zeta\in\SN.$$
	From the preceding arguments, the proof of Theorem \ref{Main3} is reduced to the proof of the following lemma.
	
	\begin{lemma}\label{Rd3}
		Let $1<p<\infty$ and $\omega\in A_p$. Let $\mu$ be a Carleson measure on $\BN$. Then for $f\in L^p_\omega$,
		$$\|\g^*_{\mu,\lambda}(P[f])\|_{L^p_\omega}\leq C[\omega]_{A_p}^{\max\{1/2,1/(p-1)\}}\|f\|_{L^p_\omega},$$
		where the constant $C>0$ depends only on $\mu$, $\lambda$, $n$ and $p$.
	\end{lemma}
	
	For brevity, in the sequel we always assume that $\mu$ is a Carleson measure on $\BN$, and denote by $\g^*_\lambda$ both $\g^*_{X,\lambda}$ and $\g^*_{\mu,\lambda}$.
	
	Finally, we present the layout of this paper. As discussed above, we mainly need to show Theorem \ref{Main1} (ii), Corollary \ref{Mcor1}, Lemma \ref{Rd1}, Theorem \ref{Main2} and Lemma \ref{Rd3}.  As for the proof of Theorem \ref{Main2}, we need to apply Theorem \ref{Main1}, Corollary \ref{Mcor1} and Theorem B. Our proofs of Lemma \ref{Rd1} and Lemma \ref{Rd3} rely on the local mean oscillation technique, which stem from harmonic analysis. Even though our problems live in the world of complex analysis, we deal with them by utilizing the methods from real analysis. It seems that the main reason why such methods are feasible to carry out for our problems is that the geometric structures of $\SN$ share almost the same nature as those of Euclidean spaces.
	
	This paper is organized as follows. Section \ref{pre} is devoted to introducing the nonisotropic metric on $\SN$, maximal function, weighted Hardy spaces, dyadic systems on $\SN$ and the local mean oscillation formula (also called sparse domination). In Section \ref{bergra}, we introduce the Bergman gradient in $\BN$ and present some basic inequalities. In Section \ref{weaktype}, we show the weak type $(1,1)$ inequalities of $g_\lambda^*$, which allows us to implement sparse domination. We continue to provide some key estimates of $g_\lambda^*$ in Section \ref{keyestimate}. Our estimate of the smoothness of the Poisson kernel is quite different from that in harmonic analysis as some difficulties arise from $XP(z,\zeta)$ and $P(z,\zeta)$. Finally, we show Lemma \ref{Rd1} and Lemma \ref{Rd3} in Section \ref{prooflemmas}. The proofs of Theorem \ref{Main1}, Corollary \ref{Mcor1}, Theorem \ref{Main2} and Theorem \ref{Main3} will be finished in Section \ref{section6}.

	\bigskip

	\section{Preliminaries}\label{pre}
	\subsection{Notation}
	Throughout the paper, we write $A\lesssim B$ or $B\gtrsim A$ if $A\leq CB$ for some absolute positive constant $C$. Besides, $A\asymp B$ means that these inequalities as well as their inverses hold. We use the same letter $C$ to denote various positive constants which may vary at each occurrence. Given $p\in(1,\infty)$, we will denote by $p'=p/(p-1)$ its H\"older conjugate exponent. Let $\mathbb{N}=\{1,2,3,\cdots\}$, and $\mathbb{Z}$ be the set of integers. 
	
	Let $\omega\in A_p$. It is well-known that if $\omega\in A_p$, then $\omega^{-p'/p}\in A_{p'}$ and  $[\omega^{-p'/p}]_{A_{p'}}=[\omega]_{A_p}^{1/(p-1)}$ (for example, see \cite{Graf}). In the sequel, we write $\omega'=\omega^{-p'/p}$ for simplicity. Let $E$ be a measurable subset of $\SN$. We write $\omega(E)=\int_E\omega d\sigma$. We denote by $L^{1,\infty}$ the set of all complex-value  measurable functions on $\SN$ such that $$\|f\|_{L^{1,\infty}}=\sup_{\alpha>0}\alpha \sigma(\{\zeta\in \SN:|f(\zeta)|>\alpha\})<\infty.$$

	\subsection{Nonisotropic metric on $\SN$}\label{NMB}
	
	Let $n$ be a  positive integer, and $\mathbb{C}^n$ be the vector space of all ordered $n$-tuples $z=(z_1,\cdots,z_n)$ of complex numbers $z_i$. For any two points
	$$z=(z_1,\cdots,z_n), \ w=(w_1,\cdots,w_n)\in\CN,$$
	we write
	$\langle z,w\rangle=\sum_{i=1}^nz_i\overline{w}_i$ and $|z|=\sqrt{\langle z,z\rangle}$. The unit ball $\BN$ and unit sphere $\SN$ of $\mathbb{C}^n$ are defined as
	$$\BN=\{z\in\mathbb{C}^n:|z|<1\},\ \ \SN=\{z\in\mathbb{C}^n:|z|=1\}.$$
	
	For $z,w\in\overline{\BN}:=\BN\cup\SN$, let
	$$d(z,w)=|1-\langle z,w\rangle|^{1/2}.$$
	One can check that $d$ satisfies the triangle inequality  on $\overline{\BN}$. The restriction of $d(\cdot,\cdot)$ to $\SN$ is the so-called nonisotropic metric. 
	For $\zeta\in\SN$ and $r>0$, denote by $B(\zeta,r)$ the nonisotropic metric ball in $\SN$ with the center $\zeta$ and the radius $r$. When $r\geq\sqrt{2}$, $B(\zeta,r)=\SN$. Note that for all $\zeta\in\SN$ and $0<r\leq\sqrt{2}$ (see \cite{Z} for more details), 
	\begin{align}\label{eq2}
		\sigma(B(\zeta,r))\asymp r^{2n}.
	\end{align}
	It follows that there exists a constant $C>0$ depending on $n$ such that for any $\zeta\in\SN$, 
	\begin{align}\label{eq5}
		\sigma(B(\zeta,\delta r))\leq C\delta^{2n}\sigma(B(\zeta,r)),
	\end{align}
	whenever $0<r<\infty$ and $1\leq\delta<\infty$. Actually, the triple $(\SN,d,\sigma)$ forms a space of homogeneous type. We refer the interested reader to \cite{CW2} for more information on spaces of homogeneous type.
	
	We also need the following useful fact $$v(\{z\in\BN:d(\zeta,z)<r\})\lesssim r^{2n+2},$$
	whenever $0<r<\infty$. Indeed, if $0<r<1$, the desired result can be obtained by \eqref{eq2}, \eqref{eq5}, and the integration in polar coordinates (\cite[Lemma 1.8]{Z}). If $r\geq1$, then $$v(\{z\in\BN:d(\zeta,z)<r\})\leq1\leq r^{2n+2}.$$
	
	Another useful fact proved in \cite{MPPW} reads that for $\alpha>1$,
	\begin{align}\label{eq10}
		\int_{\SN}\int_{D_{\alpha}(\zeta)}F(z)\frac{d\mu(z)}{(1-|z|^2)^n}d\sigma(\zeta)\asymp\int_{\BN}F(z)d\mu(z)
	\end{align}
	holds for any positive measurable function $F$ and  any  finite positive measure $\mu$ on $\BN$. The reader is referred to \cite{R} and \cite{Z} for more information about complex analysis on $\SN$. 
	
	\subsection{Maximal function}
	For $f\in L^1$, the Hardy-Littlewood maximal function on $\SN$ denoted by $Mf$ is defined as
	$$Mf(\zeta)=\sup_{\delta>0}\frac{1}{\sigma(B(\zeta,\delta))}\int_{B(\zeta,\delta)}|f|d\sigma, \ \ \zeta\in\SN.$$
	It has already been proved in \cite{Z} that $M$ is of strong type $(p,p)$ and weak type $(1,1)$. Furthermore, the following weighted norm inequalities can be found in \cite{Buc,Graf}.

	\begin{theorem}\label{Apw}
		Let $1<p<\infty$ and $\omega\in A_p$. Then for any $f\in L_\omega^p$, $$\|Mf\|_{L_\omega^p}\lesssim[\omega]_{A_p}^{1/(p-1)}\|f\|_{L_\omega^p}.$$
	\end{theorem}

	Let $\alpha>1/2$. The non-tangential maximal function of a measurable function $u$ on $\BN$ is defined by
	\begin{equation}\label{haonon}
		N_\alpha u(\zeta)=\sup_{z\in D_\alpha(\zeta)}|u(z)|,\ \ \zeta\in\SN.
	\end{equation}
	From \cite[Theorem 4.10]{Z}, we have that for any $f\in L^1$, $N_\alpha(P[f])(\zeta)\lesssim Mf(\zeta)$. Then Theorem \ref{Apw} implies that for $1<p<\infty$ and $\omega\in A_p$,
	\begin{align}\label{NonM}
		\|N_\alpha(P[f])\|_{L^p_\omega}\lesssim[\omega]_{A_p}^{1/(p-1)}\|f\|_{L^p_\omega}, \ \ \forall  f\in L_\omega^p.
	\end{align}
	
	\subsection{Weighted Hardy spaces}
	Recall that the weighted Hardy space $H_\omega^p$ is defined as the set of all holomorphic functions $f$ on $\BN$ equipped with the norm: $$\|f\|_{H_\omega^p}^p=\sup_{0<r<1}\int_{\SN}|f(r\zeta)|^p\omega(\zeta)d\sigma(\zeta)<\infty.$$
	In this  subsection, we will show that each $f\in H_\omega^p$ can be written as the Poisson integral of a measurable function. To that end, we introduce the K-limit of a function $f$.
	\begin{definition}
		A function $f$ on $\BN$ has K-limit $L$ at some $\zeta\in\SN$ if for every $\alpha>1/2$,
		$$\lim_{z\to\zeta,z\in D_\alpha(\zeta)}f(z)=L.$$
		In this case, we write $\text{K-}\lim f(\zeta)=L$. For brevity, we denote by $f_*(\zeta)=\text{K-}\lim f(\zeta)$ if it exists.
	\end{definition}
	
	It was shown in \cite[Theorem 4.25 and Corollary 4.27]{Z} that the K-limit exists whenever $f$ belongs to $H^p$.
	
	\begin{theorem}\label{Hp2}
		Let $0<p<\infty$. If $f\in H^p$, then $f_*(\zeta)$ exists for almost all $\zeta\in \SN$. Moreover, $\|f\|_{H^p}=\|f_*\|_{L^p}$ and
		$$\lim_{r\to1^-}\int_{\SN}|f(r\zeta)-f_*(\zeta)|^pd\sigma(\zeta)=0.$$
		In addition, if $1\leq p<\infty$, then $f=P[f_*]$.
	\end{theorem}

	We now present a similar phenomenon as above in the weighted Hardy spaces, which can be stated as follows.
	
	\begin{theorem}\label{Hp}
		Let $1<p<\infty$ and $\omega\in A_p$. If $f\in H^p_\omega$, then $f_*(\zeta)$ exists for almost all $\zeta\in\SN$, and  $f=P[f_*]$. Moreover,  
		$$\|f_*\|_{L^p_\omega}\leq\|f\|_{H^p_\omega}\lesssim[\omega]_{A_p}^{1/(p-1)}\|f_*\|_{L^p_\omega}$$
		and
		$$\lim_{r\to1^-}\int_{\SN}|f(r\zeta)-f_*(\zeta)|^p\omega(\zeta)d\sigma(\zeta)=0.$$
	\end{theorem}
	
	\begin{proof}
		Recall $\omega'=\omega^{-p'/p}$. If $\omega\in A_p$, then $\omega'(\SN)<\infty$. Then H\"older's inequality gives
		\begin{align}\label{h}
			\int_{\SN}|f(r\zeta)|d\sigma(\zeta)=&\int_{\SN}|f(r\zeta)|\omega^{1/p}(\zeta)\omega^{-1/p}(\zeta)d\sigma(\zeta)\\
			\nonumber \leq&\|f\|_{H^p_\omega}(\omega'(\SN))^{1/p'},
		\end{align}
		which implies $f\in H^1$. Hence, $f_*(\zeta)$ exists for almost all $\zeta\in \SN$ and $f=P[f_*]$ according to Theorem \ref{Hp2}.
		
		For the second statement, notice that each region $ D_\alpha(\zeta)$ contains $r\zeta$ whenever $r$ is sufficiently close to $1$. Hence, $\lim_{r\to1^-}f_r(\zeta)=f_*(\zeta)$ and by Fatou's lemma $$\|f_*\|_{L^p_\omega}^p\leq\liminf_{r\to1^-}\int_{\SN}|f(r\zeta)|^p\omega(\zeta)d\sigma(\zeta)\leq\|f\|_{H^p_\omega}^p.$$
		On the other hand, since $r\zeta\in D_1(\zeta)$, one has $|f(r\zeta)|=|P[f_*](r\zeta)|\leq N_1(P[f_*])(\zeta)$. It follows from \eqref{NonM} that
		$$\int_{\SN}|f(r\zeta)|^p\omega(\zeta)d\sigma(\zeta)\leq\|N_1(P[f_*])\|_{L^p_\omega}^p\lesssim[\omega]_{A_p}^{p/(p-1)}\|f_*\|_{L^p_\omega}^p,$$
		which implies $\|f\|_{H^p_\omega}\lesssim[\omega]_{A_p}^{1/(p-1)}\|f_*\|_{L^p_\omega}$. Finally, by the dominated convergence theorem, we have
		$$\lim_{r\to1^-}\int_{\SN}|f(r\zeta)-f_*(\zeta)|^p\omega(\zeta)d\sigma(\zeta)=0.$$
		This completes the proof.
	\end{proof}

	\subsection{Dyadic systems}
	We introduce dyadic systems on $\SN$. The following theorem can be found in \cite{HK}.
	
	\begin{theorem}\label{de}
		Let $0<c_0\leq C_0<\infty$ and $0<r<1$ be fixed constants satisfying $12 C_0r\leq c_0$. Given a set of points $\zeta_j^k\in\SN$, $j\in\mathscr{A}_k$ (an index set), for every $k\in\mathbb{Z}$, with the properties that if $i\neq j$ then $d(\zeta_i^k,\zeta_j^k)\geq c_0r^k$, and for any $\zeta\in\SN$, $\min_{j\in\mathscr{A}_k}d(\zeta,\zeta_j^k)<C_0r^k$. Then there exists a family of open subsets (called dyadic cubes)
		$$\mathscr{D}:=\{Q_j^k\subset\SN:j\in\mathscr{A}_k,k\in\mathbb{Z}\}$$
		($\mathscr{D}$ is called a dyadic system with parameters $C_0$, $c_0$ and $r$) such that the following properties hold:
		\begin{itemize}
			\item[(i)] $\SN=\cup_{j\in\mathscr{A}_k}Q_j^k$ for all $k\in\mathbb{Z}$, and $Q_j^k\cap Q_i^k=\emptyset$ whenever $i\neq j$;
			\item[(ii)] If $\ell\geq k$, $i\in\mathscr{A}_\ell$, and $j\in\mathscr{A}_k$, then either $Q_i^\ell\subset Q_j^k$ or $Q_i^\ell\cap Q_j^k=\emptyset$;
			\item[(iii)] For every $k\in\mathbb{Z}$ and $j\in\mathscr{A}_k$,
			\begin{align}\label{dya1}
				B(\zeta_j^k,c_1r^k)\subset Q_j^k\subset B(\zeta_j^k,C_1r^k)=:B(Q_j^k),
			\end{align}
			where $c_1:=c_0/3$ and $C_1:=2C_0$; If $\ell\geq k$ and $Q^\ell_i\subset Q^k_j$, then $B(Q^\ell_i)\subset B(Q^k_j)$.
		\end{itemize}
	\end{theorem}
	
	\begin{definition}\label{def2.8}
		Given a cube $Q_j^k$, denote by $\zeta_j^k$ the ``center" of $Q_j^k$ and $r^k$ the ``radius" of $Q_j^k$. For $\delta>1$, we denote by $\delta Q_j^k$ the ``dilated" set $\delta Q_j^k=B(\zeta_j^k,\delta C_1r^k)$.
	\end{definition}
	
	\begin{remark}\label{de1}
		The following observations are very helpful.
		\begin{itemize}
			\item [(i)] It follows from Theorem \ref{de} that for each $Q_j^k, j\in\mathscr{A}_k, k\in\mathbb{Z}$, there is at least one child $Q_i^{k+1}, i\in\mathscr{A}_{k+1}$ and exactly one parent $Q_\ell^{k-1}, \ell\in\mathscr{A}_{k-1}$ such that $Q_i^{k+1}\subset Q_j^k\subset Q_\ell^{k-1}$. If $Q_i^{k+1}$ is a child of $Q_j^k$ with $i\in\mathscr{A}_{k+1}$, $j\in\mathscr{A}_k,$ and  $k\in\mathbb{Z}$, then $B(\zeta_i^{k+1},c_1r^{k+1})\subset Q_i^{k+1}$ and $Q_j^{k}\subset B(\zeta_j^{k},C_1r^{k})$. Note that there exists a positive integer $\mathscr{N}$ such that $C_1\leq 2^{\mathscr{N}}c_1r$. Then by \eqref{eq5} and the fact $B(\zeta_j^{k},C_1r^{k})\subset B(\zeta_i^{k+1},2C_1r^{k})$,
			\begin{align}\label{doub2}
				\sigma(Q_j^{k})\leq&\sigma(B(\zeta_j^{k},C_1r^{k}))\leq\sigma(B(\zeta_i^{k+1},2C_1r^{k}))\\
				\nonumber \leq&\sigma(B(\zeta_i^{k+1},2^{\mathscr{N}+1}c_1r^{k+1}))\lesssim\sigma(B(\zeta_i^{k+1},c_1r^{k+1}))\\
				\nonumber \leq&\sigma(Q_i^{k+1}).
			\end{align}
			\item [(ii)] It is well-known from \cite{HK} that there exist a constant $\rho>0$ and a finite collection of dyadic systems $\{\mathscr{D}_\tau: \tau=1,2,\cdots,\mathscr{K}\}$ so that for any nonisotropic metric ball $B(\zeta,r)\subset\SN$ there exists $Q\in\mathscr{D}_\tau$ for some $\tau\in\{1,2,\cdots,\mathscr{K}\}$ satisfying $B(\zeta,r)\subset Q$ and $\text{diam}(Q)\leq \rho r$. 
		\end{itemize}
	\end{remark}

	\subsection{Local mean oscillation formula}
	Let $Q$ be a subset of $\SN$. Lerner's local oscillation formula, which is also called sparse domination, involves the following concepts:
	\begin{itemize}
		\item [(i)] A median of a real-valued measurable function $f$ on $Q$ is any real number $m_f(Q)$ such that 
		$$\sigma(\{\zeta\in Q:f(\zeta)>m_f(Q)\})\leq\sigma(Q)/2, \ \ \sigma(\{\zeta\in Q:f(\zeta)<m_f(Q)\})\leq\sigma(Q)/2;$$
		\item [(ii)] The decreasing rearrangement of $f$ is defined by
		$$f^*(t)=\inf\{\alpha>0:\sigma(\{\zeta\in\SN:|f(\zeta)|>\alpha\})\leq t\}=\inf\{\|f\chi_{\SN\setminus E}\|_{L^\infty}:\sigma(E)\leq t\},$$
		where both infimums are actually attained respectively by letting $\alpha=f^*(t)$ and $$E=\{\zeta\in\SN:|f(\zeta)|>f^*(t)\};$$
		\item [(iii)] The oscillation of $f$ on $Q$ is
		$$\omega_{\epsilon}(f;Q)=\inf_c(\chi_Q(f-c))^*(\epsilon\sigma(Q)),$$
		where $0<\epsilon<1$.
	\end{itemize}
	We list here some basic properties of these objects and refer to \cite{CUMP,Hy,Ler2} for more details.
	\begin{itemize}
		\item[(i)] $|m_f(Q)|\leq (\chi_Qf)^*(\nu\sigma(Q))$, where $\nu\in(0,1/2)$;
		\item[(ii)] $f^*(t)\leq\frac{1}{t}\|f\|_{L^{1,\infty}}$, where $0<t<\infty$;
		\item[(iii)] $(\chi_Q(f-m_f(Q)))^*(\nu\sigma(Q))\leq2\omega_\nu(f;Q)$, where $\nu\in(0,1/2)$;
		\item[(iv)] if $|f|\leq|g|$ a.e., then $f^*(t)\leq g^*(t)$, where $0<t<\infty$;
		\item[(v)] $(f+g)^*(t)\leq f^*(t/2)+g^*(t/2)$, where $0<t<\infty$;
		\item[(vi)] $(fg)^*(t)\leq f^*(t/2)g^*(t/2)$, where $0<t<\infty$.
	\end{itemize}
	
	\begin{definition}
		Let $\mathscr{D}$ be a dyadic system. A collection of dyadic sets $S\subset\mathscr{D}$ is called a sparse family if for $Q\in S$, we have
		$$\sigma(Q)\leq 2\sigma(E(Q)),$$
		where $E(Q)=Q\setminus(\bigcup_{P\in S;P\subsetneqq Q}P)$.
	\end{definition}
	
	Lerner's local mean oscillation formula in the following form is from \cite{Hy}. Their proof is based on the Lebesgue differentiation theorem for the median \cite{F}. Although the arguments of \cite{F} and \cite{Hy} involve $\mathbb{R}^n$, they are also applicable on spaces of homogeneous type. For reader's convenience, we still provide the proof below. We denote by $\mathscr{D}(Q)$ the set of all dyadic cubes in $\mathscr{D}$ which are contained in $Q$.
	
	\begin{theorem}\label{LMO}
		Let $Q^0$ be a dyadic cube in a dyadic system $\mathscr{D}$ and let $f$ be a real-valued measurable function on $\SN$. Then there exists a (possibly empty) sparse family $S(Q^0)\subset \mathscr{D}(Q^0)$ such that for a.e. $\zeta\in Q^0$,
		$$|f(\zeta)-m_f(Q^0)|\leq 2\sum_{Q\in S(Q^0)}\omega_\epsilon(f;Q)\chi_{Q}(\zeta)$$
		for some $\epsilon\in(0,1)$.
	\end{theorem}
	
	\begin{proof}
		We can write the median Calder\'on-Zygmund decomposition
		\begin{align}\label{mCZ}
			\chi_{Q^0}(f-m_f(Q^0))
			=&\chi_{Q^0\setminus\cup Q_j^{(1)}}(f-m_f(Q^0))+\sum_j\chi_{Q_j^{(1)}}(m_f(Q_j^{(1)})-m_f(Q^0))\\
			\nonumber&+\sum_j\chi_{Q_j^{(1)}}(f-m_f(Q_j^{(1)})),
		\end{align}
		where $\{Q_j^{(1)}\}_j$ is a family of pairwise disjoint subcubes of $Q^0$ following a specific choice of the stopping cubes: they are the maximal dyadic subcubes of $Q^0$ such that
		\begin{align}\label{St}
			\max_{Q'\in ch(Q_j^{(1)})}|m_f(Q')-m_f(Q^0)|>(\chi_{Q^0}(f-m_f(Q^0)))^*(\epsilon\sigma(Q^0)).
		\end{align}
		Here, the notation $Q'\in ch(Q_j^{(1)})$ means that $Q'$ is one child of $Q_j^{(1)}$ and the constant $0<\epsilon<1/2$ will be chosen  later.
		
		By the choice of the stopping cubes $Q_j^{(1)}$, the converse estimate of \eqref{St} holds for all dyadic subset $Q'$ of $Q^0$ with $Q'\subset(Q^0\setminus\cup Q_j^{(1)})$. More precisely, for such $Q'$, we have
		$$|m_f(Q')-m_f(Q^0)|\leq(\chi_{Q^0}(f-m_f(Q^0)))^*(\epsilon\sigma(Q^0)).$$
		Then by the Lebesgue differentiation theorem for the median proved in \cite{F} ($m_f(Q')\to f(\zeta)$ as $Q'\to \zeta$ for almost every $\zeta$), the first term on the right-side of \eqref{mCZ} is dominated by
		$$\chi_{Q^0\setminus\cup Q_j^{(1)}}(\chi_{Q^0}(f-m_f(Q^0)))^*(\epsilon\sigma(Q^0))\leq\chi_{Q^0\setminus\cup Q_j^{(1)}}2\omega_{\epsilon}(f;Q^0).$$
		By the maximality in the choice of the stopping cubes $Q_j^{(1)}$, 
		$$|m_f(Q_j^{(1)})-m_f(Q^0)|\leq(\chi_{Q^0}(f-m_f(Q^0)))^*(\epsilon\sigma(Q^0)).$$
		Then, the second term on the right-side of \eqref{mCZ} is dominated by $\sum_j\chi_{Q_j^{(1)}}2\omega_{\epsilon}(f;Q^0)$.
		
		In summary, we have
		$$|\chi_{Q^0}(f-m_f(Q^0))|
		\leq\chi_{Q^0}2\omega_\epsilon(f;Q^0)+\sum_j|\chi_{Q_j^{(1)}}(f-m_f(Q_j^{(1)}))|.$$
		Repeat this stopping step, and by iteration, one has
		\begin{align*}
			&|\chi_{Q^0}(f-m_f(Q^0))|\\
			\leq&\chi_{Q^0}2\omega_\epsilon(f;Q^0)+\sum_j|\chi_{Q_j^{(1)}}(f-m_f(Q_j^{(1)}))|\\
			\leq&\chi_{Q^0}2\omega_\epsilon(f;Q^0)+\sum_j\chi_{Q_j^{(1)}}2\omega_\epsilon(f;Q_j^{(1)})+\sum_i|\chi_{Q_i^{(2)}}(f-m_f(Q_i^{(2)}))|\\
			\leq&\cdots\leq\sum_{k=0}^m\sum_j\chi_{Q_j^{(k)}}2\omega_\epsilon(f;Q_j^{(k)})+\sum_i|\chi_{Q_i^{(m+1)}}(f-m_f(Q_i^{(m+1)}))|,
		\end{align*}
		where the cubes $Q_i^{(m+1)}$ are dyadic subcubes of some $Q_j^{(m)}$, chosen by a similar stopping criterion as in the choice of $Q_j^{(1)}$ from $Q^0$ in \eqref{St}.
		
		Let $$E(Q_j^{(m)}):=Q_j^{(m)}\setminus\cup_i Q_i^{(m+1)}.$$
		We claim $\sigma(E(Q_j^{(m)}))\geq\frac{1}{2}\sigma(Q_j^{(m)})$, then
		$$\sigma(\Omega^{m+1}):=\sigma(\cup_i Q_i^{(m+1)})
		\leq\frac{1}{2}\sigma(\Omega^m)\leq\cdots\leq\frac{1}{2^{m+1}}\sigma(Q^0).$$
		As $m\to\infty$, the support of the term $\sum_i|\chi_{Q_i^{(m+1)}}(f-m_f(Q_i^{(m+1)}))|$ tends to a null set. Therefore,
		$$|\chi_{Q^0}(f-m_f(Q^0))|\leq\sum_{k=0}^\infty\sum_j\chi_{Q_j^{(k)}}2\omega_{\epsilon}(f;Q_j^{(k)}),$$
		which implies the desired result.
		
		Now we show $\sigma(E(Q_j^{(m)}))\geq\frac{1}{2}\sigma(Q_j^{(m)})$ to finish the proof. By symmetry, it suffices to consider $m=0$.
		The stopping condition yields that there exists some $Q'\in ch(Q_j^{(1)})$ such that
		$$|m_f(Q')-m_f(Q^0)|>(\chi_{Q^0}(f-m_f(Q^0)))^*(\epsilon\sigma(Q^0)).$$
		By \eqref{doub2}, there exists a constant $C\in(0,1/2)$ such that $C\sigma(Q_j^0)\leq\sigma(Q')$. Let $f_0=f-m_f(Q^0)$. We can choose $m_{f_0}(Q')=m_f(Q')-m_f(Q^0)$. Then for any $\nu\in(0,1/2)$, one has
		$$\alpha:=(\chi_{Q^0}f_0)^*(\epsilon\sigma(Q^0))<|m_{f_0}(Q')|\leq(\chi_{Q'}f_0)^*(\nu\sigma(Q'))
		\leq(\chi_{Q_j^1}f_0)^*(\nu C\sigma(Q_j^{(1)})).$$
		Thus, by the definition of $(\chi_{Q_j^1}f_0)^*(\nu C\sigma(Q_j^{(1)}))$,
		$$\sigma(Q_j^{(1)}\cap\{\zeta\in\SN:|f_0(\zeta)|>\alpha\})\geq \nu C\sigma(Q_j^{(1)}).$$
		We hence have
		\begin{align*}
			\nu C\sum_j\sigma(Q_j^{(1)})\leq&\sum_j\sigma(Q_j^{(1)}\cap\{\zeta\in\SN:|f_0(\zeta)|>\alpha\})\\
			\leq&\sigma(Q^0\cap\{\zeta\in\SN:|f_0(\zeta)|>\alpha\})\leq\epsilon\sigma(Q^0).
		\end{align*}
		By taking $\epsilon=C/4$ and letting $\nu\to\frac{1}{2}^-$, we have $\sum_j\sigma(Q_j^{(1)})\leq\frac{1}{2}\sigma(Q^0)$. This yields the desired inequality. 
	\end{proof}

	\bigskip

	\section{Bergman gradient on $\BN$}\label{bergra}
	In this section, we introduce the Bergman gradient on $\BN$. We will also give some useful inequalities involving the Lusin area integral associated with the Bergman gradient.
	
	The Bergman kernel on $\BN$ is given by $$K(z,w)=\frac{1}{(1-\langle z,w\rangle)^{n+1}}.$$
	Let
	$$g_{i,j}(z)=\frac{\partial^2}{\partial z_i\partial\overline{z}_j}
	\log K(z,z)=\frac{n+1}{(1-|z|^2)^2}[(1-|z|^2)\delta_{i,j}+\overline{z}_iz_j],$$
	where $\delta_{i,j}$ is the Kronecker delta,  which is equal to one when $i=j$ and equal to zero when $i\neq j$. For $z\in\BN$ and $\zeta\in\mathbb{C}^n$, set
	$$\beta_{\BN}^2(z,\zeta)=\sum_{i,j=1}^{n}g_{i,j}(z)\zeta_i\overline{\zeta}_j.$$
	If $\gamma:[0,1]\to\BN$ is a $\mathcal{C}^1$ curve, the Bergman length of $\gamma$, denoted by $|\gamma|_{\BN}$, is given by
	$$|\gamma|_{\BN}=\int_{0}^{1}\beta_{\BN}(\gamma(t),\gamma'(t))dt.$$
	For $z,w\in\BN$, the Bergman metric is defined by
	$$\delta_{\BN}(z,w)=\inf\{|\gamma|_{\BN}: \gamma(0)=z,\gamma(1)=w\},$$
	where the infimum is taken over all $\mathcal{C}^1$ curves from $z$ to $w$. The vector fields
	$$\left\{\frac{\partial}{\partial z_1},\frac{\partial}{\partial \overline{z}_1},\cdots,\frac{\partial}{\partial z_n},\frac{\partial}{\partial \overline{z}_n}\right\}$$
	form a basis on the complex tangent space. The Hermitian inner product with respect to the Bergman metric on the complexification of the tangent space at a point $z\in\BN$ is defined by
	$$\left\langle\frac{\partial}{\partial z_i},\frac{\partial}{\partial z_j}\right\rangle_B
	=\left\langle\frac{\partial}{\partial \overline{z}_j},\frac{\partial}{\partial \overline{z}_i}\right\rangle_B
	=\frac{1}{2}g_{i,j}(z)$$
	and
	$$\left\langle\frac{\partial}{\partial z_i},\frac{\partial}{\partial \overline{z}_j}\right\rangle_B=\left\langle\frac{\partial}{\partial \overline{z}_i},\frac{\partial}{\partial z_j}\right\rangle_B=0.$$
	For $u\in \mathcal{C}^1(\BN)$, the gradient with respect to the Bergman metric is given by $$\widetilde{\bigtriangledown}u=2\sum_{i, j}g^{i,j}\left(\frac{\partial u}{\partial \overline{z}_i}\frac{\partial}{\partial z_j}+\frac{\partial u}{\partial z_j}\frac{\partial}{\partial\overline{z}_i}\right),$$
	where $(g^{i,j})_{i,j}$ is the inverse matrix of $(g_{i,j})_{i,j}$ and
	$$g^{i,j}(z)=\frac{1-|z|^2}{n+1}(\delta_{i,j}-\overline{z}_iz_j).$$
	In particular, for $u\in \mathcal{C}^1(\BN)$,
	$$|\widetilde{\bigtriangledown}u|^2=\langle \widetilde{\bigtriangledown}u,\widetilde{\bigtriangledown}u\rangle_B
	=(\widetilde{\bigtriangledown}u)\overline{u}=2\sum_{i,j}g^{i,j}\left(\frac{\partial u}{\partial\overline{z}_i}\overline{\frac{\partial u}{\partial\overline{z}_j}}+\overline{\frac{\partial u}{\partial z_i}}\frac{\partial u}{\partial z_j}\right).$$
	For more details, see \cite{ABC,BBG,R,Sto}.
	
	Denote by $Y\in\{T_{i,j},\overline{T}_{i,j}: 1\leq i<j\leq n\}$ one of the following complex tangential differential operators:
	\begin{align}\label{Comtan}
		T_{i,j}u(z)=\overline{z}_i\frac{\partial u}{\partial z_j}(z)-\overline{z}_j\frac{\partial u}{\partial z_i}(z) \ \ \text{and} \ \ \overline{T}_{i,j}u(z)=z_i\frac{\partial u}{\partial \overline{z}_j}(z)-z_j\frac{\partial u}{\partial \overline{z}_i}(z),
	\end{align}
	where $u\in\mathcal{C}^1(\BN)$ and $z\in\BN$. For $u\in\mathcal{C}^1(\BN)$, note that
	$$|z|^2\sum_{i=1}^{n}\left|\frac{\partial u}{\partial z_i}(z)\right|^2=|Ru(z)|^2+\sum_{1\leq i<j\leq n}|T_{i,j}u(z)|^2,$$
	$$|z|^2\sum_{i=1}^{n}\left|\frac{\partial u}{\partial \overline{z}_i}(z)\right|^2=|\overline{R}u(z)|^2+\sum_{1\leq i<j\leq n}|\overline{T}_{i,j}u(z)|^2,$$
	and
	\begin{align}\label{gra5}
		|\widetilde{\bigtriangledown}u(z)|^2=\frac{2(1-|z|^2)}{n+1}\left(\sum_{i=1}^{n}\left|\frac{\partial u}{\partial z_i}(z)\right|^2-|Ru(z)|^2+\sum_{i=1}^{n}\left|\frac{\partial u}{\partial \overline{z}_i}(z)\right|^2-|\overline{R}u(z)|^2\right),
	\end{align}
	then one has
	\begin{align}\label{gra}
		|\widetilde{\bigtriangledown}u(z)|^2=\frac{2(1-|z|^2)}{(n+1)|z|^2}\left[(1-|z|^2)\sum_X|Xu(z)|^2
		+\sum_Y|Yu(z)|^2\right].
	\end{align}

	\begin{remark}
		From \eqref{gra}, the following pointwise estimate holds:
		\begin{align}\label{gra2}
			S_{\alpha}^X(u)(\zeta)\lesssim S_{\alpha}^{\widetilde{\bigtriangledown}}(u)(\zeta), \quad \ \forall  \zeta\in\SN.
		\end{align}
	\end{remark}
	
	Let us recall that $u\in\mathcal{C}^2(\BN)$ is said to be M-harmonic if $\widetilde{\bigtriangleup}u=0$, where $\widetilde{\bigtriangleup}$ is the Laplace-Beltrami operator given by $$\widetilde{\bigtriangleup}=\frac{4(1-|z|^2)}{n+1}\sum_{i,j=1}^{n}(\delta_{i,j}-\overline{z}_iz_j)\frac{\partial^2}{\partial z_j\partial\overline{z}_i}.$$
	The operator $\widetilde{\bigtriangleup}$ is invariant under the action of the automorphism group of $\BN$. Note that $P[f]$ is M-harmonic on $\BN$ for any $f\in L^1$ (\cite[P. 51]{Sto}). By the same argument as in \eqref{h}, we see $L^p_\omega\subset L^1$ for $1<p<\infty$ and $\omega\in A_p$. Therefore $P[f]$ is M-harmonic in $\BN$ for any $f\in L^p_\omega$. We refer to \cite{Sto} for more information about M-harmonic functions.

	Recall that the invariant Green function (\cite{Sto}) of $\BN$ is given by
	$$G(z)=\frac{1}{2n}\int_{|z|}^{1}(1-t^2)^{n-1}t^{-2n+1}dt$$
	and in the case $n=1$, this reduces to the usual logarithm, that is $G(z)=\frac{1}{2}\log\frac{1}{|z|}$. 
	
	\begin{proposition}\label{Gdelta}
		Let $n\geq2$. There exists a constant $\delta_G\in\left(0,\frac{1}{2}\right)$ such that if $|z|\leq\delta_G$, then 
		$$ G(z)\lesssim\frac{1}{|z|^{2n-2}}. $$
	\end{proposition}
	
	\begin{proof}
		The result can be obtained by the fact that if $|z|\to0^+$ then $G(z)\asymp\frac{1}{|z|^{2n-2}}$, where $n\geq2$. See \cite[Proposition 1.26]{Z} for more details.
	\end{proof}
	
	Let $u\in\mathcal{C}^1(\BN)$. Denote by $\bigtriangledown u(z)=\left(\frac{\partial u}{\partial z_1},\frac{\partial u}{\partial z_2},\cdots,\frac{\partial u}{\partial z_n}\right)$ the complex gradient of $u$. By \eqref{gra5} and the fact that $\frac{\partial u}{\partial \overline{z}_i}=\overline{\frac{\partial \overline{u}}{\partial z_i}}$, we deduce that
	$$|\widetilde{\bigtriangledown} u(z)|^2=\frac{2(1-|z|^2)}{n+1}
	(|\bigtriangledown u(z)|^2-|Ru(z)|^2
	+|\bigtriangledown \overline{u}(z)|^2
	-|R\overline{u}(z)|^2).$$
	Together with $|Ru(z)|\leq|z||\bigtriangledown u(z)|,$ we deduce that
	\begin{align}\label{gra3}
		\begin{split}
			(1-|z|^2)^2(|\bigtriangledown u(z)|^2+|\bigtriangledown \overline{u}(z)|^2)
			\lesssim|\widetilde{\bigtriangledown} u(z)|^2
			\lesssim|\bigtriangledown u(z)|^2+|\bigtriangledown \overline{u}(z)|^2.
		\end{split}
	\end{align}
	For $0<r<1$, set
	$$r\BN=\{z\in\mathbb{C}^n: |z|<r\} \quad\quad \text{and} \quad\quad r\SN=\{z\in\mathbb{C}^n: |z|=r\}.$$
	
 The following pointwise estimate will be used later.
	
	\begin{lemma}\label{Gdelta2}
		Let $1<\beta<\alpha$ and $\delta_G$ be the constant as in Proposition \ref{Gdelta}. For any $f\in L^1$, the following inequality holds:
		$$|\widetilde{\bigtriangledown}P[f](z)|\lesssim S_\alpha^{\widetilde{\bigtriangledown}}(P[f])(\zeta),\quad\forall~z\in D_\beta(\zeta)\cap\delta_G\BN,~\zeta\in\SN.$$
	\end{lemma}
	
	\begin{proof}
		By \eqref{gra3}, 
		\begin{align}\label{gra4}
			|\widetilde{\bigtriangledown} P[f](z)|\asymp|\bigtriangledown P[f](z)|+|{\bigtriangledown} P[\overline{f}](z)|,\quad\forall~|z|<\delta_G.
		\end{align}
		Let $\beta\in(1,\alpha)$. Let $E_\epsilon(z)=\varphi_z(\epsilon \BN)$, where $0<\epsilon<1$, $z\in\BN$ and $\varphi_z$ is an automorphism of $\BN$ that interchanges the points $0$ and $z$ (see \cite[Chapter 1]{Z} for more details). Choose $\epsilon$ sufficiently small so that $E_\epsilon(z)\subset D_\alpha(\zeta)$ for any $z\in D_\beta(\zeta)$. Note that $P[f]$ is M-harmonic. By \cite[Lemma 2.7]{ABC}, we get
		$$|\bigtriangledown P[f](z)|^2\lesssim\frac{1}{(1-|z|^2)^{n+1}}\int_{E_\epsilon(z)}|\bigtriangledown P[f](w)|^2 dv(w)$$
		and
		$$|{\bigtriangledown} P[\overline{f}](z)|^2\lesssim\frac{1}{(1-|z|^2)^{n+1}}\int_{E_\epsilon(z)}|{\bigtriangledown} P[\overline{f}](w)|^2 dv(w).$$
		Combining with \eqref{gra4} and the fact that $1\asymp1-|z|^2\asymp1-|w|^2$ for any $w\in E_\epsilon(z)$ with $|z|\leq\delta_G$, we have that for any $|z|\leq\delta_G$,
		\begin{align*}
			|\widetilde{\bigtriangledown} P[f](z)|^2
			\lesssim&\int_{E_\epsilon(z)}|\bigtriangledown P[f](w)|^2+|{\bigtriangledown} P[\overline{f}](w)|^2\frac{dv(w)}{(1-|w|^2)^{n+1}}\\
			=&\int_{E_\epsilon(z)}(1-|w|^2)^2|\bigtriangledown P[f](w)|^2+(1-|w|^2)^2|{\bigtriangledown} P[\overline{f}](w)|^2\frac{dv(w)}{(1-|w|^2)^{n+1}}\\
			\lesssim&\int_{D_\alpha(\zeta)}|\widetilde{\bigtriangledown} P[f](w)|^2 \frac{dv(w)}{(1-|w|^2)^{n+1}},
		\end{align*}
		as desired.
	\end{proof}
	
	Let $$d\lambda_n(z)=\frac{(n+1)dv(z)}{(1-|z|^2)^{n+1}} \quad\quad \text{and} \quad\quad d\widetilde{\sigma}(r\zeta)=\frac{2nr^{2n-1}d\sigma(\zeta)}{(1-r^2)^n}.$$
	It is well-known that Green's formula for $\widetilde{\bigtriangleup}$ on $\BN$ itself with $n>1$ does not hold because $\BN$ equipped with the Bergman metric is a noncompact manifold. But it holds on every compact subset of $\BN$ (see \cite[Theorem 1.25]{Z}). For $u\in\mathcal{C}^2(\BN)$, Zheng \cite[P. 17]{Zheng} proved that
	\begin{align}\label{Greenn}
		\int_{r\BN}\widetilde{\bigtriangleup}(u)(z)(G(z)-G(r))d\lambda_n(z)
		=&-\int_{r\SN}u(z)\frac{\partial G}{\partial\widetilde{n}_r}(z)d\widetilde{\sigma}(z)-u(0)\\
		\nonumber =&\int_{\SN}u(r\zeta)d\sigma(\zeta)-u(0),
	\end{align}
	where $\partial\widetilde{n}_r$ is the outward normal derivative along $r\SN$, and in the last equality we use the following identity (see \cite[P. 30]{Z})
	$$-\frac{\partial G}{\partial\widetilde{n}_r}(r\zeta)d\widetilde{\sigma}(r\zeta)=d\sigma(\zeta).$$

	The following lemma is very useful in the proof of Theorem \ref{Main1} and Theorem \ref{Main2}.

	\begin{lemma}\label{Grlem}
		Let $\alpha>1$. Assume $f,g\in L^1$ and $P[f](0)P[g](0)=0$. Then, for any $0<r<1$,
		\begin{align}\label{Green9}
			\left|\int_{\SN}P[f](r\zeta)P[h](r\zeta)d\sigma(\zeta)\right|\lesssim\int_{\SN}S_\alpha^{\widetilde{\bigtriangledown}}(P[f])(\zeta)
			S_\alpha^{\widetilde{\bigtriangledown}}(P[\overline{h}])(\zeta)d\sigma(\zeta).
		\end{align}
	\end{lemma}

	\begin{proof}
		Choose $\beta\in(1,\alpha)$. Note that the function $P[f]P[h]$ is of class $\mathcal{C}^2(\BN)$. Let $0<r<1$. Applying \eqref{Greenn}, we have
		\begin{align}\label{Green8}
			\begin{split}
				\left|\int_{\SN}P[f](r\zeta)P[h](r\zeta)d\sigma(\zeta)\right|
				\leq&\int_{r\BN}|\widetilde{\bigtriangleup}(P[f]P[h])(z)|(G(z)-G(r))d\lambda_n(z)\\
				\leq&\int_{\BN}|\widetilde{\bigtriangleup}(P[f]P[h])(z)|G(z)d\lambda_n(z).
			\end{split}
		\end{align}
		This implies from \eqref{eq10} that the last term above is comparable with
		\begin{align}\label{Green2}
			\int_{\SN}\int_{D_\beta(\zeta)}|\widetilde{\bigtriangleup}(P[f]P[h])(z)|G(z)\frac{d\lambda_n(z)}{(1-|z|^2)^n}d\sigma(\zeta).
		\end{align}
		Note $\widetilde{\bigtriangleup}P[f](z)=\widetilde{\bigtriangleup}P[h](z)=0$ (\cite[P. 51]{Sto}). Then (see \cite[(3.11) and P. 29]{Sto}) 
		\begin{align*}
			\widetilde{\bigtriangleup}(P[f]P[h])=&P[h]\widetilde{\bigtriangleup}P[f]+2(\widetilde{\bigtriangledown}P[f])(P[h])
			+P[f]\widetilde{\bigtriangleup}P[h]\\
			=&2\langle\widetilde{\bigtriangledown}P[f],\widetilde{\bigtriangledown}P[\overline{h}]\rangle_B,
		\end{align*}
		which implies 
		\begin{align}\label{Green3}
			|\widetilde{\bigtriangleup}(P[f]P[h])(z)|\leq2|\widetilde{\bigtriangledown}P[f](z)||\widetilde{\bigtriangledown}P[\overline{h}](z)|.
		\end{align}
		
		Let $\delta_G$ be the constant as in Proposition \ref{Gdelta}. For $z\in D_\beta(\zeta)\cap(\BN\setminus\delta_G\BN)$, it follows from \cite[Lemma 6.6]{Sto} that $G(z)\lesssim(1-|z|^2)^n$. Then by \eqref{Green3}, we have
		\begin{align}\label{Green4}
			\begin{split}
				&\int_{\SN}\int_{D_\beta(\zeta)\cap(\BN\setminus\delta_G\BN)}|\widetilde{\bigtriangleup}(P[f]P[h])(z)|
				G(z)\frac{d\lambda_n(z)}{(1-|z|^2)^n}d\sigma(\zeta)\\
				\lesssim&\int_{\SN}\int_{D_\beta(\zeta)}|\widetilde{\bigtriangledown}P[f](z)|
				|\widetilde{\bigtriangledown}P[\overline{h}](z)|\frac{dv(z)}{(1-|z|^2)^{n+1}}d\sigma(\zeta)\\
				\leq&\int_{\SN}S_\alpha^{\widetilde{\bigtriangledown}}(P[f])(\zeta)
				S_\alpha^{\widetilde{\bigtriangledown}}(P[\overline{h}])(\zeta)d\sigma(\zeta),
			\end{split}
		\end{align}
		where in the last inequality we use H\"older's inequality.

		By \eqref{Green3} again and the fact that $1\asymp1-|z|^2$ for any $|z|<\delta_G$,
		\begin{align}\label{Green5}
			\begin{split}
				&\int_{\SN}\int_{D_\beta(\zeta)\cap\delta_G\BN}|\widetilde{\bigtriangleup}(P[f]P[h])(z)|
				G(z)\frac{d\lambda_n(z)}{(1-|z|^2)^n}d\sigma(\zeta)\\
				\lesssim&\int_{\SN}\int_{D_\beta(\zeta)\cap\delta_G\BN}|\widetilde{\bigtriangledown}P[f](z)|
				|\widetilde{\bigtriangledown}P[\overline{h}](z)|G(z)dv(z)d\sigma(\zeta)
				=:V.
			\end{split}
		\end{align}
		By Lemma \ref{Gdelta2},
		\begin{align}\label{Green6}
			\begin{split}
				V\lesssim&\int_{\SN}\int_{D_\beta(\zeta)\cap\delta_G\BN}G(z)dv(z)S_\alpha^{\widetilde{\bigtriangledown}}(P[f])(\zeta)
				S_\alpha^{\widetilde{\bigtriangledown}}(P[\overline{h}])(\zeta)d\sigma(\zeta)\\
				\lesssim&\int_{\SN}S_\alpha^{\widetilde{\bigtriangledown}}(P[f])(\zeta)
				S_\alpha^{\widetilde{\bigtriangledown}}(P[\overline{h}])(\zeta)d\sigma(\zeta),
			\end{split}
		\end{align}
		where in the last inequality we use Proposition \ref{Gdelta} if $n\geq2$. In fact,
		$$\int_{\delta\BN}\frac{1}{|z|^{2n-2}}dv(z)\lesssim1.$$
		If $n=1$, then $G(z)=\frac{1}{2}\log\frac{1}{|z|}\leq\frac{1-|z|}{2|z|}$. Then
		$$\int_{\delta\mathbb{B}_1}G(z)dv(z)\lesssim\int_{0}^{1}(1-t)dt\lesssim1.$$
		By \eqref{Green5} and \eqref{Green6}, we have
		\begin{align}\label{Green7}
			\begin{split}
				&\int_{\SN}\int_{D_\beta(\zeta)\cap\delta_G\BN}|\widetilde{\bigtriangleup}(P[f]P[h])(z)|
				G(z)\frac{d\lambda_n(z)}{(1-|z|^2)^n}d\sigma(\zeta)\\
				\lesssim&\int_{\SN}S_\alpha^{\widetilde{\bigtriangledown}}(P[f])(\zeta)
				S_\alpha^{\widetilde{\bigtriangledown}}(P[\overline{h}])(\zeta)d\sigma(\zeta).
			\end{split}
		\end{align}
		Finally, by \eqref{Green8}, \eqref{Green4} and \eqref{Green7}, we get \eqref{Green9}. This completes the proof.
	\end{proof}

By \eqref{gra2}, one has that for $\alpha>1/2$, $f\in L^1$ and $\zeta\in\SN$,
$$ \sum_{X\in\{R,\overline{R}\}}S_{\alpha}^{X}(P[f])(\zeta)\lesssim S_{\alpha}^{\widetilde{\bigtriangledown}}(P[f])(\zeta). $$
 The following proposition indicates that the inverse inequality also holds.

\begin{proposition}\label{dim2}
	Let $1/2<\beta<\alpha$. 
	For $f\in L^1$, if $\sum\limits_{X\in\{R,\overline{R}\}}S_{\alpha}^{X}(P[f])(\zeta)<\infty$, then
	\begin{align}\label{hol}
		S_{\beta}^{\widetilde{\bigtriangledown}}(P[f])(\zeta)\lesssim \sum_{X\in\{R,\overline{R}\}}S_{\alpha}^{X}(P[f])(\zeta).
	\end{align}
\end{proposition}

\begin{proof}
	Let $\delta_G$ be the constant as in Proposition \ref{Gdelta}. We first show
	\begin{align}\label{hgra1}
		\left(\int_{D_\beta(\zeta)\cap(\BN\setminus\delta_G\BN)}
		\left|\widetilde{\bigtriangledown} P[f](z)\right|^2\frac{dv(z)}{(1-|z|^2)^{n+1}}\right)^{1/2}\lesssim \sum_{X\in\{R,\overline{R}\}}S_{\alpha}^{X}(P[f])(\zeta).
	\end{align}
	By \cite[Lemma 2.1]{JP} and as argued in \cite[Lemma 3.5]{AB}, one can show that
	$$\sum_{Y\in\{T_{i,j}: 1\leq i<j\leq n\}}S_{\beta}^Y(P[f])(\zeta)\lesssim \sum_{X\in\{R,\overline{R}\}}S_{\alpha}^X(P[f])(\zeta),\quad\forall\zeta\in\SN,$$
	where $Y\in\{T_{i,j}: 1\leq i<j\leq n\}$ is defined in \eqref{Comtan} and
	$$S_{\beta}^Y(P[f])(\zeta)=\left(\int_{D_\beta(\zeta)}
	\left|YP[f](z)\right|^2\frac{dv(z)}{(1-|z|^2)^{n}}\right)^{1/2}.$$
	Note that for $z\in D_\beta(\zeta)\cap(\BN\setminus\delta_G\BN)$, $|z|\geq\delta_G$. Hence, by \eqref{gra}, we get \eqref{hgra1}.
	
	Next, we show
	\begin{align}\label{hgra2}
		\left(\int_{D_\beta(\zeta)\cap\delta_G\BN}
		\left|\widetilde{\bigtriangledown} P[f](z)\right|^2\frac{dv(z)}{(1-|z|^2)^{n+1}}\right)^{1/2}\lesssim \sum_{X\in\{R,\overline{R}\}}S_{\alpha}^{X}(P[f])(\zeta).
	\end{align}
	By \eqref{gra3},
	$$|\widetilde{\bigtriangledown} P[f](z)|^2\asymp|\bigtriangledown P[f](z)|^2+|\bigtriangledown P[\overline{f}](z)|^2,\quad\forall z\in D_\beta(\zeta)\cap\delta_G\BN.$$
	Combining with $1-|z|^2\asymp1$ for $z\in D_\beta(\zeta)\cap\delta_G\BN$, we have
	\begin{align}\label{hgra4}
		\begin{split}
			&\left(\int_{D_\beta(\zeta)\cap\delta_G\BN}
			\left|\widetilde{\bigtriangledown} P[f](z)\right|^2\frac{dv(z)}{(1-|z|^2)^{n+1}}\right)^2\\
			\asymp&\left(\int_{D_\beta(\zeta)\cap\delta_G\BN}
			\left|\bigtriangledown P[f](z)\right|^2+\left|\bigtriangledown P[\overline{f}](z)\right|^2\frac{dv(z)}{(1-|z|^2)^{n-1}}\right)^2.
		\end{split}
	\end{align}
	Set
	$$S_\alpha(r,\zeta):=\left\{w\in D_\alpha(\zeta):\frac{1-r^2}{2}<1-|w|^2<2(1-r^2)\right\},\quad\forall\zeta\in\SN,~0<r<1.$$
	Let $E_\epsilon(z)=\varphi_z(\epsilon \BN)$, where $z\in\BN$ and $\varphi_z$ is an automorphism of $\BN$ that interchanges the points $0$ and $z$. Choose $\epsilon$ sufficiently small so that $E_\epsilon(z)\subset S_\alpha(r,\zeta)$ for any $z=r\eta\in D_\beta(\zeta)\cap\delta_G\BN$, $\eta\in\SN$ and $0<r<1$. In the same way as in \cite[Lemma 2.3 and Lemma 2.6]{BBG},
	\begin{align*}
		|\widetilde{\bigtriangledown}[\re (RP[f])](r\eta)|+|\widetilde{\bigtriangledown}[\im (RP[f])](r\eta)|
		\lesssim\left(\int_{E_\epsilon(r\eta)}|RP[f](w)|^2\frac{dv(w)}{(1-|w|^2)^{n+1}}\right)^{1/2}
	\end{align*}
	and
	\begin{align*}
		|\widetilde{\bigtriangledown}[\re (RP[\overline{f}])](r\eta)|+|\widetilde{\bigtriangledown}[\im (RP[\overline{f}])](r\eta)|
		\lesssim\left(\int_{E_\epsilon(r\eta)}|RP[\overline{f}](w)|^2\frac{dv(w)}{(1-|w|^2)^{n+1}}\right)^{1/2}.
	\end{align*}
	Here $\re h$ and $\im h$ denote the real part and imaginary part of $h$, respectively. It follows from \eqref{gra3} that for any $z=r\eta\in D_\beta(\zeta)\cap\delta_G\BN$,
	\begin{align*}
		&|\bigtriangledown[\re (RP[f])](r\eta)|
		+|\bigtriangledown[\im (RP[f])](r\eta)|+|\bigtriangledown[\re (RP[\overline{f}])](r\eta)|
		+|\bigtriangledown[\im (RP[\overline{f}])](r\eta)|\\
		\lesssim&\left(\sum_{X\in\{R,\overline{R}\}}\int_{S_\alpha(r,\zeta)}|XP[f](w)|^2\frac{dv(w)}{(1-|w|^2)^{n+1}}\right)^{1/2}=:\mathcal{J}_r.
	\end{align*}
	Together with
	$$P[f](z)-P[f](0)=\sum_{X\in\{R,\overline{R}\}}\int_{0}^{1}\frac{XP[f](tz)}{t}dt,$$
	we deduce that
	$$|\bigtriangledown P[f](r\eta)|^2\lesssim\left(\frac{1}{r}\int_{0}^{r}\mathcal{J}_tdt\right)^2\leq\frac{1}{r}\int_{0}^{r}\mathcal{J}_t^2dt.$$
	Integrating in polar coordinates in $D_\beta(\zeta)$, and using the fact that
	$$\sigma(\{\eta\in\SN:r\eta\in D_\beta(\zeta)\})\lesssim(1-r)^n,$$
	we deduce that
	\begin{align*}
		\int_{D_\beta(\zeta)\cap\delta_G\BN}
		\left|\bigtriangledown P[f](z)\right|^2\frac{dv(z)}{(1-|z|^2)^{n-1}}
		\lesssim&\int_{0}^{1}(1-r^2)\int_{0}^{r}\mathcal{J}_t^2dtdr\\
		\lesssim&\int_{0}^{1}(1-r^2)^2\mathcal{J}_r^2dr,
	\end{align*}
	where in the last inequality we use Hardy's inequality (\cite[Lemma 3.4]{AB}). Hence, we obtain
	\begin{align*}
		\int_{D_\beta(\zeta)\cap\delta_G\BN}
		\left|\bigtriangledown P[f](z)\right|^2\frac{dv(z)}{(1-|z|^2)^{n-1}}
		\lesssim&\sum_{X\in\{R,\overline{R}\}}\int_{0}^{1}(1-r^2)^2\int_{S_\beta(r,\zeta)}|XP[f](w)|^2\frac{dv(w)}{(1-|w|^2)^{n+1}}dr\\
		\asymp&\sum_{X\in\{R,\overline{R}\}}\int_{0}^{1}\int_{S_\beta(r,\zeta)}|XP[f](w)|^2\frac{dv(w)}{(1-|w|^2)^{n-1}}dr\\
		\lesssim&\left(\sum_{X\in\{R,\overline{R}\}}S_{\alpha}^{X}(P[f])(\zeta)\right)^2.
	\end{align*}
	By similar arguments,
	\begin{align*}
		\int_{D_\beta(\zeta)\cap\delta_G\BN}
		\left|\bigtriangledown P[\overline{f}](z)\right|^2\frac{dv(z)}{(1-|z|^2)^{n-1}}
		\lesssim\left(\sum_{X\in\{R,\overline{R}\}}S_{\alpha}^{X}(P[\overline{f}])(\zeta)\right)^2
		=\left(\sum_{X\in\{R,\overline{R}\}}S_{\alpha}^{X}(P[f])(\zeta)\right)^2.
	\end{align*}
	We hence have from \eqref{hgra4} that
	\begin{align*}
		\int_{D_\beta(\zeta)\cap\delta_G\BN}
		\left|\widetilde{\bigtriangledown} P[f](z)\right|^2\frac{dv(z)}{(1-|z|^2)^{n+1}}
		\lesssim\left(\sum_{X\in\{R,\overline{R}\}}S_{\alpha}^{X}(P[f])(\zeta)\right)^2,
	\end{align*}
	which implies \eqref{hgra2}. Finally, by \eqref{hgra1} and \eqref{hgra2}, we complete the proof of \eqref{hol}.
\end{proof}

	\bigskip
	
	\section{Weak $(1,1)$ estimate of $\g^*_{\lambda}$ in the unweighted setting}\label{weaktype}
	In this section, we establish the weak type $(1,1)$ estimate of $\g^*_{\lambda}$ associated with Poisson integral in the unweighted setting for the purpose of proving Lemma \ref{Rd1} and Lemma \ref{Rd3}. Recall that $\g^*_\lambda\in\{\g^*_{X,\lambda},\g^*_{\mu,\lambda}\}$ and $X\in\{R,\overline{R}\}$, where $R$ and $\overline{R}$ are given by
	$$R=\sum_{i=1}^{n}z_i\frac{\partial}{\partial z_i},  \ \ \overline{R}=\sum_{i=1}^{n}\overline{z}_i\frac{\partial}{\partial \overline{z}_i},$$
	and $\mu$ is a Carleson measure on $\BN$.
	
	\begin{theorem}\label{Weak7}
		Let $\lambda\in\mathbb{N}$ with $\lambda\geq4$. Then for any $f\in L^1$,
		$$\|\g^*_{\lambda}(P[f])\|_{L^{1,\infty}}\lesssim\|f\|_{L^1}.$$
	\end{theorem}
	We would like to remark that Theorem \ref{Weak7} is also new in the complex setting. The proof of Theorem \ref{Weak7} relies on the weak type $(1,1)$ estimate of area integrals $S_\alpha^X(P[f])$ and $S_\alpha^\mu(P[f])$, i.e. Lemma \ref{Weak}. Apart from Lemma \ref{Weak}, we also need some auxiliary facts that enable us to reduce the proof of Theorem \ref{Weak7} to Lemma \ref{Weak}. 
	
	We divide this section into two parts. The first one aims to show Lemma \ref{Weak}. The other one is to prove Lemma \ref{ACest9}, which allows us to implement the reduction.

	\subsection{Weak $(1,1)$ estimates of area integrals}
	This subsection is devoted to the weak type $(1, 1)$ estimate of $S_\alpha^X(P[f])$ and $S_\alpha^\mu(P[f])$, which reads as follows.
	\begin{lemma}\label{Weak}
		Let $\alpha>1/2$. Then for any $f\in L^1$,
		\begin{itemize}
			\item[(i)] $\|S_\alpha^X(P[f])\|_{L^{1,\infty}}\lesssim\|f\|_{L^1}$;
			\item[(ii)] $\|S_\alpha^\mu(P[f])\|_{L^{1,\infty}}\lesssim\|f\|_{L^1}$.
		\end{itemize}
	\end{lemma}
	We begin to show Lemma \ref{Weak}. Fix $f\in L^1$ and $t>0$. Let $T\in\{S_\alpha^X\circ P,S_\alpha^\mu\circ P\}$ and set 
	\begin{align}\label{reduce2}
		\Omega:=\{\zeta\in\SN: M(f)(\zeta)>t\}.
	\end{align}
	Since $M$ is of weak type $(1,1)$, one has $\sigma(\Omega)\lesssim t^{-1}\|f\|_{L^1}$. Hence, if $\Omega=\SN$,
	$$\sigma(\{\zeta\in\SN:|Tf(\zeta)|>t\})\leq\sigma(\SN)=\sigma(\Omega)\lesssim\frac{1}{t}\|f\|_{L^1}.$$
	Therefore, we only need to consider $\Omega\subsetneqq\SN$, and it suffices to show
	\begin{align}\label{reduce1}
		\sigma(\{\zeta\in\SN\setminus\Omega:|Tf(\zeta)|>t\})\lesssim\frac{1}{t}\|f\|_{L^1}.
	\end{align}
	So we transfer the proof of Lemma \ref{Weak} to the proof of (\ref{reduce1}), which relies on the standard Calder\'on-Zygmund argument. At first, let us introduce the following Whitney decomposition of $\Omega$, see for instance \cite[Page 69]{CW}.
	\begin{lemma}\label{Wh1}
		Let $\Omega\subsetneqq\SN$ be defined as in (\ref{reduce2}). There are constants $\mathscr{N}\in\mathbb{N}$, $0<a_1<1$ and $a_2,a_3>1$ depending only on $n$, and a sequence $\{B(\zeta_k,r_k)\}_k$ of nonisotropic metric balls such that
		\begin{itemize}
			\item[(i)] $\cup_k B(\zeta_k,r_k)=\Omega$;
			\item[(ii)] $B(\zeta_k,a_2r_k)\subset\Omega$ and $B(\zeta_k,a_3r_k)\cap(\SN\setminus\Omega)\neq\emptyset$;
			\item[(iii)] the balls $B(\zeta_k,a_1r_k)$ are pairwise disjoint;
			\item[(iv)] no point in $\Omega$ lies in more than $\mathscr{N}$ different balls $B(\zeta_k,a_2r_k)$.
		\end{itemize}
	\end{lemma}
	
	\begin{remark}\label{Wh2}
		One can construct  a collection of cubes $\{Q_k\}$ from the above collection $\{B(\zeta_k,r_k)\}_k$
		so that $Q_k$ are disjoint, $\cup_kQ_k=\Omega$ and $B(\zeta_k,a_1r_k)\subset Q_k\subset B(\zeta_k,r_k)$. See \cite[Page 15]{S3} for more details.
	\end{remark}
	
	From Lemma \ref{Wh1} and Remark \ref{Wh2}, we obtain collections of balls $\{B(\zeta_k,r_k)\}$ and cubes $\{Q_k\}$ satisfying the properties stated above. Now define $g(\zeta)=f(\zeta)$ for $\zeta\notin\Omega$, and
	$$g(\zeta)=\frac{1}{\sigma(Q_k)}\int_{Q_k}fd\sigma, \ \ \text{if} \ \ \zeta\in Q_k.$$
	Decompose $f=g+b=g+\sum_{k}b_k$, where
	$$b_k(\zeta)=\chi_{Q_k}\left[f(\zeta)-\frac{1}{\sigma(Q_k)}\int_{Q_k}fd\sigma\right],$$
	such that the following assertions hold (see e.g. \cite[Page 17]{S3}):
	\begin{itemize}
		\item[(cz-i)] $\|g\|_{L^\infty}\lesssim t $ and $\|g\|_{L^2}^2\lesssim t\|f\|_{L^1}$;
		\item[(cz-ii)] Each $b_k$ is supported in $Q_k$,
		$$\|b_k\|_{L^1}\lesssim t\sigma(Q_k) \ \ \text{and} \ \ \int_{\SN}b_kd\sigma=0;$$
		\item[(cz-iii)] $\sum_k\sigma(Q_k)\lesssim \frac{1}{t}\|f\|_{L^1}$.
	\end{itemize}
	
	Therefore, in order to prove (\ref{reduce1}), it suffices to show
	\begin{align}\label{reduce3}
		\sigma\big(\{\zeta\in\SN\setminus\Omega:|Tg(\zeta)|>t\}\big)\lesssim t^{-1}\|f\|_{L^1},
	\end{align}
	and
	\begin{align}\label{reduce4}
		\sigma\big(\left\{\zeta\in\SN\setminus\Omega:|Tb(\zeta)|>t\right\}\big)\lesssim t^{-1}\|f\|_{L^1}
	\end{align}
	for $T\in\{S_\alpha^X\circ P,S_\alpha^\mu\circ P\}$. Next, we will show \eqref{reduce3} and \eqref{reduce4} respectively.

	\subsubsection{Proof of (\ref{reduce3}): estimate for the good function $g$}\quad
	
	
	
	\begin{proof}[Proof of (\ref{reduce3}).]
		Firstly, consider $T=S_\alpha^X\circ P$. By \eqref{gra2} and Theorem A, we have $$\|S_{\alpha}^X(P[g])\|_{L^2}\lesssim\|S_{\alpha}^{\widetilde{\bigtriangledown}}(P[g])\|_{L^2}\lesssim\|g\|_{L^2},$$
		where the last inequality is due to \eqref{NonM} with $\omega=1$.
		Hence,  by using Chebyshev's inequality and (cz-i), one gets
		$$\sigma\big(\{\zeta\in\SN\setminus\Omega:S_{\alpha}^X(P[g])>t\}\big)\le\frac{\|S_{\alpha}^X(P[g])\|^2_{L^2}}{t^2}\lesssim\frac{\|g\|^2_{L^2}}{t^2}\lesssim\frac{1}{t}\|f\|_{L^1}.$$
		
		As for $T=S_\alpha^\mu\circ P$, we use  \eqref{eq10} to deduce that
		$$\|S_{\alpha}^{\mu}(P[g])\|_{L^2}^2\lesssim\int_{\BN}|P[g](z)|^2d\mu(z)
		\lesssim\|g\|_{L^2}^2,$$
		where the last inequality follows from the well-known Carleson embedding theorem \cite[Theorem 5.9]{Z}. Therefore, in the same way, one has
		$$\sigma\big(\{\zeta\in\SN\setminus\Omega:S_{\alpha}^\mu(P[g])>t\}\big)\lesssim\frac{1}{t}\|f\|_{L^1}.$$
		This finishes the proof.
	\end{proof}

	\subsubsection{Proof of (\ref{reduce4}): estimate for the bad function $b$}\quad
	
	The estimate for the bad function $b$ is much more involved compared with the Euclidean setting. As mentioned before, the main difficulties lie in the smoothness of $XP(z,\zeta)$ and $P(z,\zeta)$. We proceed with the pointwise estimate of $XP(z,\zeta)$ and $P(z,\zeta)$.
	
	For $\ell\in\mathbb{N}$ and $z,w\in\overline{\BN}$, we define
	$$K_{\ell}(z,w):=\frac{1}{(1-\langle z,w\rangle)^\ell}.$$
	We have the following size condition and regularity condition for $XP(z,\zeta)$.
	\begin{lemma}\label{smooth3}
		Suppose that $z\in\BN$ and $\xi,\eta\in\SN$. Then 
		\begin{itemize}
			\item[(i)] $|XP(z,\xi)|\lesssim(1-|z|^2)^{n-1}|K_{2n}(z,\xi)|$;
			\item[(ii)] $|XP(z,\xi)-XP(z,\eta)|\lesssim\mathscr{M}(z,\xi,\eta)$.
		\end{itemize}
		Here $\mathscr{M}(z,\xi,\eta)$ is given by
		\begin{align*}
			\mathscr{M}(z,\xi,\eta):=&(1-|z|^2)^n|1-\langle z,\xi\rangle|\cdot|K_{2n+2}(z,\xi)-K_{2n+2}(z,\eta)|\\
			&+(1-|z|^2)^n|\langle z,\xi\rangle-\langle z,\eta\rangle|\cdot|K_{2n+2}(z,\eta)|\\
			&+(1-|z|^2)^{n-1}|K_{2n}(z,\xi)-K_{2n}(z,\eta)|.
		\end{align*}
	\end{lemma}
	
	\begin{proof}
		We only deal with the case $X=R$, since the estimate for $\overline{R}$ is quite similar. By a simple calculation,
		$$RP(z,\xi)=n\frac{(\overline{1-\langle z,\xi\rangle})(1-|z|^2)^n\langle z,\xi\rangle}{|1-\langle z,\xi\rangle|^{2n+2}}-n\frac{(1-|z|^2)^{n-1}|z|^2}{|1-\langle z,\xi\rangle|^{2n}}.$$
		This immediately gives (i) by using $|1-\langle z,\xi\rangle|\gtrsim1-|z|^2$.
		
		On the other hand, by the triangle inequality,
		\begin{align*}
			&|(\overline{1-\langle z,\xi\rangle})\langle z,\xi\rangle|K_{2n+2}(z,\xi)|-(\overline{1-\langle z,\eta\rangle})\langle z,\eta\rangle|K_{2n+2}(z,\eta)||\\
			\leq&|(1-\langle z,\xi\rangle)\langle z,\xi\rangle|\cdot||K_{2n+2}(z,\xi)|-|K_{2n+2}(z,\eta)||\\
			&+|K_{2n+2}(z,\eta)|\cdot|(\overline{1-\langle z,\xi\rangle})\langle z,\xi\rangle-(\overline{1-\langle z,\eta\rangle})\langle z,\eta\rangle|
		\end{align*}
		and using the fact that 
		$$|(\overline{1-\langle z,\xi\rangle})\langle z,\xi\rangle-(\overline{1-\langle z,\eta\rangle})\langle z,\eta\rangle|\lesssim|\langle z,\xi\rangle-\langle z,\eta\rangle|,$$
		we get (ii).
	\end{proof}
	
	The following lemma can be found in \cite[Lemma 4.29]{Z} (or see \cite[Lemma 6.1.1]{R}).
	\begin{lemma}\label{smooth1}
		Suppose that $z\in\BN$, $\xi,\eta\in\SN$ and $\ell\in\mathbb{N}$. Then 
		\begin{itemize}
			\item[(i)] $|\langle z,\xi\rangle-\langle z,\eta\rangle|\leq2|1-\langle z,\eta\rangle|^{1/2}|1-\langle\xi,\eta\rangle|^{1/2}+|1-\langle\xi,\eta\rangle|$;
			\item[(ii)] $|K_\ell(z,\eta)-K_\ell(z,\xi)|\leq\sum_{k=0}^{\ell-1}|(\langle z,\eta\rangle-\langle z,\xi\rangle)K_{k+1}(z,\eta)K_{\ell-k}(z,\xi)|$.
		\end{itemize}
	\end{lemma}

	\begin{lemma}\label{smooth2}
		Let $\alpha>1/2$. Let $\zeta,\xi,\zeta_0\in\SN$ satisfy $d(\zeta_0,\xi)<\delta$ and $d(\zeta_0,\zeta)\geq C\delta$ for some  $\delta>0$ and $C>1$. Then for all $z\in D_\alpha(\zeta)$,
		$$d(\zeta,\zeta_0)\lesssim d(z,\xi) \ \ \mbox{and} \ \ d(\zeta,\zeta_0)\lesssim d(z,\zeta_0).$$
	\end{lemma}
	
	\begin{proof}
		For the first inequality, by using the triangle inequality, one has
		$$d(\zeta,\xi)
		\geq d(\zeta,\zeta_0)-d(\zeta_0,\xi)\geq(C -1)\delta,$$
		which implies that
		$$d(\zeta,\zeta_0)\leq d(\zeta,\xi)+d(\xi,\zeta_0)<d(\zeta,\xi)+\delta<(1+1/(C -1))d(\zeta,\xi).$$
		Note that for $z\in D_\alpha(\zeta)$ and $\xi\in\SN$,
		\begin{align*}
			d(\zeta,\xi)\leq&d(z,\zeta)+d(z,\xi)<\sqrt{\alpha(1-|z|^2)}+d(z,\xi)\\
			\leq&\sqrt{2\alpha(1-|z|)}+d(z,\xi)
			\leq(\sqrt{2\alpha}+1)d(z,\xi).
		\end{align*}
		Hence, $d(\zeta,\zeta_0)\lesssim d(z,\xi)$.
		
		We turn to the second inequality. 
		For any  $z\in D_\alpha(\zeta)$, we obtain
		\begin{align*}
			d(\zeta,\zeta_0)\leq&d(\zeta,z)+d(z,\zeta_0)<\sqrt{\alpha(1-|z|^2)}+d(z,\zeta_0)\\
			\leq&\sqrt{2\alpha(1-|z|)}+d(z,\zeta_0)
			\leq(\sqrt{2\alpha}+1)d(z,\zeta_0),
		\end{align*}
		which completes the proof.
	\end{proof}
	
	We are now in a position to show (\ref{reduce4}). We will show (\ref{reduce4}) for $T=S_\alpha^X\circ P$ and $T=S_\alpha^\mu\circ P$ respectively.
	\begin{proof}[Proof of (\ref{reduce4}) for $T=S_\alpha^X\circ P$.]
		By  Chebyshev’s inequality,	 
		$$\sigma(\{\zeta\in\SN\setminus\Omega: S_{\alpha}^X(P[b])(\zeta)>t\})\lesssim\frac{1}{t}\int_{\SN\setminus \Omega}S_{\alpha}^X(P[b])(\zeta)d\sigma.$$
		By Lemma \ref{smooth3}, (cz-ii), and (cz-iii), 
		$$\sum_{k}\int_{\SN}|b_k(\xi)||XP(z,\xi)|d\sigma(\xi)\lesssim\frac{1}{(1-|z|^2)^{n+1}}\sum_{k}\|b_k\|_{L^1}\lesssim\frac{1}{(1-|z|^2)^{n+1}}\|f\|_{L^1},$$
		which implies $XP[b](z)=\sum_{k}XP[b_k](z)$. Then by the triangle inequality, we have
		\begin{align*}
			\int_{\SN\setminus \Omega}S_{\alpha}^X(P[b])(\zeta)d\sigma
			\leq&\int_{\SN\setminus \Omega}\left(\int_{D_\alpha(\zeta)}\left(\sum_{k}|XP[b_k](z)|\right)^2\frac{dv(z)}{(1-|z|^2)^{n-1}}\right)^{1/2}d\sigma(\zeta)\\
			\leq &\sum_k\int_{\SN\setminus \Omega}S_{\alpha}^X(P[b_k])(\zeta)d\sigma\\
			\leq&\sum_k\int_{\SN\setminus B(\zeta_k,a_2r_k)}S_{\alpha}^X(P[b_k])(\zeta)d\sigma,
		\end{align*}
		where the last inequality follows from $B(\zeta_k,a_2r_k)\subset\Omega$ for each $k$. The proof will be completed if for each $\zeta\in\SN\setminus B(\zeta_k,a_2r_k)$,
		\begin{align}\label{PbClaim1}
			S_{\alpha}^X(P[b_k])(\zeta)\lesssim|r_kK_{(2n+1)/2}(\zeta,\zeta_k)|\cdot\|b_k\|_{L^1}.
		\end{align}
		Indeed, by (\ref{PbClaim1}), (cz-ii) and (cz-iii), we deduce that
		\begin{align*}
			\int_{\SN\setminus \Omega}S_{\alpha}^X(P[b])(\zeta)d\sigma
			\leq&\sum_{k}\sum_{j=1}^\infty\int_{B(\zeta_k,2^ja_2r_k)\setminus B(\zeta_k,2^{j-1}a_2r_k)}\frac{r_kd\sigma}{|1-\langle \zeta,\zeta_k\rangle|^{(2n+1)/2}}\|b_k\|_{L^1}\\
			\lesssim&\sum_k\sum_{j=1}^\infty\frac{r_k\sigma(B(\zeta_k,2^ja_2r_k))}{(2^{j-1}a_2r_k)^{2n+1}}\|b_k\|_{L^1}\\
			\lesssim&\sum_k\|b_k\|_{L^1}\lesssim\|f\|_{L^1}.
		\end{align*}

		It remains to show \eqref{PbClaim1}. Fix
		$\zeta\in\SN\setminus B(\zeta_k,a_2r_k)$. Define
		$$E_0(\zeta)= D_\alpha(\zeta)\cap\{z\in\BN: d(z,\zeta)<2d(\zeta,\zeta_k)\}$$
		and for each $j\in\mathbb{N}$
		$$E_j(\zeta)= D_\alpha(\zeta)\cap\{z\in\BN: 2^jd(\zeta,\zeta_k)\leq d(z,\zeta)<2^{j+1}d(\zeta,\zeta_k)\}.$$
		By the cancellation property  of $b_k$-(cz-ii), 
		\begin{align}\label{ga1}
			[S_{\alpha}^X(P[b_k])(\zeta)]^2\leq\sum_{j=0}^{\infty}\int_{E_j(\zeta)}F(z)\frac{dv(z)}{(1-|z|^2)^{n-1}},
		\end{align}
		where $F$ is defined by $$F(z)=\left(\int_{B(\zeta_k,r_k)}
		|XP(z,\xi)-XP(z,\zeta_k)||b_k(\xi)|d\sigma(\xi)\right)^2.$$
		
		We claim that
		\begin{align}\label{ga2}
			\int_{E_0(\zeta)}F(z)\frac{dv(z)}{(1-|z|^2)^{n-1}}\lesssim|r^2_kK_{2n+1}(\zeta,\zeta_k)|\cdot\|b_k\|^2_{L^1}. 
		\end{align}
		To see this, by Lemma \ref{smooth3}, 
		\begin{align}\label{ker2}
			|XP(z,\xi)-XP(z,\zeta_k)|\lesssim(1-|z|^2)^{n-1}I_1+(1-|z|^2)^n(I_2+I_3),
		\end{align}
		where $I_1,I_2,I_3$ are given by
		\begin{eqnarray*}
			I_1& = &|K_{2n}(z,\xi)-K_{2n}(z,\zeta_k)|,\\
			I_2& = &|1-\langle z,\xi\rangle|\cdot|K_{2n+2}(z,\xi)-K_{2n+2}(z,\zeta_k)|, \\
			I_3& = &|\langle z,\xi\rangle-\langle z,\zeta_k\rangle|\cdot|K_{2n+2}(z,\zeta_k)|.
		\end{eqnarray*}
		Then by (\ref{ker2}), we see that 
		\begin{align}\label{ker3}
			\int_{E_0}F(z)\frac{dv(z)}{(1-|z|^2)^{n-1}}\lesssim II_1+II_2,
		\end{align}
		where $II_1$ and $II_2$ are defined by
		\begin{eqnarray*}
			II_1& = &\int_{E_0(\zeta)}\left(\int_{B(\zeta_k,r_k)}I_1|b_k(\xi)|d\sigma(\xi)\right)^2(1-|z|^2)^{n-1}dv(z),\\
			II_2& = &\int_{E_0(\zeta)}\left(\int_{B(\zeta_k,r_k)}(I_2+I_3)|b_k(\xi)|d\sigma(\xi)\right)^2(1-|z|^2)^{n+1}dv(z).
		\end{eqnarray*}
		Hence, it is enough to show
		\begin{align}\label{ga3}
			\max\{II_1,II_2\}\lesssim|r^2_kK_{2n+1}(\zeta,\zeta_k)|\cdot\|b_k\|^2_{L^1}. 
		\end{align}
		
		For $\xi\in B(\zeta_k,r_k)$ and $z\in  D_\alpha(\zeta)$, since $d(\xi,\zeta_k)<r_k$ and $d(\zeta,\zeta_k)\geq a_2r_k$, it follows from Lemma \ref{smooth2} that
		$$d(\zeta,\zeta_k)\lesssim d(z,\xi) \ \ \text{and} \ \ d(\zeta,\zeta_k)\lesssim d(z,\zeta_k),$$
		and as a consequence, $$d(\xi,\zeta_k)<r_k\lesssim d(\zeta,\zeta_k)\lesssim d(z,\zeta_k).$$
		Combining with Lemma \ref{smooth1}, we deduce that
		\begin{align}\label{smooth4}
			\begin{split}
				I_1\leq&\sum_{i=0}^{2n-1}\frac{|\langle z,\xi\rangle-\langle z,\zeta_k\rangle|}{|1-\langle z,\xi\rangle|^{i+1}|1-\langle z,\zeta_k\rangle|^{2n-i}}\\
				\leq&\sum_{i=0}^{2n-1}\frac{2|1-\langle z,\zeta_k\rangle|^{1/2}|1-\langle \xi,\zeta_k\rangle|^{1/2}+|1-\langle \xi,\zeta_k\rangle|}{|1-\langle z,\xi\rangle|^{i+1}|1-\langle z,\zeta_k\rangle|^{2n-i}}\\
				\lesssim&\sum_{i=0}^{2n-1}\frac{|1-\langle z,\zeta_k\rangle|^{1/2}|1-\langle \xi,\zeta_k\rangle|^{1/2}}{|1-\langle z,\xi\rangle|^{i+1}|1-\langle z,\zeta_k\rangle|^{2n-i}}\\
				\lesssim&\frac{r_k}{|1-\langle\zeta,\zeta_k\rangle|^{2n+1/2}}.
			\end{split}
		\end{align}
		In a similar way, 
		\begin{align}\label{ga12}
			I_2\lesssim\frac{r_k}{|1-\langle\zeta,\zeta_k\rangle|^{2n+3/2}},
		\end{align}
		and for the term $I_3$,
		\begin{align}\label{smooth5}
			\begin{split}
				I_3=&|\langle z,\xi\rangle-\langle z,\zeta_k\rangle|\cdot|K_{2n+2}(z,\zeta_k)|\\
				\lesssim&\frac{|1-\langle z,\zeta_k\rangle|^{1/2}|1-\langle \xi,\zeta_k\rangle|^{1/2}}{|1-\langle z,\zeta_k\rangle|^{2n+2}}\\
				\lesssim&\frac{r_k}{|1-\langle\zeta,\zeta_k\rangle|^{2n+3/2}}.
			\end{split}
		\end{align}
		Now by (\ref{smooth4}), (\ref{ga12}), and (\ref{smooth5}), applying $v(E_0(\zeta))\lesssim|1-\langle \zeta,\zeta_k\rangle|^{n+1}$ and noting for $z\in E_0(\zeta)$ $$ 1-|z|^2\lesssim|1-\langle\zeta,\zeta_k\rangle|,$$
		we obtain
		\begin{align*}
			II_1\lesssim\frac{r_k^2v(E_0(\zeta))}{|1-\langle\zeta,\zeta_k\rangle|^{2(2n+1/2)-n+1}}\|b_k\|_{L^1}^2
			\lesssim\frac{r_k^2}{|1-\langle\zeta,\zeta_k\rangle|^{2n+1}}\|b_k\|_{L^1}^2
		\end{align*}
		and
		\begin{align*}
			II_2\lesssim\frac{r_k^2v(E_0(\zeta))}{|1-\langle\zeta,\zeta_k\rangle|^{2(2n+3/2)-n-1}}\|b_k\|_{L^1}^2
			\lesssim\frac{r_k^2}{|1-\langle\zeta,\zeta_k\rangle|^{2n+1}}\|b_k\|_{L^1}^2.
		\end{align*}
		This gives (\ref{ga3}) and thus (\ref{ga2}) is proved.
		
		We then deal with other terms on the right-side of (\ref{ga1}) and we also state that
		\begin{align}\label{ga4}
			\sum_{j=1}^{\infty}\int_{E_j(\zeta)}F(z)\frac{dv(z)}{(1-|z|^2)^{n-1}}\lesssim|r^2_kK_{2n+1}(\zeta,\zeta_k)|\cdot\|b_k\|^2_{L^1}.
		\end{align}
		Indeed, the argument for (\ref{ga4}) is quite similar as in (\ref{ga2}).  
		Observe that by Lemma \ref{smooth3}, 
		\begin{align*}
			|XP(z,\xi)-XP(z,\zeta_k)|\lesssim&(1-|z|^2)^{n-1}I_1+(1-|z|^2)^n(I_2+I_3)\\
			=&\left[(1-|z|^2)^{\frac{3n}{2}}I_1+(1-|z|^2)^{\frac{3n}{2}+1}(I_2+I_3)\right](1-|z|^2)^{-\frac{n+2}{2}}.
		\end{align*}
		On the one hand, by using the same argument as in \eqref{smooth4}, \eqref{ga12}, and \eqref{smooth5}, one can conclude from  Lemmas \ref{smooth1} and \ref{smooth2} that for $z\in  D_\alpha(\zeta)$,
		\begin{align}\label{1}
			(1-|z|^2)^{3n/2}I_1\lesssim\frac{r_k}{|1-\langle\zeta,\zeta_k\rangle|^{(n+1)/2}}
		\end{align}
		and
		\begin{align}\label{2}
			(1-|z|^2)^{3n/2+1}(I_2+I_3)\lesssim\frac{r_k}{|1-\langle\zeta,\zeta_k\rangle|^{(n+1)/2}}.
		\end{align}
		Hence, it follows from \eqref{1} and \eqref{2} that for each $j\in\mathbb{N}$ and $z\in D_\alpha(\zeta)$,
		\begin{align}\label{4}
			|XP(z,\xi)-XP(z,\zeta_k)|\lesssim(1-|z|^2)^{-\frac{n+2}{2}}\frac{r_k}{|1-\langle\zeta,\zeta_k\rangle|^{(n+1)/2}}.
		\end{align}
		On the other hand, since $1-|z|^2\leq2|1-\langle z,\zeta\rangle|\leq 2\alpha(1-|z|^2)$ for $z\in D_\alpha(\zeta)$, we have
		\begin{align}\label{3}
			2^{2j}|1-\langle\zeta,\zeta_k\rangle| \asymp1-|z|^2
		\end{align} 
		for $z\in E_j(\zeta)$ with each $j\in\mathbb{N}$. Plug \eqref{4} into the left-side of (\ref{ga4}), and use \eqref{3} and $v(E_j(\zeta))\lesssim(2^{2j})^{n+1}|1-\langle \zeta,\zeta_k\rangle|^{n+1}$ for each $j\geq1$, then we find
		\begin{align*}
			\sum_{j=1}^{\infty}\int_{E_j(\zeta)}F(z)\frac{dv(z)}{(1-|z|^2)^{n-1}}\lesssim&\sum_{j=1}^{\infty}\frac{r_k^2 v(E_j(\zeta))}{2^{2j(2n+1)}|1-\langle\zeta,\zeta_k\rangle|^{3n+2}}\|b_k\|_{L^1}^2\\
			\lesssim&|r_k^2K_{2n+1}(\zeta,\zeta_k)|\|b_k\|_{L^1}^2,
		\end{align*}
		which is exactly (\ref{ga4}). Finally, by (\ref{ga1}), (\ref{ga2}), and (\ref{ga4}), we complete the proof of \eqref{PbClaim1}.
	\end{proof}
	
	Now we come to the proof of (\ref{reduce4}) for $T=S_{\alpha}^\mu\circ P$.
	\begin{proof}[Proof of (\ref{reduce4}) for $T=S_{\alpha}^\mu\circ P$.]
		Observe that
		\begin{align*}
			\int_{\SN\setminus \Omega}S_{\alpha}^\mu(P[b])d\sigma
			\leq&\int_{\SN\setminus \Omega}\left(\int_{D_\alpha(\zeta)}\left(\sum_{k}|P[b_k](z)|\right)^2\frac{d\mu(z)}{(1-|z|^2)^{n}}\right)^{1/2}d\sigma(\zeta)\\
			\leq&\sum_k\int_{\SN\setminus B(\zeta_k,a_2r_k)}S_{\alpha}^\mu(P[b_k])d\sigma.
		\end{align*}
		Hence, by using the same argument as in the case $T=S_\alpha^X\circ P$, we also show
		\begin{align}\label{ACest4}
			S^{\mu}_{\alpha}(P[b_k])(\zeta)\lesssim|r_kK_{(2n+1)/2}(\zeta,\zeta_k)|\cdot\|b_k\|_{L^1},
		\end{align}
		for $\zeta\in\SN\setminus B(\zeta_k,a_2r_k)$. To see this,  by definition
		$$[S_{\alpha}^{\mu}(P[b_k])(\zeta)]^2= \int_{D_\alpha(\zeta)}\left(\int_{\SN}(1-|z|^2)^n|K_{2n}(z,\xi)|b_k(\xi)\sigma(\xi)\right)^2\frac{d\mu(z)}{(1-|z|^2)^n},$$
		it then follows from (cz-ii) that
		\begin{align}\label{ACest5}
			S^{\mu}_{\alpha}(P[b_k])(\zeta)\lesssim III_1+III_2.
		\end{align}
		Here $III_1$ and $III_2$ are given by
		\begin{eqnarray*}
			III_1& = &\int_{E_0(\zeta)}\left(\int_{B(\zeta_k,r_k)}I_1|b_k(\xi)|d\sigma(\xi)\right)^2(1-|z|^2)^nd\mu(z),\\
			III_2& = &\sum_{j=1}^\infty\int_{E_j(\zeta)}\left(\int_{B(\zeta_k,r_k)}(1-|z|)^{3n/2}I_1|b_k(\xi)|d\sigma(\xi)\right)^2 \frac{d\mu(z)}{(1-|z|^2)^{2n}},
		\end{eqnarray*}
		where $I_1$, $E_0$ and $E_j$ for $j\geq1$ are the same as before.
		
		Firstly, consider the term $III_1$. By the definition of Carleson measure, $\mu(E_0(\zeta))\lesssim|1-\langle\zeta,\zeta_k\rangle|^n$.  Then by the fact that $1-|z|^2\lesssim|1-\langle\zeta,\zeta_k\rangle|$ for $z\in E_0(\zeta)$ and using (\ref{smooth4}), we get
		\begin{align}\label{ACest6}
			III_1\lesssim\|b_k\|_{L^1}^2\frac{r_k^2\mu(E_0(\zeta))}{|1-\langle\zeta,\zeta_k\rangle|^{2(2n+1/2)-n}}
			\lesssim|r_k^2K_{2n+1}(\zeta,\zeta_k)|\cdot\|b_k\|_{L^1}^2.
		\end{align}
		
		For the term $III_2$, since $2^{2j}|1-\langle\zeta,\zeta_k\rangle|\asymp 1-|z|^2$ when $z\in E_j(\zeta)$ and for each $j\geq1$, $\mu(E_j(\zeta))\lesssim2^{2nj}|1-\langle\zeta,\zeta_k\rangle|^n$, one has
		\begin{align}\label{ACest7}
			III_2\lesssim\|b\|_{L^1}^2
			\sum_{j=1}^\infty\frac{r_k^2\mu(E_j(\zeta))}{(2^{2j})^{2n}|1-\langle\zeta,\zeta_k\rangle|^{3n+1}}
			\lesssim|r_k^2K_{2n+1}(\zeta,\zeta_k)|\cdot\|b_k\|_{L^1}^2,
		\end{align}
		where in the first inequality we also use \eqref{1}. 
		Finally, by (\ref{ACest6}), (\ref{ACest7}) and (\ref{ACest5}), we obtain \eqref{ACest4}. This completes the proof.
	\end{proof}
	
	Finally, we finish the proof of Lemma \ref{Weak}. Next, we will apply Lemma \ref{Weak} to show Theorem \ref{Weak7} with the help of Lemma \ref{ACest9}.
	
	\subsection{Some auxiliary facts}
	This subsection is devoted to showing Lemma \ref{ACest9}, which permits us to complete the proof of Theorem \ref{Weak7}. We proceed with some useful propositions.
	
	Let  $u$ be a measurable function on $\BN$ and $\nu$ be a positive Borel measure on $\BN$. For $\alpha>0$ and $\zeta\in\SN$, define 
	\begin{align}\label{ACest8}
		T_\alpha u(\zeta)=\left(\int_{\Gamma_\alpha(\zeta)}|u|^2d\nu\right)^{1/2},
	\end{align}
	where
	$$\Gamma_\alpha(\zeta)=\{z=r\eta:|1-\langle \eta,\zeta\rangle|<\alpha(1-r), \eta\in\SN,0\leq r<1\}.$$
	\begin{proposition}\label{T}
		For any $0<\beta\leq1$ and $\alpha>\beta$,
		$$\|T_\alpha u\|_{L^{1,\infty}}\lesssim\frac{\alpha^n}{\beta^n}\|T_\beta u\|_{L^{1,\infty}}.$$
	\end{proposition}
	
	To prove Proposition \ref{T}, we need a couple of lemmas. Let $0<\beta\leq1$ and $\alpha>\beta$. By \eqref{eq2} and \eqref{eq5}, there exist two constants $C,C'>0$ such that
	\begin{align}\label{eq4}
		\begin{split}
			\sigma(B(\zeta,\sqrt{\alpha(1-r)}))\leq&C\alpha^{n}\sigma(B(\zeta,\sqrt{1-r}))\\
			\leq&CC'(\alpha/\beta)^{n}\sigma(B(\eta,\sqrt{\beta(1-r)})),
		\end{split}
	\end{align}
	for any $\zeta,\eta\in \SN$ and $0\leq r<1$. Denote by $c=CC'$. For $t>0$, define
	$$U_t=\{\zeta\in\SN: M(\chi_{O_t})(\zeta)>(2^{2n+1}c(\alpha/\beta)^{n})^{-1}\},$$
	where $O_t=\{\zeta\in\SN:T_\beta u(\zeta)>t\}$.

	\begin{lemma}\label{T1}
		Let $0<\beta\leq1$ and $\alpha>\beta$. Let $t>0$.
		\begin{itemize}
			\item[(i)] If $z=r\eta\in \cup_{\zeta\in\SN\setminus U_t}\Gamma_\alpha(\zeta)$, then
			$$\sigma(B(\eta,\sqrt{\beta (1-r)}))\leq 2\sigma(B(\eta,\sqrt{\beta (1-r)})\cap(\SN\setminus O_t));$$
			\item[(ii)] The following statement holds:
			$$\bigcup_{\zeta\in\SN\setminus U_t}\Gamma_\alpha(\zeta)\subset\bigcup_{\zeta\in\SN\setminus O_t}\Gamma_\beta(\zeta).$$
		\end{itemize}
	\end{lemma}
	
	\begin{proof}
		(i) By definition, the assumption implies that there is $\zeta\in \SN\setminus U_t$ such that
		$$|1-\langle \eta,\zeta\rangle|<\alpha(1-r),$$
		which implies $\zeta\in B(\eta,\sqrt{\alpha(1-r)})$. Then, we have
		$$B(\eta,\sqrt{\alpha(1-r)})\subset B(\zeta,2\sqrt{\alpha(1-r)}).$$
		It then follows from \eqref{eq4} that
		\begin{align*}
			\frac{\sigma(B(\eta,\sqrt{\beta(1-r)})\cap O_t)}{B(\eta,\sqrt{\beta(1-r)})}
			\leq&\frac{c2^{2n}(\alpha/\beta)^{n}\sigma(B(\zeta,2\sqrt{\alpha(1-r)})\cap O_t)}{B(\zeta,2\sqrt{\alpha(1-r)})}\\
			\leq& c2^{2n}\frac{\alpha^n}{\beta^n}M(\chi_{O_t})(\zeta)\leq\frac{1}{2}.
		\end{align*}
		This gives (i).
		
		(ii) If $\bigcup_{\zeta\in\SN\setminus U_t}\Gamma_\alpha(\zeta)=\emptyset$, we are done. Otherwise, there exists $z=r\eta\in\bigcup_{\zeta\in\SN\setminus U_t}\Gamma_\alpha(\zeta)$, and by (i) we know that
		$$B(\eta,\sqrt{\beta (1-r)})\cap(\SN\setminus O_t)\neq\emptyset.$$
		Let $\zeta'\in [B(\eta,\sqrt{\beta (1-r)})\cap(\SN\setminus O_t)]\subset \SN\setminus O_t$. Then
		$$|1-\langle \eta,\zeta'\rangle|<\beta (1-r),$$
		which implies $z\in\Gamma_\beta(\zeta')$, as desired.
	\end{proof}

	\begin{lemma}\label{T2}
		Let $0<\beta\leq1$ and $\alpha>\beta$. Let $t>0$. Then  for any measurable function $u$ on $\BN$, one has
		$$\int_{\SN\setminus U_t}(T_\alpha u)^2d\sigma\lesssim\frac{\alpha^n}{\beta^n}\int_{\SN\setminus O_t}(T_\beta u)^2d\sigma.$$
	\end{lemma}
	
	\begin{proof}
		Let $E_\alpha(z)=\{\zeta\in\SN: z\in\Gamma_\alpha (\zeta)\}$ and $E_\beta(z)=\{\zeta\in\SN: z\in\Gamma_\beta (\zeta)\}$. Observe that
		$$\int_{\SN\setminus U_t}(T_\alpha u)^2d\sigma
		=\int_{\bigcup_{\zeta\in\SN\setminus U_t} \Gamma_\alpha(\zeta)}|u(z)|^2\sigma(E_\alpha(z)\cap(\SN\setminus U_t))d\nu(z)$$
		and
		$$\int_{\SN\setminus O_t}(T_\beta u)^2d\sigma
		=\int_{\bigcup_{\zeta\in\SN\setminus O_t} \Gamma_\beta(\zeta)}|u(z)|^2\sigma(E_\beta(z)\cap(\SN\setminus O_t))d\nu(z).$$
		Moreover, if we write
		$z=r\eta$ with $\eta\in\SN$ and $0\leq r<1$, then
		$E_\alpha(z)=B(\eta,\sqrt{\alpha(1-r)})$ and $E_\beta(z)=B(\eta,\sqrt{\beta(1-r)})$. On the other hand, by Lemma \ref{T1} (ii), 
		$$\bigcup_{\zeta\in\SN\setminus U_t}\Gamma_\alpha(\zeta)\subset\bigcup_{\zeta\in\SN\setminus O_t} \Gamma_\beta(\zeta)$$
		and by \eqref{eq4} and Lemma \ref{T1} (i),
		\begin{align*}
			\sigma(B(\eta,\sqrt{\alpha(1-r)})\cap(\SN\setminus U_t))
			\leq&\sigma(B(\eta,\sqrt{\alpha(1-r)}))\\
			\lesssim&(\alpha/\beta)^{n}\sigma(B(\eta,\sqrt{\beta(1-r)}))\\
			\lesssim&(\alpha/\beta)^{n}\sigma(B(\eta,\sqrt{\beta(1-r)})\cap(\SN\setminus O_t)).
		\end{align*}
		This yields the desired result.
	\end{proof}
	Now we are ready to prove Proposition \ref{T}.
	
	\begin{proof}[Proof of Proposition \ref{T}.]
		Fix $t>0$. Since $M$ is of weak type $(1,1)$, we have
		$$\sigma(U_t)\lesssim\frac{\alpha^{n}}{\beta^n}\sigma(O_t),$$
		which implies
		$$\sigma(\{\zeta\in\SN:T_\alpha u(\zeta)>t\})\lesssim\frac{\alpha^{n}}{\beta^n}\sigma(O_t)+\sigma(\{\zeta\in\SN\setminus U_t:T_\alpha u(\zeta)>t\}).$$
		Hence, it is enough to show
		\begin{align}\label{T3}
			\sigma(\{\zeta\in\SN\setminus U_t:T_\alpha u(\zeta)>t\})\lesssim\frac{\alpha^{n}}{\beta^nt}\|T_\beta u\|_{L^{1,\infty}}.
		\end{align} 
		By Chebyshev’s inequality and Lemma \ref{T2},
		\begin{align*}
			\sigma(\{\zeta\in\SN\setminus U_t:T_\alpha u(\zeta)>t\})
			\leq&\frac{1}{t^2}\int_{\SN\setminus U_t}(T_\alpha u)^2(\zeta)d\sigma(\zeta)\\
			\lesssim&\frac{\alpha^{n}}{\beta^n t^2}\int_{\SN\setminus O_t}(T_\beta u)^2(\zeta)d\sigma(\zeta)\\
			\leq&2\frac{\alpha^{n}}{\beta^n t^2}\int_{0}^{t}s\sigma(\{\zeta\in\SN:T_\beta u(\zeta)>s\})ds,
		\end{align*}
		where in the last inequality we use Fubini's theorem.
		
		By the definition of the $L^{1,\infty}$-norm, we finally get
		$$\frac{\alpha^{n}}{\beta^n t^2}\int_{0}^{t}s\sigma(\{\zeta\in\SN:T_\beta u(\zeta)>s\})ds\lesssim\frac{\alpha^{n}}{\beta^nt}\|T_\beta u\|_{L^{1,\infty}},$$
		which yields \eqref{T3}.
	\end{proof}

	As an application of Proposition \ref{T}, we can show Lemma \ref{ACest9}, which concerns the following two special cases of $T_\alpha$ defined in (\ref{ACest8}). More precisely, for $\alpha>0$, define for $\zeta\in\SN$
	$$T_{\alpha}^X(u)(\zeta)=\left(\int_{\Gamma_\alpha(\zeta)}
	\left|Xu(z)\right|^2\frac{dv(z)}{(1-|z|^2)^{n-1}}\right)^{1/2} \quad \text{if} \quad u\in \mathcal{C}^1(\BN),$$
	and
	$$T_{\alpha}^\mu(u)(\zeta)=\left(\int_{\Gamma_\alpha(\zeta)}
	\left|u(z)\right|^2\frac{d\mu(z)}{(1-|z|^2)^{n}}\right)^{1/2} \quad \text{if} \quad u \ \text{is measurable on} \ \BN.$$

	\begin{lemma}\label{ACest9}
		For any measurable function $u$ on $\BN$, one has
		\begin{itemize}
			\item[(i)] $\|T_{\alpha}^\mu(u)\|_{L^{1,\infty}}\lesssim\alpha^n\|S_4^\mu(u)\|_{L^{1,\infty}}$, where $\alpha\geq1$;
			\item[(ii)] $S_\alpha^\mu(u)(\zeta)\leq T_{(1+\sqrt{2\alpha})^2}^\mu(u)(\zeta)$, where $\alpha>1/2$.
		\end{itemize}
		For any $u\in \mathcal{C}^1(\BN)$, one has
		\begin{itemize}
			\item[(iii)] $\|T_{\alpha}^X(u)\|_{L^{1,\infty}}\lesssim\alpha^n\|S_4^X(u)\|_{L^{1,\infty}}$, where $\alpha\geq1$;
			\item[(iv)] $S_\alpha^X(u)(\zeta)\leq T_{(1+\sqrt{2\alpha})^2}^X(u)(\zeta)$, where $\alpha>1/2$.
		\end{itemize}
	\end{lemma}
	\begin{proof}
		At first, we prove (i) and (iii). Let $\alpha\geq1$. For $z\in\BN$, write $z=r\eta$ with $\eta\in\SN$ and $0\leq r<1$. If $z=r\eta\in\Gamma_1(\zeta)$, then by the triangle inequality
		$$d(r\eta,\zeta)\leq d(r\eta,\eta)+d(\eta,\zeta)<\sqrt{1-r}+\sqrt{1-r}\leq2\sqrt{1-|z|^2},$$
		which implies
		$\Gamma_{1}(\zeta)\subset D_{4}(\zeta).$
		Combining with Proposition \ref{T}, we find
		$$\|T_{\alpha}^X(u)\|_{L^{1,\infty}}\lesssim\alpha^n\|T_1^X(u)\|_{L^{1,\infty}}\leq\alpha^n\|S_4^X(u)\|_{L^{1,\infty}}$$
		and
		$$\|T_{\alpha}^\mu(u)\|_{L^{1,\infty}}\lesssim\alpha^n\|T_1^\mu(u)\|_{L^{1,\infty}}\leq\alpha^n\|S_4^\mu(u)\|_{L^{1,\infty}}.$$
		
		Then we show (ii) and (iv). Let $\alpha>1/2$. For $z=r\eta\in D_\alpha(\zeta)$, it follows from the triangle inequality that
		$$d(\eta,\zeta)\leq d(\eta,z)+d(z,\zeta)<\sqrt{1-r}+\sqrt{\alpha(1-r^2)}\leq(1+\sqrt{2\alpha})\sqrt{1-r},$$
		which implies $D_\alpha(\zeta)\subset\Gamma_{(1+\sqrt{2\alpha})^2}(\zeta)$. Therefore, for any $\zeta\in\SN$
		$$S_\alpha^X(u)(\zeta)\leq T_{(1+\sqrt{2\alpha})^2}^X(u)(\zeta)\quad\text{and}\quad S_\alpha^\mu(u)(\zeta)\leq T_{(1+\sqrt{2\alpha})^2}^\mu(u)(\zeta).$$
		This completes the proof. 
	\end{proof}

	With Lemma \ref{Weak} and Lemma \ref{ACest9} in hand, we prove Theorem \ref{Weak7}.
	\begin{proof}[Proof of Theorem \ref{Weak7}]
		Let $A\in\{X,\mu\}$. Observe that $$\BN=\bigcup_{k=1}^\infty( D_{4^{k+1}}(\zeta)\setminus D_{4^k}(\zeta))\cup D_{4}(\zeta)$$
		for any $\zeta\in\SN$. Since $4^k(1-|z|^2)\leq|1-\langle z,\zeta\rangle|$ when $z\in D_{4^{k+1}}(\zeta)\setminus D_{4^k}(\zeta)$ and $1-|z|^2\lesssim|1-\langle z,\zeta\rangle|$, one gets
		$$\g^*_{A,\lambda}(P[f])(\zeta)^2\lesssim\sum_{k=1}^{\infty}\frac{1}{4^{\lambda n(k-1)}}S^A_{4^k}(P[f])(\zeta)^2.$$
		Therefore, we have
		$$\g^*_{A,\lambda}(P[f])(\zeta)\lesssim\left(\sum_{k=1}^{\infty}\frac{1}{4^{\lambda n(k-1)}}S^A_{4^k}(P[f])(\zeta)^2\right)^{1/2}\leq\sum_{k=1}^{\infty}\frac{1}{4^{\lambda n(k-1)/2}}S^A_{4^k}(P[f])(\zeta).$$
		By Lemma \ref{ACest9} (ii) and (iv), 
		\begin{align}\label{T31}
			\g^*_{A,\lambda}(P[f])(\zeta)\lesssim\sum_{k=1}^{\infty}\frac{1}{4^{\lambda n(k-1)/2}}T^A_{4^{k+2}}(P[f])(\zeta).
		\end{align} 
		On the other hand, using the quasi-triangle inequality and Lemma \ref{ACest9} (i) and (iii), we see that for every $N\in\mathbb{N}$,
		\begin{align*}
			\left\|\sum_{k=1}^{N}\frac{1}{4^{\lambda n(k-1)/2}}T^A_{4^{k+2}}(P[f])\right\|_{L^{1,\infty}}
			\leq&\sum_{k=1}^{N}\frac{2^k}{4^{\lambda n(k-1)/2}}\|T^A_{4^{k+2}}(P[f])\|_{L^{1,\infty}}\\
			\lesssim&\sum_{k=1}^{\infty}\frac{2^k4^{nk}}{4^{\lambda n(k-1)/2}}\|S^A_{4}(P[f])\|_{L^{1,\infty}}\\
			\lesssim&\|S^A_{4}(P[f])\|_{L^{1,\infty}},
		\end{align*}
		where the last inequality holds since $\lambda\geq4$. For $t>0$, since 
		\begin{align*}
			&\sigma\left(\left\{\zeta\in\SN:\sum_{k=1}^{\infty}\frac{1}{4^{\lambda n(k-1)/2}}T^A_{4^{k+2}}(P[f])(\zeta)>t\right\}\right)\\
			=&\lim_{N\to\infty}\sigma\left(\left\{\zeta\in\SN:\sum_{k=1}^{N}\frac{1}{4^{\lambda n(k-1)/2}}T^A_{4^{k+2}}(P[f])(\zeta)>t\right\}\right),
		\end{align*}
		we have
		$$t\sigma\left(\left\{\zeta\in\SN:\sum_{k=1}^{\infty}\frac{1}{4^{\lambda n(k-1)/2}}T^A_{4^{k+2}}(P[f])(\zeta)>t\right\}\right)\lesssim\|S^A_{4}(P[f])\|_{L^{1,\infty}}.$$
		This implies by Lemma \ref{Weak} that
		$$\|\g^*_{A,\lambda}(P[f])\|_{L^{1,\infty}}\lesssim\|S^A_{4}(P[f])\|_{L^{1,\infty}}\lesssim\|f\|_{L^1}.$$
		The proof is completed.
	\end{proof}
	
	\bigskip

	\section{Some more estimates for $\g^*_\lambda$}\label{keyestimate}
	This section is devoted to some preparatory work for $\g^*_\lambda$  in order to finish the proofs of Lemma \ref{Rd1} and Lemma \ref{Rd3} in the next section. In the sequel, we fix a dyadic system $\mathscr{D}$ on $\SN$ with the parameters $C_0\geq c_0>0$, $r\in(0,1)$ and $12 C_0r\leq c_0$ (see Theorem \ref{de}). Let $c_1=c_0/3$ and $C_1=2C_0$. Fix $m_0\in\mathbb{Z}$ such that $\SN=Q_1^{m_0}\in\mathscr{D}$. Denote by $Q^0:=Q_1^{m_0}$.
	
	Given $Q\in\mathscr{D}(Q^0)$ with ``center" $\zeta_Q$ and ``radius" $r_Q$, let 
	$$E_{Q}:=\{z\in\BN:d(\zeta_Q,z)<2C_1r_Q\}$$
	and $E^c_{Q}=\BN\setminus E_Q$ be the complement of $E_Q$.
	For $u\in\mathcal{C}^1(\BN)$, define
	$$\g^*_{E_Q,X,\lambda}(u)(\zeta)=\left(\int_{E_Q}|Xu(z)|^2\left(\frac{1-|z|^2}{|1-\langle z,\zeta\rangle|}\right)^{\lambda n}\frac{dv(z)}{(1-|z|^2)^{n-1}}\right)^{1/2},\quad \forall  \zeta\in\SN$$
	and
	$$\g^*_{E^c_Q,X,\lambda}(u)(\zeta)=\left(\int_{E^c_Q}|Xu(z)|^2\left(\frac{1-|z|^2}{|1-\langle z,\zeta\rangle|}\right)^{\lambda n}\frac{dv(z)}{(1-|z|^2)^{n-1}}\right)^{1/2},\quad \forall  \zeta\in\SN.$$
	Accordingly, $\g^*_{E_Q,\mu,\lambda}$ and $\g^*_{E^c_Q,\mu,\lambda}$ are defined in the same way, that is $\g^*_{E_Q,\mu,\lambda}:=\g^*_{\chi_{E_Q}d\mu,\lambda}$ and $\g^*_{E^c_Q,\mu,\lambda}:=\g^*_{\chi_{E^c_Q}d\mu,\lambda}$. In the sequel, we use the following notation $$\g^*_{E_Q,\lambda}\in\{\g^*_{E_Q,X,\lambda},\g^*_{E_Q,\mu,\lambda}\} \ \ \text{and} \ \ \g^*_{E^c_Q,\lambda}\in\{\g^*_{E^c_Q,X,\lambda},\g^*_{E^c_Q,\mu,\lambda}\}.$$
	In the following, we will study some key estimates for $\g^*_{E_Q,\lambda}$ and $\g^*_{E^c_Q,\lambda}$ that will be crucial in the next section. Recall that in Definition \ref{def2.8}, for every $\delta>1$, the ``dilated" set of $Q$ is given by
	$$\delta Q=B(\zeta_Q,\delta C_1 r_Q). $$
	At first, we present a basic fact, which will be frequently used in this section. It follows from \eqref{eq2} and \eqref{eq5} that for any $\delta\geq1$,
	\begin{align}\label{Poisson2}
		\sigma(\delta Q)\lesssim\delta^{2n}\sigma(B(\zeta_Q,r_Q))\lesssim(\delta r_Q)^{2n}.
	\end{align}

	\begin{lemma}\label{P2}
		Let $\lambda\in\mathbb{N}$ with $\lambda\geq4$. Then for any $f\in L^1$,
		$$\|(\g^*_{E_Q,\lambda}(P[f\chi_{\SN\setminus8Q}]))^2\|_{L^1}\lesssim\sigma(Q)\sum_{j=1}^{\infty}\frac{1}{2^{2nj}}\left(\frac{1}{\sigma(2^j Q)}\int_{2^j Q}|f|d\sigma\right)^2.$$
	\end{lemma}

	\begin{proof}
		We show this lemma for $\g^*_{E_Q,X,\lambda},\g^*_{E_Q,\mu,\lambda}$ respectively. We proceed with the proof of $\g^*_{E_Q,X,\lambda}$. Firstly, assume $\g^*_{E_Q,\lambda}=\g^*_{E_Q,X,\lambda}$. For $z\in E_Q$, define
		$$I(z)=(1-|z|^2)^{1/2}XP[f\chi_{\SN\setminus8Q}](z).$$
		Then applying Fubini's theorem, we have
		\begin{align*}
			&\|(\g^*_{E_Q,X,\lambda}(P[f\chi_{\SN\setminus8Q}]))^2\|_{L^1}\\
			=&\int_{\SN}\int_{E_Q}|XP[f\chi_{\SN\setminus8Q}](z)|^2\left(\frac{1-|z|^2}{|1-\langle z,\xi\rangle|}\right)^{\lambda n}\frac{dv(z)}{(1-|z|^2)^{n-1}}d\sigma(\xi)\\
			=&\int_{\SN}\int_{E_Q}I(z)^2\frac{(1-|z|^2)^{\lambda n-n}}{|1-\langle z,\xi\rangle|^{\lambda n}}dv(z)d\sigma(\xi)\\
			=&\int_{E_Q}\int_{\SN}\frac{(1-|z|^2)^{\lambda n-n}}{|1-\langle z,\xi\rangle|^{\lambda n}}d\sigma(\xi)I(z)^2dv(z)\\
			\lesssim&\int_{E_Q}I(z)^2dv(z),
		\end{align*}
		where the last inequality follows from the fact that for $\lambda\geq4$, 
		\begin{align}\label{Poisson1}
			\int_{\SN}\frac{(1-|z|^2)^{\lambda n-n}}{|1-\langle z,\xi\rangle|^{\lambda n}}d\sigma(\xi)\lesssim P[1](z)=1.
		\end{align}
		Furthermore, by Lemma \ref{smooth3}, one has  for each $z\in E_Q$, 
		$$|I(z)|
		\lesssim\sum_{j=1}^{\infty}\int_{2^{j+3}Q \setminus2^{j+2}Q}\frac{(1-|z|^2)^{n-1/2}}{|1-\langle z,\xi\rangle|^{2n}}|f(\xi)|d\sigma(\xi).$$
		Observe that  $1-|z|^2\leq2d(\zeta_Q,z)^2\lesssim r_Q^2$ and $d(\xi,z)\geq d(\xi,\zeta_Q)-d(\zeta_Q,z)\gtrsim2^{j}r_Q$ whenever $z\in E_Q$ and $\xi\in2^{j+3}Q \setminus2^{j+2}Q$ with $j\in\mathbb{N}$. Consequently, by \eqref{Poisson2}, we obtain
		\begin{align}\label{Poisson3}
			|I(z)|\lesssim\sum_{j=1}^{\infty}\frac{r_Q^{2(n-1/2)}}{(2^jr_Q)^{4n}}\int_{2^{j+3}Q}|f|d\sigma
			\lesssim\sum_{j=1}^{\infty}\frac{1}{2^{2nj}r_Q\sigma(2^j Q)}\int_{2^jQ}|f|d\sigma.
		\end{align}
		Finally, by (\ref{Poisson3}), H\"older’s inequality and $v(E_Q)\lesssim r^2_Q\sigma(Q)$, we obtain
		\begin{align*}
			\int_{E_Q}I(z)^2dv(z)
			\lesssim&\int_{E_Q}\left(\sum_{j=1}^{\infty}\frac{1}{2^{2nj}r_Q\sigma(2^j Q)}\int_{2^jQ}|f|d\sigma\right)^2dv(z)\\
			\lesssim&\int_{E_Q}\sum_{j=1}^{\infty}\frac{1}{2^{2nj}r_Q^2}\left(\frac{1}{\sigma(2^jQ)}\int_{2^jQ}|f|d\sigma\right)^2dv(z)\\
			\lesssim&\sigma(Q)\sum_{j=1}^{\infty}\frac{1}{2^{2nj}}\left(\frac{1}{\sigma(2^jQ)}\int_{2^jQ }|f|d\sigma\right)^2.
		\end{align*} 
		This implies the desired result for $\g^*_{E_Q,\lambda}=\g^*_{E_Q,X,\lambda}$. 
		
		We come to the proof of $\g^*_{E_Q,\lambda}=\g^*_{E_Q,\mu,\lambda}$.
		The argument is quite similar as in the previous case. Let us sketch the proof. Using \eqref{Poisson1} and Fubini’s theorem, we have
		$$\|(\g^*_{E_Q,\mu,\lambda}(P[f\chi_{\SN\setminus8Q}]))^2\|_{L^1}\lesssim\int_{E_Q}\big|P[f\chi_{\SN\setminus8Q}](z))\big|^2d\mu(z).$$
		Note that for each $z\in E_Q$, 
		$$|P[f\chi_{\SN\setminus8Q}](z))|
		\lesssim\sum_{j=1}^{\infty}\int_{2^{j+3}Q \setminus2^{j+2}Q}\frac{(1-|z|^2)^n}{|1-\langle z,\xi\rangle|^{2n}}|f(\xi)|d\sigma(\xi).$$
		Then using the same arguemnt as in (\ref{Poisson3}), we get
		$$|P[f\chi_{\SN\setminus8Q}](z))|\lesssim\sum_{j=1}^{\infty}\frac{r_Q^{2n}}{(2^jr_Q)^{4n}}\int_{2^{j}Q}|f|d\sigma
		\lesssim\sum_{j=1}^{\infty}\frac{1}{2^{2nj}\sigma(2^j Q)}\int_{2^{j+3}Q}|f|d\sigma.$$
		On the other hand, by the definition of Carleson measure, one has $\mu(E_Q)\lesssim\sigma(Q)$. Therefore, combining these observations and applying H\"older’s inequality, we deduce that
		$$\int_{E_Q}\big|P[f\chi_{\SN\setminus8Q}](z))\big|^2d\mu(z)\lesssim\sigma(Q)\sum_{j=1}^{\infty}\frac{1}{2^{2nj}}\left(\frac{1}{\sigma(2^jQ)}\int_{2^jQ }|f|d\sigma\right)^2.$$
		This finishes the proof.
	\end{proof}

	Before dealing with $\g^*_{E^c_{Q},\lambda}$, we need the following lemma. For each $k\in\mathbb{N}$, define
	\begin{align}\label{Poisson6}
		F_k=\{z\in\BN:2^kC_1r_Q\leq d(\zeta_Q,z)<2^{k+1}C_1r_Q\}.
	\end{align}
	\begin{lemma}\label{Poisson7}
		Let $k\in\mathbb{N}$. Then for each $s\in\{1,2,\cdots,\lambda-1\}$, $l\in\{-1,-1/4,0,1\}$, $\zeta\in Q $ and $z\in F_k$, we have
		$$(1-|z|^2)^{\lambda n-sn-l}|K_{\lambda n}(z,\zeta_Q)-K_{\lambda n}(z,\zeta)|\lesssim\frac{r_Q}{(2^kr_Q)^{2(sn+l)+1}}.$$
	\end{lemma}
	\begin{proof}
		Let $s\in\{1,2,\cdots,\lambda-1\}$, $l\in\{-1,-1/4,0,1\}$, $\zeta\in Q $ and $z\in F_k$.
		Observe that 
		$$d(z,\zeta)\geq d(z,\zeta_Q)-d(\zeta_Q,\zeta )\gtrsim2^kr_Q\quad \mbox{and}\quad d(\zeta,\zeta_Q)\lesssim r_Q\lesssim d(\zeta ,z).$$
		Hence, by Lemma \ref{smooth1}, we have
		\begin{align*}
			\nonumber &(1-|z|^2)^{\lambda n-sn-l}|K_{\lambda n}(z,\zeta_Q)-K_{\lambda n}(z,\zeta)|\\
			\leq&(1-|z|^2)^{\lambda n-sn-l}\sum_{k=0}^{\lambda n-1}\frac{2|1-\langle z,\zeta\rangle|^{1/2}|1-\langle \zeta_Q,\zeta\rangle|^{1/2}+|1-\langle \zeta_Q,\zeta\rangle|}{|1-\langle z,\zeta\rangle|^{k+1}|1-\langle z,\zeta_Q\rangle|^{\lambda n-k}}\\
			\nonumber\lesssim&(1-|z|^2)^{\lambda n-sn-l}\sum_{k=0}^{\lambda n-1}\frac{|1-\langle z,\zeta\rangle|^{1/2}r_Q}{|1-\langle z,\zeta\rangle|^{k+1}|1-\langle z,\zeta_Q\rangle|^{\lambda n-k}}\\
			\nonumber \lesssim&\frac{r_Q}{(2^kr_Q)^{2(sn+l)+1}},
		\end{align*}
		as required.
	\end{proof}
	
	Now we deal with $\g^*_{E^c_{Q},\lambda}$.
	\begin{lemma}\label{P3}
		Let $\lambda\in\mathbb{N}$ with $\lambda\geq4$. Then for any $f\in L^1$,
		\begin{align*}
			\sup_{\zeta\in Q}\left|[\g^*_{E^c_{Q},\lambda}(P[f])(\zeta)]^2-[\g^*_{E^c_{Q},\lambda}(P[f])(\zeta_Q)]^2\right|
			\lesssim\sum_{j=1}^{\infty}\frac{1}{2^{j/4}}\left(\frac{1}{\sigma(2^jQ)}\int_{2^jQ}|f|d\sigma\right)^2.
		\end{align*}
	\end{lemma}

	\begin{proof}
		We begin to show this lemma for $\g^*_{E^c_Q,\lambda}=\g^*_{E^c_Q,X,\lambda}$. Let $\zeta\in Q$. It is not difficult to verify
		$$\left|[\g^*_{E^c_{Q},X,\lambda}(P[f])(\zeta)]^2-[\g^*_{E^c_{Q},X,\lambda}(P[f])(\zeta_Q)]^2\right|\lesssim F_{X,1}+F_{X,2}.$$
		Here $F_{X,1}$ and $F_{X,2}$ are defined as follows:
		\begin{eqnarray*}
			F_{X,1}& = &\sum_{k=1}^{\infty}\int_{F_k}\frac{|K_{\lambda n}(z,\zeta)-K_{\lambda n}(z,\zeta_Q)|}{(1-|z|^2)^{(1-\lambda)n-1}}|XP[f\chi_{2^{k+4}Q}](z)|^2dv(z),\\
			F_{X,2}& = &\sum_{k=1}^{\infty}\int_{F_k}\frac{|K_{\lambda n}(z,\zeta)-K_{\lambda n}(z,\zeta_Q)|}{(1-|z|^2)^{(1-\lambda)n-1}}|XP[f\chi_{\SN\setminus2^{k+4}Q}](z)|^2dv(z),
		\end{eqnarray*}
		where $F_k$ is defined in (\ref{Poisson6}) with $k\geq1$.
		Hence, it suffices to show
		\begin{align}\label{Poisson8}
			\max\{F_{X,1},F_{X,2}\}\lesssim \sum_{j=1}^{\infty}\frac{1}{2^{j/4}}\left(\frac{1}{\sigma(2^jQ)}\int_{2^jQ}|f|d\sigma\right)^2.
		\end{align}
		We first consider $F_{X,1}$. By Lemma \ref{smooth3}, for each $z\in\BN$,
		$$|XP[f\chi_{2^{k+4}Q}](z)|\lesssim\frac{1}{(1-|z|^2)^{n+1}}\int_{2^{k+4}Q }|f|d\sigma.$$
		This gives that 
		\begin{align*}
			F_{X,1}\lesssim&\sum_{k=1}^{\infty}\int_{F_k}\frac{|K_{\lambda n}(z,\zeta)-K_{\lambda n}(z,\zeta_Q)|}{(1-|z|^2)^{(1-\lambda)n-1}}\left(\frac{1}{(1-|z|^2)^{n+1}}\int_{2^{k+4}Q }|f|d\sigma\right)^2dv(z)\\
			=&\sum_{k=1}^{\infty}\int_{F_k}\frac{|K_{\lambda n}(z,\zeta)-K_{\lambda n}(z,\zeta_Q)|}{(1-|z|^2)^{(3-\lambda)n+1}}dv(z)\left(\int_{2^{k+4}Q }|f|d\sigma\right)^2.
		\end{align*}
		Now applying Lemma \ref{Poisson7} with $s=3$ and $l=1$ and noting  $v(F_k)\lesssim (2^{k}r_Q)^{2n+2}$ for each $k\geq1$, we get
		\begin{align*}	F_{X,1}\lesssim&\sum_{k=1}^\infty\frac{r_Qv(F_k)}{(2^kr_Q)^{6n+3}}
			\left(\int_{2^{k+4}Q }|f|d\sigma\right)^2\\
			\lesssim&\sum_{k=1}^\infty\frac{1}{2^{kn}(2^kr_Q)^{4n}}
			\left(\int_{2^{k+4}Q }|f|d\sigma\right)^2\\
			\lesssim&\sum_{j=1}^{\infty}\frac{1}{2^{j/4}}\left(\frac{1}{\sigma(2^jQ)}\int_{2^jQ}|f|d\sigma\right)^2,
		\end{align*}
		where in the last inequality we use \eqref{Poisson2}.
		
		We then turn to $F_{X,2}$. Note that for each $z\in F_k$ and $\xi\in2^{j+k+4}Q \setminus2^{j+k+3}Q $ with $j\in\mathbb{N}$ and $k\in\mathbb{N}$, 
		\begin{align}\label{Poisson11}
			d(\xi,z)\geq d(\xi,\zeta_Q)-d(\zeta_Q,z)\gtrsim2^{j+k}r_Q.
		\end{align}
		Hence, by Lemma \ref{smooth3}, for each $z\in F_k$,
		\begin{align*}
			|XP[f\chi_{\SN\setminus2^{k+4}Q }](z)|
			\lesssim&\sum_{j=1}^{\infty}\int_{2^{j+k+4}Q\setminus2^{j+k+3}Q}\frac{|f(\xi)|}{|1-\langle z,\xi\rangle|^{n+1}}d\sigma(\xi)\\
			\lesssim&\sum_{j=1}^{\infty}\frac{1}{2^{(2n+2)(j+k)}r_Q^{2n+2}}\int_{2^{j+k+4}Q}|f(\xi)|d\sigma(\xi).
		\end{align*}
		On the other hand, we use Lemma \ref{Poisson7} with $s=1$ and $l=-1$ to conclude that for each  $\zeta\in Q $ and $z\in F_k$,
		$$\frac{|K_{\lambda n}(z,\zeta)-K_{\lambda n}(z,\zeta_Q)|}{(1-|z|^2)^{(1-\lambda)n-1}}\lesssim\frac{r_Q}{(2^kr_Q)^{2n-1}}.$$
		By the above estimates and $v(F_k)\lesssim (2^{k}r_Q)^{2n+2}$ for each $k\in\mathbb{N}$, we have
		\begin{align*}
			F_{X,2}\lesssim&\sum_{k=1}^\infty\frac{r_Qv(F_k)}{(2^kr_Q)^{2n-1}2^{7k/2}r_Q^4}
			\left(\sum_{j=1}^{\infty}\frac{1}{2^{(j+k)/4}2^{2n(j+k)}r_Q^{2n}}\int_{2^{j+k+4}Q}|f|d\sigma\right)^2\\
			\lesssim&\sum_{k=1}^\infty\frac{1}{2^{k/2}}
			\left(\sum_{j=1}^{\infty}\frac{1}{2^{(j+k)/4}\sigma(2^{j+k+4}Q)}\int_{2^{j+k+4}Q}|f|d\sigma\right)^2\\
			\lesssim&\sum_{j=1}^{\infty}\frac{1}{2^{j/4}}\left(\frac{1}{\sigma(2^jQ)}\int_{2^jQ}|f|d\sigma\right)^2,
		\end{align*}
		where in the second inequality we use \eqref{Poisson2}, and the last inequality follows from H\"older's inequality. This proves \eqref{Poisson8}.

		We come to the proof of $\g^*_{E^c_Q,\lambda}=\g^*_{E^c_Q,\mu,\lambda}$. Let $\zeta\in Q$. Following the same argument as in the previous case, one also has
		$$\left|[\g^*_{E^c_Q,\mu,\lambda}(P[f])(\zeta)]^2-[\g^*_{E^c_Q,\mu,\lambda}(P[f])(\zeta_Q)]^2\right|\lesssim F_{\mu,1}+F_{\mu,2},$$
		where $F_{\mu,1}$ and $F_{\mu,2}$ are given by 
		\begin{eqnarray*}
			F_{\mu,1}& = &\sum_{k=1}^{\infty}\int_{F_k}\frac{|K_{\lambda n}(z,\zeta)-K_{\lambda n}(z,\zeta_Q)|}{(1-|z|^2)^{(1-\lambda)n}}|P[f\chi_{2^{k+4}Q}](z)|^2d\mu(z),\\
			F_{\mu,2}& = &\sum_{k=1}^{\infty}\int_{F_k}\frac{|K_{\lambda n}(z,\zeta)-K_{\lambda n}(z,\zeta_Q)|}{(1-|z|^2)^{(1-\lambda)n}}|P[f\chi_{\SN\setminus2^{k+4}Q}](z)|^2d\mu(z).
		\end{eqnarray*}
		It suffices to show
		\begin{align}\label{Poisson18}
			\max\{F_{\mu,1},F_{\mu,2}\}\lesssim\sum_{j=1}^{\infty}\frac{1}{2^{j/4}}\left(\frac{1}{\sigma(2^jQ)}\int_{2^jQ}|f|d\sigma\right)^2.
		\end{align}
		Note that for each $z\in\BN$,
		$$|P[f\chi_{2^{k+4}Q }](z)|\lesssim\frac{1}{(1-|z|^2)^n}\int_{2^{k+4}Q }|f|d\sigma.$$
		Therefore, by \eqref{Poisson2}, we find
		\begin{align*}
			F_{\mu,1}\lesssim&\sum_{k=1}^{\infty}\int_{F_k}\frac{|K_{\lambda n}(z,\zeta)-K_{\lambda n}(z,\zeta_Q)|}{(1-|z|^2)^{(3-\lambda)n}}d\mu(z)\left(\int_{2^{k+4}Q }|f|d\sigma\right)^2\\
			\lesssim&\sum_{k=1}^\infty\frac{r_Q\mu(F_k)}{(2^kr_Q)^{6n+1}}
			\left(\int_{2^{k+4}Q}|f|d\sigma\right)^2\\
			\lesssim&\sum_{k=1}^\infty\frac{1}{2^k}\left(\frac{1}{\sigma(2^{k}Q)}\int_{2^{k}Q}|f|d\sigma\right)^2\\
			\lesssim&\sum_{j=1}^{\infty}\frac{1}{2^{j/4}}\left(\frac{1}{\sigma(2^jQ)}\int_{2^jQ}|f|d\sigma\right)^2,
		\end{align*}
		where the second inequality comes from Lemma \ref{Poisson7} by taking $s=3$ and $l=0$, and the third inequality is from the definition of Carleson measure, that is $\mu(F_k)\lesssim (2^kr_Q)^{2n}$ for each $k\geq1$.

		It remains to estimate $F_{\mu,2}$. Notice that
		$$F_{\mu,2}=\sum_{k=1}^{\infty}\int_{F_k}\frac{|K_{\lambda n}(z,\zeta)-K_{\lambda n}(z,\zeta_Q)|}{(1-|z|^2)^{(1-\lambda)n-1/4}}\left(\frac{P[f\chi_{\SN\setminus2^{k+4}Q}](z)}{(1-|z|^2)^{1/8}}\right)^2d\mu(z).$$ 
		It follows from (\ref{Poisson11}) that for each $z\in F_k$,
		\begin{align*}
			&(1-|z|^2)^{-1/8}|P[f\chi_{\SN\setminus2^{k+4}Q}](z)|\\
			\lesssim&\sum_{j=1}^{\infty}\int_{2^{j+k+4}Q\setminus2^{j+k+3}Q}\frac{1}{|1-\langle z,\xi\rangle|^{n+1/8}}|f\chi_{\SN\setminus2^{k+3}Q}(\xi)|d\sigma(\xi)\\
			\lesssim&\sum_{j=1}^{\infty}\frac{1}{(2^{2(j+k)}r_Q ^{2})^{n+1/8}}\int_{2^{j+k+4}Q}|f|d\sigma.
		\end{align*}
		On the other hand, applying Lemma \ref{Poisson7} with $s=1$ and $l=-1/4$ one gets for each  $\zeta\in Q $ and $z\in F_k$,
		$$\frac{|K_{\lambda n}(z,\zeta)-K_{\lambda n}(z,\zeta_Q)|}{(1-|z|^2)^{(1-\lambda)n-1/4}}\lesssim\frac{r_Q}{(2^kr_Q)^{2n+1/2}}.$$
		Therefore, by \eqref{Poisson2},
		\begin{align*}
			F_{\mu,2}\lesssim&\sum_{k=1}^\infty\frac{r_Q\mu(F_k)}{(2^kr_Q)^{2n+1/2}}\left(\sum_{j=1}^{\infty}
			\frac{1}{2^{(j+k)/4}r_Q^{1/4}\sigma(2^{j+k+4}Q)}\int_{2^{j+k+4}Q}|f|d\sigma\right)^2\\
			=&\sum_{k=1}^\infty\frac{r_Q\mu(F_k)}{(2^kr_Q)^{2n+1/2}r_Q^{1/2}}\left(\sum_{j=1}^{\infty}
			\frac{1}{2^{(j+k)/4}\sigma(2^{j+k+4}Q)}\int_{2^{j+k+4}Q}|f|d\sigma\right)^2.
		\end{align*}
		Thanks to $\mu(F_k)\lesssim (2^kr_Q)^{2n}$ for each $k\geq1$ and using H\"older's inequality, we finally deduce that 
		$$F_{\mu,2}\lesssim\left(\sum_{j=1}^{\infty}
		\frac{1}{2^{j/4}\sigma(2^{j}Q)}\int_{2^{j}Q}|f|d\sigma\right)^2\\
		\lesssim\sum_{j=1}^{\infty}\frac{1}{2^{j/4}}\left(\frac{1}{\sigma(2^jQ)}\int_{2^jQ}|f|d\sigma\right)^2.$$
		This completes the proof.
	\end{proof}

	\bigskip

	\section{Proofs of Lemma \ref{Rd1} and Lemma \ref{Rd3}}\label{prooflemmas}
	In this section, we prove Lemma \ref{Rd1} and Lemma \ref{Rd3}. Our proofs rely on the so-called local mean oscillation technique. Now we begin the proofs, and implement the local mean oscillation formula.
	
	Recall $\g^*_\lambda\in\{\g^*_{X,\lambda},\g^*_{\mu,\lambda}\}$ and $Q^0=\SN$. Recall that $m_{(\g^*_{\lambda}(P[f]))^2}(Q^0)$ is a median of $(\g^*_{\lambda}(P[f]))^2$ on $Q^0$. For any $\zeta\in Q^0$,
	\begin{align*}
		|\g^*_\lambda(P[f])(\zeta)|^2
		\leq&|m_{(\g^*_{\lambda}(P[f]))^2}(Q^0)|+|(\g^*_\lambda(P[f])(\zeta))^2-m_{(\g^*_{\lambda}(P[f]))^2}(Q^0)|.
	\end{align*}
	By the local mean oscillation formula (i.e. Theorem \ref{LMO}), there exists a sparse family $S(Q^0)\subset\mathscr{D}(Q^0)$ such that for a.e. $\zeta\in Q^0$,
	$$|(\g^*_\lambda(P[f])(\zeta))^2-m_{(\g^*_{\lambda}(P[f]))^2}(Q^0)|\leq2\sum_{Q\in S(Q^0)}\omega_\epsilon((\g^*_\lambda(P[f]))^2;Q)\chi_Q(\zeta),$$
	for some $\epsilon\in(0,1)$. Therefore, in order to show Lemma \ref{Rd1} and Lemma \ref{Rd3}, it suffices to show
	\begin{align}\label{Rd2}
		\|\chi_{Q^0}|m_{(\g^*_{\lambda}(P[f]))^2}(Q^0)|^{1/2}\|_{L^p_\omega}\lesssim[\omega]_{A_p}^{\max\{1/2,1/(p-1)\}}\|f\|_{L^p_\omega}
	\end{align}
	and
	\begin{align}\label{Rd9}
		\left\|\left[\sum_{Q\in S(Q^0)}\omega_\epsilon((\g^*_\lambda(P[f]))^2;Q)\chi_Q(\zeta)\right]^{1/2}\right\|_{L^p_\omega}\lesssim[\omega]_{A_p}^{\max\{1/2,1/(p-1)\}}\|f\|_{L^p_\omega}.
	\end{align}
	In the remaining of this section, we are devoted to the proofs of \eqref{Rd2} and \eqref{Rd9}. We proceed with the proof of \eqref{Rd2}.
	\subsection{Proof of \eqref{Rd2}}
	Note
	\begin{align*}
		|m_{(\g^*_{\lambda}(P[f]))^2}(Q^0)|
		\leq&((\chi_{Q^0}\g^*_{\lambda}(P[f]))^2)^*(\sigma(Q^0)/3)\\
		\leq&[(\chi_{Q^0}\g^*_{\lambda}(P[f]))^*(\sigma(Q^0)/6)]^2.
	\end{align*}
	Thus, from Theorem \ref{Weak7} and $Q^0=\SN$, one has
	$$(\chi_{Q^0}\g^*_{\lambda}(P[f]))^*(\sigma(Q^0)/6)\leq\frac{6}{\sigma(Q^0)}
	\|\g^*_{\lambda}(P[f])\|_{L^{1,\infty}}\lesssim\frac{1}{\sigma(\SN)}\int_{\SN}|f|d\sigma,$$
	and hence
	$$\chi_{Q^0}(\zeta)|m_{(\g^*_{\lambda}(P[f]))^2}(Q^0)|^{1/2}\lesssim Mf(\zeta),$$
	which yields \eqref{Rd2} by Theorem \ref{Apw}. This shows \eqref{Rd2}.
	
	\subsection{Some key estimates} In this subsection, we aim to show \eqref{Rd9}, which finishes the proofs of Lemma \ref{Rd1} and Lemma \ref{Rd3}. We need to use the following two auxiliary lemmas so as to prove \eqref{Rd9}. Before presenting these two helpful lemmas, we introduce some definitions.
	
	Let $m\in\mathbb{N}$. Given a sparse family $S\subset \mathscr{D}$, for any $f\in L^1$, define
	$$T_{2,m}^{S}f(\zeta)=\left(\sum_{Q\in S}\left(\frac{1}{\sigma(2^mQ)}\int_{2^mQ}|f|d\sigma\right)^2
	\chi_{Q}(\zeta)\right)^{1/2}.$$
	Given $\omega\in A_3$, let $\mathbb{X}=L^{3/2}_\omega$. The dual space of $\mathbb{X}$ can be identified with $\mathbb{X}'=L^{3}_{\omega^{-2}}$ under the pair $\int_{\SN}fgd\sigma$ for $f\in\mathbb{X}$ and $g\in\mathbb{X}'$. Denote by $\mathbb{X}^{(2)}$ the space consisting of all measurable functions $f$ such that $|f|^2\in \mathbb{X}$, which is endowed with the norm $\|f\|_{\mathbb{X}^{(2)}}=\||f|^2\|_\mathbb{X}^{1/2}$. Indeed, $\mathbb{X}^{(2)}=L^3_\omega$.
	
	Recall that $\mathscr{D}_i$ for $i=1,\cdots,\mathscr{K}$ is a collection of dyadic systems given in Remark \ref{de1} (ii). The following lemma has been shown in \cite{Ler3} on Euclidean spaces, and we extend it to $\SN$ following the same arguments as in \cite{Ler3}. 
	\begin{lemma}\label{Ler31}
		Under the same notation as above, assume that a measurable function $f\in L^1$ satisfies
		\begin{align}\label{as1}
			\max_{1\leq i\leq \mathscr{K}}\sup_{S: \text{sparse},S\subset\mathscr{D}_i}\|T_{2,0}^S(f)\|_{\mathbb{X}^{(2)}}<\infty.
		\end{align}
		Then one has
		$$\sup_{S:\text{sparse},S\subset\mathscr{D}}\|T_{2,m}^{S}f\|_{\mathbb{X}^{(2)}}\leq C m^{1/2}\max_{1\leq i\leq \mathscr{K}}\sup_{S: \text{sparse},S\subset\mathscr{D}_i}\|T_{2,0}^S(f)\|_{\mathbb{X}^{(2)}}.$$
		Here the constant $C>0$ does not depend on $f$, $\omega$ and $m$.
	\end{lemma}
	
	\begin{proof}
		We use the notation introduced in Theorem \ref{de} and Remark \ref{de1}. Let $S=\{Q^k_j\}\subset \mathscr{D}$ be a sparse family. Applying Remark \ref{de1} (ii), we can decompose the cubes $\{2^mQ^k_j\}$ into $\mathscr{K}$ disjoint families $F_i$ such that for each $2^mQ^k_j\in F_i$ there exists a cube $P_{j,k}^{m,i}\in \mathscr{D}_i$ such that $2^mQ_j^k\subset P_{j,k}^{m,i}$ and $\text{diam}P_{j,k}^{m,i}\leq (\rho/2)\text{diam}2^mQ^k_j$. By \eqref{eq5}, 
		\begin{align}\label{eq3}
			\sigma(P_{j,k}^{m,i})\asymp\sigma(2^mQ_j^k).
		\end{align}
		For each $i=1,\cdots,\mathscr{K}$, let
		$$\mathcal{J}^i=\{(j,k):2^mQ_j^k\in F_i\}.$$
		Let $\mathcal{I}^i$ be any finite subset of $\mathcal{J}^i$. Let
		$$T_{2,m,\mathcal{I}^i}^Sf(\zeta)=\left(\sum_{(j,k)\in\mathcal{I}^i}\left(\frac{1}{\sigma(2^mQ_j^k)}\int_{2^mQ_j^k}|f|d\sigma\right)^2
		\chi_{Q_j^k}(\zeta)\right)^{1/2}.$$
		Then, $T_{2,m,\mathcal{I}^i}^S(f)\in\mathbb{X}^{(2)}$ by \eqref{as1}. By duality, there exists $g\in \mathbb{X}'$ with $\|g\|_{\mathbb{X}'}=1$ such that
		\begin{align}\label{eq8}
			\|T_{2,m,\mathcal{I}^i}^S(f)\|_{\mathbb{X}^{(2)}}^2=\|[T_{2,m,\mathcal{I}^i}^S(f)]^2\|_{\mathbb{X}}
			=&\int_{\SN}[T_{2,m,\mathcal{I}^i}^S(f)]^2|g|d\sigma
			\lesssim\int_{\SN}\mathscr{M}_{\mathcal{I}^i,m}(f,g)|f|d\sigma,
		\end{align}
		where
		$$\mathscr{M}_{\mathcal{I}^i,m}(f,g)(\zeta)
		=\sum_{(j,k)\in\mathcal{I}^i}\left(\frac{1}{\sigma(P_{j,k}^{m,i})}\int_{P_{j,k}^{m,i}}|f|d\sigma\right)
		\left(\frac{1}{\sigma(P_{j,k}^{m,i})}\int_{Q_j^k}|g|d\sigma\right)\chi_{P_{j,k}^{m,i}}(\zeta)$$
		and in the last inequality we use \eqref{eq3}. Recall that $m_{\mathscr{M}_{\mathcal{I}^i,m}(f,g)}(\SN)$ is a median of $\mathscr{M}_{\mathcal{I}^i,m}(f,g)$ on $\SN$. Note that for a.e. $\zeta\in \SN$,
		\begin{align}\label{Tm1}
			\mathscr{M}_{\mathcal{I}^i,m}(f,g)(\zeta)
			\leq\left|m_{\mathscr{M}_{\mathcal{I}^i,m}(f,g)}(\SN)\right|+\left|\mathscr{M}_{\mathcal{I}^i,m}(f,g)(\zeta)-m_{\mathscr{M}_{\mathcal{I}^i,m}(f,g)}(\SN)\right|.
		\end{align}
		Consider $m_{\mathscr{M}_{\mathcal{I}^i,m}(f,g)}(\SN)$ firstly. For $\zeta\in\SN$, we have
		$$\mathscr{M}_{\mathcal{I}^i,m}(f,g)(\zeta)\leq Mf(\zeta)\mathscr{F}_m(g)(\zeta),$$
		where
		$$\mathscr{F}_m(g)(\zeta)=\sum_{(j,k)\in \mathcal{I}^i}\frac{1}{\sigma(P_{j,k}^{m,i})}\int_{Q_j^k}|g|d\sigma \chi_{P_{j,k}^{m,i}}(\zeta).$$
		We hence have
		\begin{align*}
			\left|m_{\mathscr{M}_{\mathcal{I}^i,m}(f,g)}(\SN)\right|
			\leq&\left(\chi_{\SN}\mathscr{M}_{\mathcal{I}^i,m}(f,g)\right)^*(\sigma(\SN)/3)\\
			\leq&\left(\chi_{\SN}Mf\mathscr{F}_m(g)\right)^*(\sigma(\SN)/3)\\
			\leq&\left(\chi_{\SN}Mf\right)^*(\sigma(\SN)/6)\left(\chi_{\SN}\mathscr{F}_m(g)\right)^*(\sigma(\SN)/6)\\
			\lesssim&\frac{1}{\sigma(\SN)}\|Mf\|_{L^{1,\infty}}\frac{1}{\sigma(\SN)}\|\mathscr{F}_m(g)\|_{L^{1,\infty}}.
		\end{align*}
		Using H\"older's inequality,
		\begin{align*}\|g\|_{L^1}=&\int_{\SN}|g|\omega^{-\frac{2}{3}}\omega^{\frac{2}{3}}d\sigma\\
			\leq&\left(\int_{\SN}|g|^3\omega^{-2}d\sigma\right)^{1/3}\left(\int_{\SN}\omega d\sigma\right)^{2/3}\\
			=&\|g\|_{\mathbb{X}'}(\omega(\SN))^{2/3}<\infty,
		\end{align*}
		which implies $g\in L^1$. By \cite[Lemma 6.5]{AtVa}, $\|\mathscr{F}_m(g)\|_{L^{1,\infty}}\lesssim m\|g\|_{L^1}$. 
		Combine with the weak type $(1,1)$ of $M$, and we obtain
		$$\left|m_{\mathscr{M}_{\mathcal{I}^i,m}(f,g)}(\SN)\right|\lesssim m\frac{1}{\sigma(\SN)}\int_{\SN}|f|d\sigma\frac{1}{\sigma(\SN)}\int_{\SN}|g|d\sigma.$$
		Note that $S_{i,1}:=\{\SN\}$ is also a sparse family, and we rewrite
		\begin{align}\label{Tm5}
			\chi_{\SN}(\zeta)\left|m_{\mathscr{M}_{\mathcal{I}^i,m}(f,g)}(\SN)\right|\lesssim m\sum_{Q'\in S_{i,1}}\frac{1}{\sigma(Q')}\int_{Q'}|f|d\sigma\frac{1}{\sigma(Q')}
			\int_{Q'}|g|d\sigma\chi_{Q'}(\zeta).
		\end{align}

		Next we consider the second term of \eqref{Tm1}. Applying Theorem \ref{LMO}, there exists a (possibly empty) sparse family $S_{i,2}\subset \mathscr{D}_i$ such that for a.e. $\zeta\in \SN$,
		\begin{align}\label{eq9}
			\left|\mathscr{M}_{\mathcal{I}^i,m}(f,g)(\zeta)-m_{\mathscr{M}_{\mathcal{I}^i,m}(f,g)}(\SN)\right|
			\leq2\sum_{Q'\in S_{i,2}}\omega_\epsilon\left(\mathscr{M}_{\mathcal{I}^i,m}(f,g);Q'\right)\chi_{Q'}(\zeta)
		\end{align}
		for some $\epsilon\in(0,1)$.
		
		For $\zeta\in Q'\in S_{i,2}$, we have
		\begin{align}\label{Tm2}
			\mathscr{M}_{\mathcal{I}^i,m}(f,g)(\zeta)
			=\sum_{(j,k)\in \mathcal{I}^i:P_{j,k}^{m,i}\subset Q'}A_{j,k}^{m,i}(f,g)(\zeta)+\sum_{(j,k)\in \mathcal{I}^i:Q'\subset P_{j,k}^{m,i}}A_{j,k}^{m,i}(f,g)(\zeta),
		\end{align}
		where
		$$A_{j,k}^{m,i}(f,g)(\zeta)=\left(\frac{1}{\sigma(P_{j,k}^{m,i})}\int_{P_{j,k}^{m,i}}|f|d\sigma\right)
		\left(\frac{1}{\sigma(P_{j,k}^{m,i})}\int_{Q_j^k}|g|d\sigma\right)\chi_{P_{j,k}^{m,i}}(\zeta).$$
		For the first term of \eqref{Tm2}, we have for $\zeta\in Q'\in S_{i,2}$ that
		\begin{align}\label{Tm4}
			\sum_{(j,k)\in \mathcal{I}^i:P_{j,k}^{m,i}\subset Q'}A_{j,k}^{m,i}(f,g)(\zeta)\leq  M(f\chi_{Q'})(\zeta)\mathscr{F}_m(g\chi_{Q'})(\zeta),
		\end{align}
		where $$\mathscr{F}_m(g\chi_{Q'})(\zeta)=\sum_{(j,k)\in \mathcal{I}^i}\frac{1}{\sigma(P_{j,k}^{m,i})}\int_{Q_j^k}|g|\chi_{Q'}d\sigma\chi_{P_{j,k}^{m,i}}(\zeta).$$
		For the second term of \eqref{Tm2}, we have for $\zeta\in Q'\in S_{i,2}$ that
		$$\sum_{(j,k)\in \mathcal{I}^i:Q'\subset P_{j,k}^{m,i}}A_{j,k}^{m,i}(f,g)(\zeta)=\sum_{(j,k)\in \mathcal{I}^i:Q'\subset P_{j,k}^{m,i}}\frac{1}{\sigma(P_{j,k}^{m,i})}\int_{P_{j,k}^{m,i}}|f|d\sigma
		\frac{1}{\sigma(P_{j,k}^{m,i})}\int_{Q_j^k}|g|d\sigma=:c.$$
		By \eqref{Tm2} and \eqref{Tm4},
		\begin{align}\label{Tm3}
			\left|\mathscr{M}_{\mathcal{I}^i,m}(f,g)(\zeta)-c\right|
			\leq M(f\chi_{Q'})(\zeta)\mathscr{F}_m(g\chi_{Q'})(\zeta).
		\end{align}
		By \cite[Lemma 6.5]{AtVa} again, $\|\mathscr{F}_m(g\chi_{Q'})\|_{L^{1,\infty}}\lesssim m\|g\chi_{Q'}\|_{L^1}$. Then by the weak type $(1,1)$ of $M$, we have from \eqref{Tm3} that
		\begin{align*}
			&\left(\chi_{Q'}\left(\mathscr{M}_{\mathcal{I}^i,m}(f,g)(\zeta)-c\right)\right)^*(\epsilon\sigma(Q'))\\
			\leq&(M(f\chi_{Q'}))^*(\epsilon\sigma(Q')/2)(\mathscr{F}_m(g\chi_{Q'}))^*(\epsilon\sigma(Q')/2)\\
			\lesssim&\frac{1}{\sigma(Q')}\|M(f\chi_{Q'})\|_{L^{1,\infty}}\frac{1}{\sigma(Q')}\|\mathscr{F}_m(g\chi_{Q'})\|_{L^{1,\infty}}\\
			\lesssim&m\frac{1}{\sigma(Q')}\int_{Q'}|f|d\sigma\frac{1}{\sigma(Q')}\int_{Q'}|g|d\sigma,
		\end{align*}
		which implies
		\begin{align*}
			\omega_\epsilon\left(\mathscr{M}_{\mathcal{I}^i,m}(f,g);Q'\right)
			\lesssim m\frac{1}{\sigma(Q')}\int_{Q'}|f|d\sigma\frac{1}{\sigma(Q')}\int_{Q'}|g|d\sigma.
		\end{align*}
		Combining with \eqref{eq9}, we have
		\begin{align}\label{Tm6}
			&\left|\mathscr{M}_{\mathcal{I}^i,m}(f,g)(\zeta)-m_{\mathscr{M}_{\mathcal{I}^i,m}(f,g)}(\SN)\right| \lesssim m\sum_{Q'\in S_{i,2}}\left(\frac{1}{\sigma(Q')}\int_{Q'}|f|d\sigma\right)\left(\frac{1}{\sigma(Q')}
			\int_{Q'}|g|d\sigma\right)\chi_{Q'}(\zeta).
		\end{align}
		Finally, plugging \eqref{Tm5} and \eqref{Tm6} into \eqref{Tm1}, we have
		$$\mathscr{M}_{\mathcal{I}^i,m}(f,g)(\zeta)\lesssim m\sum_{\ell=1}^{2}\sum_{Q'\in S_{i,\ell}}\left(\frac{1}{\sigma(Q')}\int_{Q'}|f|d\sigma\right)\left(\frac{1}{\sigma(Q')}
		\int_{Q'}|g|d\sigma\right)\chi_{Q'}(\zeta).$$
		Therefore, one has
		\begin{align*}
			\int_{\SN}\mathscr{M}_{\mathcal{I}^i,m}(f,g)|f|d\sigma
			\lesssim m\sum_{\ell=1}^{2}\int_{\SN}[T_{2,0}^{S_{i,\ell}}(f)]^2|g|d\sigma
			\lesssim m\max_{1\leq i\leq \mathscr{K}}\sup_{S: \text{sparse},S\subset\mathscr{D}_i}\|T_{2,0}^S(f)\|_{\mathbb{X}^{(2)}}^2.
		\end{align*}
		Combining with \eqref{eq8}, we get
		$$\|T_{2,m,\mathcal{I}^i}^{S}f\|_{\mathbb{X}^{(2)}}\lesssim m^{1/2}\max_{1\leq i\leq \mathscr{K}}\sup_{S: \text{sparse},S\subset\mathscr{D}_i}\|T_{2,0}^S(f)\|_{\mathbb{X}^{(2)}}.$$
		Applying Fatou's lemma,  
		\begin{align*}
			\|T_{2,m,i}^S\|_{\mathbb{X}^{(2)}}\leq\liminf_{I^i\to \mathcal{J}^i}\|T_{2,m,\mathcal{I}^i}^{S}f\|_{\mathbb{X}^{(2)}}\lesssim m^{1/2}\max_{1\leq i\leq \mathscr{K}}\sup_{S: \text{sparse},S\subset\mathscr{D}_i}\|T_{2,0}^S(f)\|_{\mathbb{X}^{(2)}},
		\end{align*}
		where 
		$$T_{2,m,i}^Sf(\zeta)=\left(\sum_{(j,k)\in\mathcal{J}^i}\left(\frac{1}{\sigma(2^mQ_j^k)}\int_{2^mQ_j^k}|f|d\sigma\right)^2
		\chi_{Q_j^k}(\zeta)\right)^{1/2}.$$
		Finally, we obtain
		\begin{align*}
			\|T_{2,m}^{S}f\|_{\mathbb{X}^{(2)}}^2=&\left\|\sum_{i=1}^{\mathscr{K}}[T_{2,m,i}^{S}f]^2\right\|_{\mathbb{X}}
			\leq\sum_{i=1}^{\mathscr{K}}\|[T_{2,m,i}^{S}f]^2\|_{\mathbb{X}}\\
			=&\sum_{i=1}^{\mathscr{K}}\|T_{2,m,i}^{S}f\|_{\mathbb{X}^{(2)}}^2\leq C m\max_{1\leq i\leq \mathscr{K}}\sup_{S: \text{sparse},S\subset\mathscr{D}_i}\|T_{2,0}^S(f)\|_{\mathbb{X}^{(2)}}^2.
		\end{align*}
		This completes the proof.
	\end{proof}
	
	The following lemma can be found in \cite{CUMP} (see also \cite{BD}).
	
	\begin{lemma}\label{CUMP1}
		Let $1<p<\infty$ and $\omega\in A_p$. Let $S\subset \mathscr{D}$ be a sparse family. Then for any $f\in L_\omega^p$,
		$$\|T_{2,0}^Sf\|_{L^p_\omega}\leq C[\omega]_{A_p}^{\max\{1/2,1/(p-1)\}}\|f\|_{L^p_\omega},$$
		where the constant $C>0$ does not depend on $f$, $\omega$ and $S$.
	\end{lemma}
	Now we are ready to show \eqref{Rd9} by virtue of Lemma \ref{Ler31} and Lemma \ref{CUMP1}. By the celebrated Rubio de Francia extrapolation theorem, we only need to show \eqref{Rd9} for $p=3$.
	\begin{proof}[Proof of \eqref{Rd9}]
		At first, it suffices to show
		\begin{align}\label{Rd10}
			\sum_{Q\in S(Q^0)}\omega_\epsilon((\g^*_{\lambda}(P[f]))^2;Q)\chi_Q(\zeta)
			\lesssim\sum_{j=1}^{\infty}\frac{1}{2^{j/4}}\left(T_{2,j}^{S(Q^0)}(f)(\zeta)\right)^2.
		\end{align}
		Indeed, assume $\omega\in A_3$. By \eqref{Rd10},
		$$\left\|\left[\sum_{Q\in S(Q^0)}\omega_\epsilon((\g^*_{\lambda}(P[f]))^2;Q)\chi_Q\right]^{1/2}\right\|_{L^3_\omega}
		\lesssim\left\|\left[\sum_{j=1}^{\infty}\frac{1}{2^{j/4}}
		\left(T_{2,j}^{S(Q^0)}(f)\right)^2\right]^{1/2}\right\|_{L^3_\omega}.$$
		Note that
		\begin{align}\label{Rd11}
			\left\|\left[\sum_{j=1}^{\infty}\frac{1}{2^{j/4}}
			\left(T_{2,j}^{S(Q^0)}(f)\right)^2\right]^{1/2}\right\|_{L^3_\omega}
			=&\left\|\sum_{j=1}^{\infty}\frac{1}{2^{j/4}}
			\left(T_{2,j}^{S(Q^0)}(f)\right)^2\right\|_{L^{3/2}_\omega}^{1/2}.
		\end{align}
		Recall $\mathbb{X}=L^{3/2}_\omega$. Then by Minkowski's inequality,
		\begin{align*}
			\mbox{RHS\ of\ \eqref{Rd11}}\leq&\left(\sum_{j=1}^{\infty}\frac{1}{2^{j/4}}\left\|
			\left(T_{2,j}^{S(Q^0)}(f)\right)^2\right\|_{L^{3/2}_\omega}\right)^{1/2}\\
			=&\left(\sum_{j=1}^{\infty}\frac{1}{2^{j/4}}\|T_{2,j}^{S(Q^0)}(f)\|_{\mathbb{X}^{(2)}}^2\right)^{1/2}.
		\end{align*}
		This implies by Lemma \ref{Ler31} that
		\begin{align*}
			\mbox{RHS\ of\ \eqref{Rd11}}\lesssim\max_{1\leq i\leq \mathscr{K}}\sup_{S: \text{sparse},S\in\mathscr{D}_i}\|T_{2,0}^S(f)\|_{\mathbb{X}^{(2)}}
			=\max_{1\leq i\leq \mathscr{K}}\sup_{S: \text{sparse},S\in\mathscr{D}_i}\|T_{2,0}^S(f)\|_{L^3_\omega}.
		\end{align*}
		Applying Lemma \ref{CUMP1}, we have
		$$\left\|\left[\sum_{Q\in S(Q^0)}\omega_\epsilon((\g^*_{\lambda}(P[f]))^2;Q)\chi_Q(\zeta)\right]^{1/2}\right\|_{L^3_\omega}
		\lesssim[\omega]_{A_3}^{1/2}\|f\|_{L^3_\omega}.$$
		Then the Rubio de Francia extrapolation theorem (see e.g. \cite{CUMP1,DGPP,Du}) for $p=3$ yields \eqref{Rd9}.
		
		It remains to show \eqref{Rd10}. However, in order to show \eqref{Rd10}, it suffices to show that for any fixed $Q\in S(Q^0)$ and $\epsilon\in(0,1)$, 
		\begin{align}\label{Rd4}
			IV:=[\chi_{Q}((\g^*_{\lambda}(P[f]))^2-c)]^*(\epsilon\sigma(Q))\lesssim\sum_{j=1}^{\infty}\frac{1}{2^{j/4}}
			\left(\frac{1}{\sigma(2^{j}Q)}\int_{2^{j}Q}|f|d\sigma\right)^2,
		\end{align}
		where $c=[\g^*_{E_{Q}^c,\lambda}(P[f])(\zeta_Q)]^2$. In fact, by \eqref{Rd4} and using the definition of $\omega_\epsilon((\g^*_{\lambda}(P[f]))^2;Q)$, we obtain 
		\begin{align*}
			\sum_{Q\in S(Q^0)}\omega_\epsilon((\g^*_{\lambda}(P[f]))^2;Q)\chi_Q(\zeta)
			\lesssim&\sum_{Q\in S(Q^0)}\sum_{j=1}^{\infty}\frac{1}{2^{j/4}}
			\left(\frac{1}{\sigma(2^{j}Q)}\int_{2^{j}Q}|f|d\sigma\right)^2\chi_Q(\zeta)\\
			=&\sum_{j=1}^{\infty}\frac{1}{2^{j/4}}\left(T_{2,j}^{S(Q^0)}(f)(\zeta)\right)^2,
		\end{align*}
		which yields \eqref{Rd10}.
		
		Now we prove \eqref{Rd4}, and this finishes the proof of \eqref{Rd9}. Note that
		$$\left|(\g^*_{\lambda}(P[f])(\zeta))^2-c\right|
		\leq(\g^*_{E_Q,\lambda}(P[f])(\zeta))^2+\left|(\g^*_{E_Q^c,\lambda}(P[f])(\zeta))^2-c\right|.$$
		By the Cauchy-Schwarz inequality, we get 
		$$(\g^*_{E_Q,\lambda}(P[f])(\zeta))^2\leq 2(\g^*_{E_Q,\lambda}(P[f\chi_{8Q}])(\zeta))^2+2(\g^*_{E_Q,\lambda}(P[f\chi_{\SN\setminus8Q}])(\zeta))^2.$$
		Therefore, one has
		\begin{align*}
			IV\leq IV_1^2+IV_2+IV_3,
		\end{align*}
		where $I_1$, $I_2$ and $I_3$ are given by
		\begin{eqnarray*}
			IV_1& = &[\sqrt{2}\chi_{Q}\g^*_{E_Q,\lambda}(P[f\chi_{8Q}])]^*(\epsilon\sigma(Q)/8),\\
			IV_2& = &[2\chi_{Q}(\g^*_{E_Q,\lambda}(P[f\chi_{\SN\setminus8Q}])^2]^*(\epsilon\sigma(Q)/4),\\
			IV_3& =
			&[\chi_{Q}((\g^*_{E_Q^c,\lambda}(P[f\chi_{\SN\setminus8Q}]))^2-c)]^*(\epsilon\sigma(Q)/2).
		\end{eqnarray*}
		
		By Theorem \ref{Weak7}, we have
		$$IV_1\lesssim\frac{1}{\epsilon\sigma(Q)}\|\g^*_{\lambda}(P[f\chi_{8Q}])\|_{L^{1,\infty}}
		\lesssim\frac{1}{\sigma(8Q)}\int_{8Q}|f|d\sigma.$$
		Using Lemma \ref{P2}, one has
		\begin{align*}
			IV_2\lesssim&\frac{1}{\epsilon\sigma(Q)}\|(\g^*_{E_Q,\lambda}(P[f\chi_{\SN\setminus8Q}]))^2\|_{L^{1,\infty}(Q)}\\
			\leq&\frac{1}{\epsilon\sigma(Q)}\|(\g^*_{E_Q,\lambda}(P[f\chi_{\SN\setminus8Q}]))^2\|_{L^1}\\
			\lesssim&\sum_{j=1}^{\infty}\frac{1}{2^{2nj}}\left(\frac{1}{\sigma(2^j Q)}\int_{2^j Q}|f|d\sigma\right)^2.
		\end{align*}
		By Lemma \ref{P3},
		\begin{align*}
			IV_3\leq&\sup_{\zeta\in Q}\left|[\g^*_{E_{Q}^c,\lambda}(P[f])(\zeta)]^2-[\g^*_{E_{Q}^c,\lambda}(P[f])(\zeta_Q)]^2\right|\\
			\lesssim&\sum_{j=1}^{\infty}\frac{1}{2^{j/4}}\left(\frac{1}{\sigma(2^j Q)}\int_{2^j Q}|f|d\sigma\right)^2.
		\end{align*}
		This completes the proof of \eqref{Rd4}.
	\end{proof}

	\bigskip
	
	\section{Proofs of Theorem \ref{Main1}, Corollary \ref{Mcor1}, Theorem \ref{Main2} and Theorem \ref{Main3}}\label{section6}
	
	In this section, we will present the proofs of our main results, i.e. Theorem \ref{Main1}, Corollary \ref{Mcor1}, Theorem \ref{Main2} and Theorem \ref{Main3}. Now we show Theorem \ref{Main1} and Corollary \ref{Mcor1}.
	
	\begin{proof}[Proof of Theorem \ref{Main1}]
		\textbf{(i)} The desired inequality follows directly from \eqref{Area1} and Lemma \ref{Rd1}. It remains to show the optimality of the exponent. We only need to  consider the case $n=1$ and $1<p\leq 3$. We just deal with $X=R$ as $\overline{R}$ can be done with in the same way. Let $0<\delta<1$. Consider $ \omega(\xi)=|1-\xi|^{(p-1)(1-\delta)}$ and $f(\xi)=|1-\xi|^{\delta-1}$ for any $ \xi\in \mathbb{S}_1$. Then $[\omega]_{A_p}\asymp \delta^{1-p}$ and $\|f\|_{L^p_\omega}^p\asymp \delta^{-1}$. Note that by similar arguments in \cite[Lemma 4.3]{AB} or \cite[Lemma 2.6]{BBG}, one can show
		$$ S_\alpha^R(P[f])(\zeta) \gtrsim  \left(\int_{0}^{1}|\re{\left(R P[f](r\zeta)\right)}|^2(1-r)dr\right)^{1/2}, \quad \forall  \zeta\in\mathbb{S}_1. $$
		When $|1-\zeta|\leq \frac{1}{10}$, one can compute
		$$ \left(\int_{0}^{1}|\re{\left(R P[f](r\zeta)\right)}|^2(1-r)dr\right)^{1/2}\gtrsim \delta^{-1} |1-\zeta|^{\delta-1}.  $$
		This implies
		$$ \|S_\alpha^R(P[f])\|_{L^p_\omega}^p\gtrsim \delta^{-p-1}. $$
		Letting $\delta\rightarrow 0$, we conclude that the optimal exponent is $1/(p-1)$.
		\\
		\textbf{(ii)} We mainly use duality, i.e. Lemma \ref{Grlem}. Suppose $\alpha>1$. Assume that $f\in L^p_\omega$ with $P[f](0)=0$. Notice that the dual space of $L^p_\omega$ can be identified with $L^{p'}_{\omega'}$ under the pair $\int_{\SN}fhd\sigma$ for $f\in L^p_\omega$ and $h\in L^{p'}_{\omega'}$. It suffices to show that
		\begin{align}\label{A2}
			\left|\int_{\SN}fhd\sigma\right|\lesssim[\omega]_{A_p}^{\frac{1}{p-1}\max\{1/2,p-1\}}
			\left(\sum_{X\in\{R,\overline{R}\}}
			\|S_\alpha^{X}(P[f])\|_{L^p_\omega}\right).
		\end{align}
		for any $h\in L^{p'}_{\omega'}$ with $\|h\|_{L^{p'}_{\omega'}}\leq1$.

		Let $h\in L^{p'}_{\omega'}$ with $\|h\|_{L^{p'}_{\omega'}}\leq1$. Using Lemma \ref{Grlem},
		\begin{align*}
			\left|\int_{\SN}P[f](r\zeta)P[h](r\zeta)d\sigma(\zeta)\right|\lesssim
			\int_{\SN}S_\beta^{\widetilde{\bigtriangledown}}(P[f])(\zeta)
			S_\beta^{\widetilde{\bigtriangledown}}(P[\overline{h}])(\zeta)d\sigma(\zeta),
		\end{align*}
		where $1<\beta<\alpha$. Using H\"older's inequality,
		\begin{align*}
			\left|\int_{\SN}P[f](r\zeta)P[h](r\zeta)d\sigma(\zeta)\right|\lesssim\|S_\beta^{\widetilde{\bigtriangledown}}(P[f])\|_{L^p_\omega}
			\|S_\beta^{\widetilde{\bigtriangledown}}(P[\overline{h}])\|_{L^{p'}_{\omega'}}.
		\end{align*}
		We hence have from Theorem \ref{Main1} (i) and Proposition \ref{dim2} that
		\begin{align}\label{A1}
			\begin{split}
				\left|\int_{\SN}P[f](r\zeta)P[h](r\zeta)d\sigma(\zeta)\right|\lesssim&[\omega']_{A_{p'}}^{\max\{1/2,1/(p'-1)\}}
			\left(\sum_{X\in\{R,\overline{R}\}}
			\|S_\alpha^{X}(P[f])\|_{L^p_\omega}\right)\\
				=&[\omega]_{A_p}^{\frac{1}{p-1}\max\{1/2,p-1\}}\left(\sum_{X\in\{R,\overline{R}\}}\|S_\alpha^{X}(P[f])\|_{L^p_\omega}\right),
			\end{split}
		\end{align}
		where we use $[\omega']_{A_{p'}}=[\omega]_{A_p}^{{1}/{(p-1)}}$. By \eqref{NonM} and H\"older's inequality, we have 
		$$N_1(P[f])N_1(P[h])\in L^1.$$
		Note that
		$$|P[f](r\zeta)P[h](r\zeta)|\leq N_1(P[f])(\zeta)N_1(P[h])(\zeta)$$
		for any $\zeta\in\SN$. By \cite[Corollary 4.16]{Z}, $P[f]_*=f$ and $P[h]_*=h$. By the dominated convergence theorem,
		\begin{align*}
			&\lim_{r\to1^-}\left|\int_{\SN}f(\zeta)h(\zeta)-P[f](r\zeta)P[h](r\zeta)d\sigma(\zeta)\right|=0.
		\end{align*}
		Combining with \eqref{A1}, we get \eqref{A2}. We hence have
		$$\|f\|_{L^p_\omega}\lesssim[\omega]_{A_p}^{\frac{1}{p-1}\max\{1/2,p-1\}}\left(\sum_{X\in\{R,\overline{R}\}}
		\|S_\alpha^{X}(P[f])\|_{L^p_\omega}\right).$$
		This completes the proof.
	\end{proof}
	
	We come to the proof of Corollary \ref{Mcor1}.
	\begin{proof}[Proof of Corollary \ref{Mcor1}]
		It follows from Theorem \ref{Main1}, Proposition \ref{dim2} and \eqref{gra}.
	\end{proof}
	
	Next, we show Theorem \ref{Main2}.

	\begin{proof}[Proof of Theorem \ref{Main2}]
		\textbf{(i)} Let $\alpha>1/2$. Let $f\in H^p_\omega$. By Theorem \ref{Hp}, there is a K-limit function $f_*$ such that $f=P[f_*]$ and $\|f_*\|_{L^p_\omega}\leq\|f\|_{H^p_\omega}$. Applying Theorem \ref{Main1} to $P[f_*]$, we have
		$$\|S_{\alpha}^R(P[f_*])\|_{L^p_\omega}\lesssim[\omega]_{A_p}^{\max\{1/2,1/(p-1)\}}\|f_*\|_{L^p_\omega},$$
		which implies
		$$\|S_{\alpha}^R(f)\|_{L^p_\omega}\lesssim[\omega]_{A_p}^{\max\{1/2,1/(p-1)\}}\|f\|_{H^p_\omega}.$$
		This gives the first assertion.\\
		\noindent \textbf{(ii)} Note that every constant function belongs to $H^p_\omega$, and then it suffices to prove $\|f\|_{H^p_\omega}\lesssim \|S_{\alpha}^{R}(f)\|_{L^p_\omega}$ under the assumption $f(0)=0$, otherwise we can replace $f$ with $f-f(0)$. 
		
		Since $L^p_\omega\subset L^1$, we have $S_{\alpha}^{R}(f)\in L^1$. By Theorem B and Theorem \ref{Hp2}, we obtain that $f\in H^1$, $f_*$ exists and $f=P[f_*]$. In addition, $f_*\in L^1$. It suffices to show $f_*\in L^p_\omega$. Indeed, if $f_*\in L^p_\omega$, since $r\zeta\in D_1(\zeta)$, one has by \eqref{NonM}
		\begin{align}\label{Hp3}
			\sup_{0<r<1}\int_{\SN}|f(r\zeta)|^p\omega(\zeta)d\sigma(\zeta)
			\leq\|N_1(P[f_*])\|_{L^p_\omega}^p\lesssim[\omega]_{A_p}^{p/(p-1)}\|f_*\|_{L^p_\omega}^p<\infty,
		\end{align}
		which implies $f\in H^p_\omega$. We also use the similar arguments as in the proof of Theorem \ref{Main1} (ii) to show $f_*\in L^p_\omega$. Let us sketch the proof. Let $1<\beta<\alpha$. Let $h$ be a continuous function on $\SN$. Using Lemma \ref{Grlem},
		\begin{align*}
			\left|\int_{\SN}f(r\zeta)P[h](r\zeta)d\sigma(\zeta)\right|\lesssim
			\int_{\SN}S_\beta^{\widetilde{\bigtriangledown}}(f)(\zeta)
			S_\beta^{\widetilde{\bigtriangledown}}(P[\overline{h}])(\zeta)d\sigma(\zeta).
		\end{align*}
		By Theorem \ref{Main1} and Proposition \ref{dim2},
		\begin{align}\label{A4}
			\left|\int_{\SN}f(r\zeta)P[h](r\zeta)d\sigma\right|
			\lesssim[\omega]_{A_p}^{\frac{1}{p-1}\max\{1/2,p-1\}}\|S_{\alpha}^{R}(f)\|_{L^p_\omega}
			\|h\|_{L^{p'}_{\omega'}}.
		\end{align}
		
		Let $H=h$ on $\SN$ and $H=P[h]$ in $\BN$. By \cite[Theorem 5.5]{Sto}, $\sup_{z\in\overline{\BN}}|H(z)|<\infty$. By \cite[Corollary 4.16]{Z}, $P[h]_*=h$. Since $f\in H^1$, by \cite[Theorem 4.24]{Z}, $N_1f\in L^1$. Note that $f(r\zeta)\leq N_1f(\zeta)$ for any $\zeta\in\SN$. We apply the dominated convergence theorem to deduce
		\begin{align*}
			&\lim_{r\to1^-}\left|\int_{\SN}f_*(\zeta)h(\zeta)-f(r\zeta)P[h](r\zeta)d\sigma(\zeta)\right|=0.
		\end{align*}
		Combining with \eqref{A4}, we have
		$$\left|\int_{\SN}f_*hd\sigma\right|\lesssim[\omega]_{A_p}^{\frac{1}{p-1}\max\{1/2,p-1\}}\|S_{\alpha}^{R}(f)\|_{L^p_\omega}
		\|h\|_{L^{p'}_{\omega'}}.$$
		Recall that the set of continuous functions is dense in $L^{p'}_{\omega'}$ (\cite[P. 69]{R2}). This implies $f_*\in L^p_\omega$ and $$\|f_*\|_{L^p_\omega}\lesssim[\omega]_{A_p}^{\frac{1}{p-1}\max\{1/2,p-1\}}\|S_{\alpha}^{R}(f)\|_{L^p_\omega}.$$
		By \eqref{Hp3}, we have
		$$\|f\|_{H^p_\omega}
		\lesssim[\omega]_{A_p}^{\frac{1}{(p-1)^2}\max\{1/2,p-1\}}\|S_{\alpha}^{R}(f)\|_{L^p_\omega}.$$
		This completes the proof.
	\end{proof}
	
	Finally we end this section with the proof of Theorem \ref{Main3} by virtues of Lemma \ref{Rd3} and Theorem \ref{Main2}.
	\begin{proof}[Proof of Theorem \ref{Main3}]
		It follows from Theorem \ref{Hp} that there exists $f_*\in L_\omega^p$ such that $f=P[f_*]$ and $\|f_*\|_{L^p_\omega}\lesssim \|f\|_{H^p_\omega}$. Let $d\mu(z)=(1-|z|^2)|Rg(z)|^2dv(z)$. Note that $\mu$ is a Carleson measure since $g\in BMOA$. Then by \eqref{Area2} and Lemma \ref{Rd3},
		\begin{align*}
			\|S_\alpha^R(J_gf)\|_{L^p_\omega}= \|S^{\mu}_{\alpha}(P[f_*])\|_{L^p_\omega}&\lesssim \|\g^*_{\mu,\lambda}(P[f_*])\|_{L^p_\omega}\\
			&\lesssim[\omega]_{A_p}^{\max\{1/2,1/(p-1)\}}\|f_*\|_{L^p_\omega}\\
			&\lesssim [\omega]_{A_p}^{\max\{1/2,1/(p-1)\}}\|f\|_{H^p_\omega}.
		\end{align*}
		By Theorem \ref{Main2}, $$\|J_gf\|_{H_\omega^p}\lesssim[\omega]_{A_p}^{\frac{1}{(p-1)^2}\max\{1/2,p-1\}}\|S_\alpha^R(J_gf)\|_{L^p_\omega}.$$ 
		We hence have
		$$\|J_gf\|_{H_\omega^p}\lesssim[\omega]_{A_p}^{\varphi(p)}\|f\|_{H_\omega^p},$$
		where
		$$\varphi(p)={\frac{1}{(p-1)^2}\max\{1/2,p-1\}}+\max\{1/2,1/(p-1)\}.$$
		This finishes the proof.
	\end{proof}
	\bigskip


\end{document}